\documentclass[oneside]{amsart}
\usepackage{inputenc}
\usepackage{cite}
\usepackage{amsthm}
\usepackage{stackengine}[2013-09-11]
\usepackage{amsmath}
\usepackage{amsfonts}
\usepackage{amssymb}
\usepackage{amsthm}
\usepackage{colonequals}
\usepackage{tikz}\usetikzlibrary{positioning,decorations.markings, patterns, calc}
\usepackage{pgfplots}\pgfplotsset{compat=1.15}
\usepackage{caption, subcaption}
\usepackage[usenames]{xcolor}
\usepackage[hyperfootnotes=false]{hyperref}
\usepackage{cite}

\title[]{Sectorial normalization of one resonant biholomorphisms near a fixed point}
\author[]{Maxime Chatal$^{\dag}$}
\address{Department of Mathematics, Kyushu University, 744 Motooka, Nishi-ku, Fukuoka 819-0395, Japan}
\email{chatal.maxime.207@m.kyushu-u.ac.jp}
\thanks{This work was supported by JSPS KAKENHI Grant Number JP25KF0117.}

\author{Laurent Stolovitch$^{\dag\dag}$}
\address{CNRS and Laboratoire J.-A. Dieudonn\'e
	U.M.R. 7351, Universit\'e C\^ote d'Azur, Parc Valrose
	06108 Nice Cedex 02, France}
\email{stolo@unice.fr}
\thanks{  }


\newcommand{\dist}{\operatorname{dist}}

\newtheorem{thm}{Theorem}[section]
\newtheorem{cor}[thm]{Corollary}
\newtheorem{prop}[thm]{Proposition}
\newtheorem{lemma}[thm]{Lemma}
\newtheorem{remark}[thm]{Remark}

\newtheorem{defis}[thm]{Definition}

\newcommand{\diag}{\operatorname{diag}}

\theoremstyle{definition}

\newtheorem{rem}[thm]{Remark}

\renewcommand{\th}[1]{\begin{thm}\label{#1}}
\newcommand{\co}[1]{\begin{cor}\label{#1}}
	\newcommand{\eco}{\end{cor}}
\renewcommand{\le}[1]{\begin{lemma}\label{#1}}
	\newcommand{\ele}{\end{lemma}}
\newcommand{\pr}[1]{\begin{prop}\label{#1}}
	\newcommand{\epr}{\end{prop}}

\newcommand{\ga}{\begin{gather}}
	\newcommand{\ega}{\end{gather}}
\newcommand{\gan}{\begin{gather*}}
	\newcommand{\egan}{\end{gather*}}
\newcommand{\al}{\begin{align}}
	\newcommand{\eal}{\end{align}}
\newcommand{\aln}{\begin{align*}}
	\newcommand{\ealn}{\end{align*}}
\newcommand{\eq}[1]{\begin{equation}\label{#1}}
	\newcommand{\eeq}{\end{equation}}

\newcommand{\B}{\mathcal B}
\newcommand{\V}{\mathcal V}

\newcommand{\D}{\mathbb{D}}
\newcommand{\C}{{\mathbb C}}
\newcommand{\N}{{\mathbb N}}
\newcommand{\R}{{\mathbb R}}
\newcommand{\Z}{{\mathbb Z}}

\newcommand{\OO}{\mathcal{O}} 
\newcommand{\GG}{\mathcal{G}}

\newcommand{\ov}{\overline}
\newcommand{\ord}{\operatorname{ord}}

\newcommand{\RE}{\operatorname{Re}}

\newcommand{\s}{\mathcal S}

\newcommand{\I}{\operatorname{I}}

\newcommand{\re}[1]{(\ref{#1})}

\newcommand{\rl}[1]{Lemma~\ref{#1}}

\newcommand{\rt}[1]{Theorem~\ref{#1}}
\newcommand{\rd}[1]{Definition~\ref{#1}}
\newcommand{\rrem}[1]{Remark~\ref{#1}}

\newcounter{pp}
\newcommand{\bpp}{\begin{list}{$\hspace{-1em}(\alph{pp})$}{\usecounter{pp}}}
	\newcommand{\epp}{\end{list}}

\newcounter{ppp}
\newcommand{\bppp}{\begin{list}{$\hspace{-1em}(\roman{ppp})$}{\usecounter{ppp}}}
		\newcommand{\eppp}{\end{list}}
	
	\def\beq{\begin{equation}}
		\def\eeq{\end{equation}}

\renewcommand{\Re}{\operatorname{Re}}
\renewcommand{\Im}{\operatorname{Im}}	
	\usepackage{time,datetime}
	
\begin{document}
		
		\begin{abstract}
			In this article, we study germs of biholomorphisms of $(\C^n,0)$, $n\geq 2$ fixing the origin, the linear part at which has all its eigenvalues on the unit circle. We prove the existence of a {\it sectorial} transformation to a polynomial normal form of a non-degenerate 1-resonant biholomorphism under a Diophantine condition. We also give their holomorphic classification. We show that the {\it sectorial} automorphisms of the polynomial normal form are {\it exponentially flat} perturbations of the identity. We deduce the Gevrey character of the formal normalizing transformation and we show how the Gevrey index depends on the Diophantine exponent. 
		\end{abstract}
		
		\date{\today}
		\maketitle
		

	\setcounter{thm}{0}\setcounter{equation}{0}
		\section{Introduction}
		
		\newcommand{\doi}[1]{\href{http://dx.doi.org/#1}{#1}}
		\newcommand{\arxiv}[1]{\href{https://arxiv.org/pdf/#1}{arXiv:#1}}
		This article deals with the problem of normal forms of so-called $1$-resonant germs of biholomorphisms at a fixed point, say the origin, in $\mathbb{C}^n$, $n\geq 2$, as well as their holomorphic classification. 
		More precisely, we consider germs of biholomorphisms of $(\C^n,0)$, fixing the origin, and of the form
		$$
		z=(z_1,\ldots,z_n)\mapsto F(z)=(F_1(z),\ldots, F_n(z))=(\lambda_1z_1+f_1(z),\ldots, \lambda_nz_n+f_n(z))
		$$
		where the $f_i$'s are germs at the origin and of order $\geq 2$ there. The study of its iterates near a fixed point is a long standing problem which has been formalized by Poincaré. Its solution depends dramatically on the properties of the eigenvalues $\lambda=(\lambda_1,\ldots, \lambda_n)$ of its linear part at the fixed point. 
		These are nonzero complex numbers and $\lambda$ is said to be resonant if $\lambda^Q:=\lambda_1^{q_1}\cdots \lambda_n^{q_n}=\lambda_i$ for some $Q=(q_1,\ldots,q_n)\in \N^n$, $|Q|:=q_1+\cdots +q_n\geq 2$ and some $i\in\{1,\ldots, n\}$. If $\lambda$ is non resonant then $F$ is formally linearizable and holomorphically linearizable in a neighborhood of the origin if $\lambda$ is Diophantine (see \rd{dio}) (see e.g.\cite{Arnold,Bruno,russmann-lin}). When the linear part is a contraction, i.e. $|\lambda_i|<1$ for all $i$ (or an expansion $|\lambda_i|>1$), there are only a finite number of resonances. In that case,  $F$ is holomorphically conjugate to a polynomial in a full neighborhood of the origin (e.g. \cite{ueda-contraction} ). New difficulties appear when there are an infinite number of resonances. The easiest case is when the resonances are generated by a single {\it resonant monomial} and $\lambda$ is then said to be $1$-resonant: there exists $\beta\in \N^n\setminus \{0\}$ such that, if $\lambda^Q:=\lambda_1^{q_1}\cdots \lambda_n^{q_n}=\lambda_i$ for some $Q=(q_1,\ldots,q_n)\in \N^n$, $|Q|:=q_1+\cdots +q_n\geq 2$ and some $i\in\{1,\ldots, n\}$, then there exists $l\in \N$ such that $Q=l\beta+E_i$ where $E_i=(0,\ldots,0,1,0,\ldots 0)$, $1$ at $i$th spot.
		The Poincaré-Dulac normal form theorem \cite{Arnold} shows that there exists a formal change of variables $z=\hat\Phi(Z)=Z+\hat\phi(Z)$, tangent to identity at the origin, such that in these new coordinates:
		$$
		\hat F:=\hat\Phi^{-1}\circ F\circ \hat\Phi(Z)=(\lambda_iZ_i\hat G_i(Z^\beta))_{1\leq i\leq n}.
		$$
		Here $\hat G_i$ is a formal power series in one variable with constant term equal to $1$. However, the transformation is, in general, merely formal and does not converge in any neighborhood of the origin. If the generator of resonances, $z^{\beta}$, is not a first integral of  the normal form, that is $\hat F^{\beta}\neq z^{\beta}$,  we can simplify further the normal form and obtain polynomials for the resonant coordinates, that is for $i=1\ldots, m$, as shown in \cite{Bracci} (see also \cite{jenkins-parbolic} in the case $m=1$). 
		
		Despite some similarity with germs of vector fields at a fixed point, very few results are known about holomorphic normalization of germs of biholomorphisms near a fixed point as well as their holomorphic classification. Our main goal is to prove the existence of "sectorial holomorphic transformation" to such a normal form of {\it non-degenerate} $1$-resonant germs \emph{when all the eigenvalues of the linear part at the origin are on the unit circle, in any dimension}. By this we mean the existence of a conjugacy to a normal form by a biholomorphism of a domain in $\C^n$ which has $0$ on its boundary. We shall further describe germs which are pairwise holomorphically conjugate in a full neighborhood of the origin and formally conjugate to the same normal form. 
		To our knowledge, in dimension greater than $1$, the only known results concern the so-called semi-hyperbolic case in dimension $2$ (i.e. one eigenvalue is 1, and the other is $>1$) which are due to S. Voronin and his students \cite{voronin-semihyperbolic, semihyperbolic}. Similar sectorial phenomena are found in \cite{KMRR25,stolo-klimes}.
		
		A special class of these objects, called {\it quasi-parabolic}, has attracted a lot of attention from the dynamical view point in the sense that most works focus on the existence of attracting/repelling "petals" or parabolic manifold: A germ $F(z)=Az+O(|z|^2)$ of a biholomorphism of $\C^n$ fixing $0$ is said to be quasi-parabolic if $A$ is a diagonal matrix with eigenvalues $1$ and $\lambda_j$, $j=2,\ldots, n$, with $|\lambda_j|=1$ and $\lambda_j\neq1$. A parabolic manifold for $F$ is given by an injective holomorphic map $g:\Delta\rightarrow\C^n$ continuous up to the boundary such that $\Delta$ is a simply connected domain of $\C^d$ for some $d\leq n$ with $0\in\partial \Delta$, $g(0)=0$, $g(\Delta)$ is $F$-invariant and $F^{\circ p}(g(\zeta))\rightarrow 0$ as $p$ tends to $\infty$ for all $\zeta\in\Delta$. When one eigenvalue is equal to $1$ and all the others are of modulus strictly less than $1$, the map is called semi-attractive. In these situations, there exist quasi-parabolic curves \cite{hakim94,rong10,rong16,bracci-molino,Bracci,bracci-raissy-zaitsev, bracci-raissy-stensones}. For a survey we refer to \cite{abate-parabolic}.
		
		The situation of conjugacy to normal form is completely understood in dimension $1$ (see e.g. \cite{ilyashenko-yakovenko-book, camacho-sad-book}). In this case, such a biholomorphism $F(z)=\lambda z+ \text{h.o.t.}$ is a {\it parabolic} germ, i.e. $\lambda^p=1$, and has a polynomial normal form $NF=\lambda z(1+cz^{kp}+dz^{2kp})$ \cite{Malgrange-bourbaki, Loday} and in general, the formal transformation $\hat \Phi$ that conjugates $F$ to its normal form $NF$ is a divergent power series. Nevertheless, it can be shown that the diffeomorphism can be transformed into its normal form by means of a "sectorial transformation", that is a holomorphic diffeomorphism of an open sector at the fixed point which is asymptotic to the formal power series transformation in that sector \cite{kimura, birkhoff}. One can prove that the normalizing formal power series transformation is, in fact, {\it summable} in the sense of Ramis \cite{Loday,ramis-ksum, malgrange-somm, ramis-stolo-cours} and that the sectorial conjugacy is a "holomorphic sectorial realization" of that series in the sector. There are several overlapping sectors $\{S_i\}_i$ that form a covering of a punctured neighborhood of $0$, together with several holomorphic conjugacies $\{\Phi_i\}_i$ to the same normal form $NF$, defined in each of these sectors. In general, these diffeomorphisms do not coincide in the intersection of two consecutive sectors.  Hence, they define a natural {\it isotropy cocycle} $\{\Phi_i\circ\Phi_{i+1}^{-1}\}$ defined on $\{S_i\cap S_{i+1}\}$. It is asymptotic to Identity at the origin and leaves invariant the normal form. One of the main results of Ecalle-Voronin \cite{EcalleIII, voronin} is that such a cocycle determines an equivalence class of such germs of diffeomorphisms up to a conjugacy by biholomorphism of a neighborhood of the origin, tangent to the identity there. See also the recent work by L. Teyssier \cite{teyssier-real}.
\begin{figure}
    \centering
    \includegraphics[width=.3\textwidth]{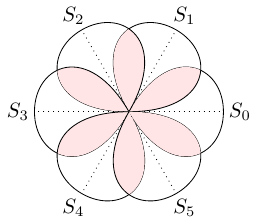}
    \caption{The Leau--Fatou petals $S_j$, $j\in\Z_{2kp}$.}
    \label{figure:parabolic}
\end{figure}
		 These sectors are {\it Leau-Fatou petals} \cite{fatou,Leau}. They are invariant domains under the dynamics, consecutively, attracting or repelling. 
		 
		 This theory of parabolic $1$-dimensional germs was developed almost in parallel to a similar theory of germs of holomorphic vector fields in $(\mathbb{C}^2,0)$ with $0$ as a fixed point, having a saddle-node or resonant saddle point there \cite{Ram-Mart1,Ram-Mart2}. These latter works have been encapsulated under the name {\it $1$-resonant} and generalized in any dimension by the second author \cite{Stolo-classif}. The present work can be seen somehow as a surprising diffeomorphism-version of the latter.
		 
		One of the main features of our method is that we extend the dynamical systems to one dimension more, the new variable playing the role of the resonant monomial. Ultimately, we shall restrict this variable to the monomial itself in order to obtain the result on our original dynamical system. 
\section{Main results}
Let $n \in \N^*$ and $\lambda:=(\lambda_1, ..., \lambda_n)\in(\C^*)^n$.
\begin{defis}\label{1-res} We say that $\lambda$ is \emph{1-resonant} if there exists a non zero $\beta\in \N^n$ such that, if  $\lambda_1^{q_1}\cdots\lambda_n^{q_n}=\lambda_j$ (i.e. a resonant relation) for some $Q:=(q_1,\ldots,q_n)\in \N^n$, $|Q|:=q_1+\cdots+q_n\geq 2$ and $1\leq j\leq n$, then there exists $\ell\in \N^*$ such that $Q=\ell \beta+E_j$, where $E_j$ denotes the $j$-th vector of the canonical basis of $\N^n$. We also say that $\lambda$ is \emph{1-resonant with respect to $\beta$}.
\end{defis} 
\begin{rem}
	We have $\lambda^{\beta}=1$.
\end{rem}
\begin{defis}\label{dio} 
	We say that $\lambda=(\lambda_1,\ldots,\lambda_n)$ is \emph{Diophantine with respect to the monomial ideal $\I:=(x^{\beta})$} generated by the monomial $x^{\beta}$ if there exist $C>0$ and $\tau\geq 0$ such that 
	for all $Q\in \N^n_2$ (that is, $Q\in\N^n$ with $|Q|\geq 2$) such that $x^Q\not\in\I$ and for all $j=1, ...,n$,
	$$
	|\lambda^Q-\lambda_j|\geq \frac{C}{|Q|^{\tau}}.
	$$
\end{defis}

Let $\tilde F(z)=Dz+f(z)$ be a germ of a holomorphic diffeomorphism of $(\C^{n},0)$ with $D=\diag(\lambda_1, ..., \lambda_n)$ as linear part at the origin.
\begin{equation}\label{system1}
	\tilde F:\left\{
	\begin{array}{l}
		z_1' =  \lambda_1 z_1 + f_1(z) \\
		\vdots  \\
		z_n' = \lambda_n z_n + f_n(z)
	\end{array}
	\right.
\end{equation}
where the $f_i(z) = \OO_{z\rightarrow 0}(|z|^2)$. We assume that $\lambda$ is 1-resonant with respect to $\beta$.
We order the coordinates so that $\beta=(\beta_1,\cdots,\beta_m,0,\cdots,0)$ and $\beta_i\neq 0$ for all $1\leq i\leq m$. According to the Poincaré-Dulac theorem (see \cite{Arnold}), there exists $k\geq 1$, such that, by eliminating all non-resonant terms up to order $k \vert \beta \vert+1$, there exists a polynomial change of coordinates, tangent to the identity, such that system \eqref{system1} is conjugate to 
\begin{equation}\label{system2}
	\left\{
	\begin{array}{l}
		z_1' =  z_1(\lambda_1 + a_1(z^{\beta})^k) + R_1(z) \\
		\vdots  \\
		z_m' =  z_m(\lambda_m + a_m(z^{\beta})^k) + R_m(z) \\
		z_{m+1}' = z_{m+1}(\lambda_{m+1} +P_{m+1}(z^{\beta})) + R_{m+1}(z)\\
		\vdots  \\
		z_{n}' = z_n(\lambda_n +P_n(z^{\beta})) + R_n(z)
	\end{array}
	\right.
\end{equation}
where the $P_i$ are polynomials of degree $\leq k$, $(a_1,\ldots,a_m) \neq 0$ and $R_i(z) = \OO_{z\rightarrow 0}(|z|^{k\vert \beta\vert+2})$.

Following \cite{Bracci}, we say that $\tilde F$ is {\it non-degenerate} if
\begin{equation}\label{defA}
A:=\sum_{j=1}^m\frac{a_j\beta_j}{\lambda_j}\neq 0.
\end{equation}

In what follows, $\ord_0$ denotes the order at $0$.
		 Our main result is  
\begin{thm}[Sectorial normalization]\label{main-sect}
Let  $\lambda$ be $1$-resonant with respect to $\beta=(\beta_1,\ldots, \beta_m,0,\dots, 0)\in \N^n$ with $\beta_1,\cdots,\beta_m\neq 0$, with $\vert \lambda_i\vert = 1$, $i=1,...,n$, and Diophantine with respect to $(x^{\beta})$. Let $\tilde F$ be a germ of a biholomorphism of $(\C^{n},0)$ as in \re{system1}.  If $\tilde F$ is non-degenerate, then $\tilde F$ is formally conjugate to a polynomial normal form
	\begin{equation}\label{nf}
		NF:\left\{
		\begin{array}{l}
			z_1' =  z_1(\lambda_1 + a_1(z^{\beta})^k+b_1(z^{\beta})^{2k})  \\
			\vdots  \\
			z_m' =  z_m(\lambda_m + a_m(z^{\beta})^k+b_m(z^{\beta})^{2k})  \\
			z_{m+1}' =  z_{m+1}(\lambda_{m+1} + Q_{m+1}(z^{\beta})) \\
			\vdots  \\
			z_n' = z_n(\lambda_n + Q_n(z^{\beta}))
		\end{array}
		\right.
	\end{equation}
	where $Q_i$ are polynomials of degree $\leq 2k$, vanishing at $0$, with $k\geq 1$. We assume furthermore
	\begin{equation}\label{H_0}
		\tag{$H_0$} \ord_0 Q_i \geq k, \quad \forall i = m+1,...,n.
	\end{equation}
    \begin{equation}
    \label{H_1}
		\tag{$H_1$} \nu_i  := \RE(a_{i,1}(\lambda_i A)^{-1}) >0, \quad \forall i = 1,...,n
	\end{equation}
    where $a_{i,1}:= a_i$ for $i=1,...,m$ and $a_{i,1}$ is the coefficient of $(z^{\beta})^k$ in $Q_i$ for $i=m+1,...,n$.
	Without loss of generality, we can assume that $A = -\frac{1}{k}$ (see \rrem{dilation}). Then, in holomorphic coordinates in which $\tilde F$ is of the form \eqref{system3} (see Section 3), there are $2k$ bounded sectors at $0\in \mathbb{C}$ of the form 
    \beq\label{sector}
	S_i=\left\{u\in \C, \left|\arg u-\frac{i\pi}{k}\right|<\frac{\alpha}{k}-\epsilon,\,|u|<r\right\},\, i=0,\ldots 2k-1
	\eeq
	with $\frac{\pi}{2}+k\epsilon < \alpha < \frac{3\pi}{4}$, for some positive $\epsilon, r$,
a small enough $\rho>0$ and for each $i$, a biholomorphism onto its image $\Psi_i:\{z\in\C^n,\,z^{\beta}\in S_i\}\cap \D^n_\rho\rightarrow \C^n$ conjugating $\tilde F$ to its normal form,
	that is $\Psi_i\circ \tilde F= NF\circ\Psi_i$ on $\{z\in\C^n,\,z^{\beta}\in S_i\}\cap \D^n_\rho$.
\end{thm}
\begin{remark}\label{cov-sect}
If $k\geq 2$, we remark that each sector has opening $\frac{2\alpha}{k}-2\epsilon > \frac{\pi}{k}$, hence $S_i \cap S_{i+1}\neq \emptyset$ for all $i$ (mod $2k$), and we have $\frac{2\alpha}{k} - 2\epsilon < \frac{3\pi}{2k} < \frac{2\pi}{k}$ which implies that $S_i \cap S_{i+2} = \emptyset$ (since $\alpha < \frac{3\pi}{4}$). If $k=1$, there are only two sectors $S_0$ and $S_1$ with again $2\alpha-2\epsilon > \pi$ and then $S_0 \cap S_1 \neq \emptyset$ has two distinct connected components. The collection of the sectors makes a covering of $\D_r\setminus\{0\}$.
\end{remark}
\begin{remark}
    A condition similar to \eqref{H_0} on the order of the $Q_i$ also occurs in \cite[section 3.3, Equation (3.2)]{Stolo-classif} for the sectorial normalization of vector fields. 
\end{remark}
\begin{remark}
    The proof will be done in the most difficult case $\ord_0 Q_i=k$. See Remark \ref{kplusgrand} for more details.
\end{remark}
Let $F_1$ and $F_2$ be two germs of biholomorphisms satisfying the conditions of \rt{main-sect}, having the same normal form $NF$ \re{nf}. Let $\{\Psi_{j,i}\}_{i=0,\ldots, 2k-1}$, $j=1,2$, be their collections of sectorial normalizations given by \rt{main-sect}. Let us define 
$$
{\mathcal D}_i:=\{z\in\C^n,\,z^{\beta}\in S_i\}\cap \D^n_\rho.
$$
Let us consider their associated cocycles $C_j:=\{C_{j,i}\}:=\{\Psi_{j,i}\circ \Psi_{j,i+1}^{-1}\}_{i=0,\ldots ,2k-1}$ defined on ${\mathcal D}_i\cap {\mathcal D}_{i+1}$. These are "automorphisms" of the normal form in the sense that $C_{j,i}\circ NF=NF\circ C_{j,i}$ on ${\mathcal D}_i\cap {\mathcal D}_{i+1}$ for all $i$.

\begin{thm}[Holomorphic classification]\label{classif}
Let $F_1$ and $F_2$ be biholomorphisms as above, satisfying the assumptions of \rt{main-sect} and formally conjugate. Let $E := \{z \in \C^n, z^\beta =0\}$. Then $F_1$ and $F_2$ are holomorphically conjugate in a neighborhood $(\mathbb{C}^{n},0)$ of the origin by a biholomorphism tangent to the identity if and only if there exists, for each $i$, a biholomorphism $T_i$ of $\mathcal{D}_i$ on its image, which commutes with the normal form and which admits along $E$ an asymptotic expansion $\hat T$ tangent to the identity and independent of $i$ such that
\[ C_{2,i} = T_i \circ C_{1,i} \circ T_{i+1}^{-1}\]
on $\mathcal{D}_i \cap \mathcal{D}_{i+1}$, $i=0,...,2k-1$.
In particular, if $F_1$ and $F_2$ are conjugate by a biholomorphism $\Theta$ tangent to the identity such that $\Theta(z)^\beta = z^\beta + \OO ((z^\beta)^{k+2})$ in a neighborhood of $0$ ($\OO((z^\beta)^{k+2})$ denotes then a holomorphic function divisible by $(z^{\beta})^{k+2}$), then $C_1=C_2$.
\end{thm}
\begin{thm}\label{thm-exp}
	The sectorial centralizer of a normal form NF \re{nf} is exponentially decreasing of order $\frac{k}{\tau+1}$ in the resonant monomial $z^{\beta}$ close to Identity: if $\Phi_{i,i+1}\circ NF= NF\circ \Phi_{i,i+1}$ on ${\mathcal D}_i\cap {\mathcal D}_{i+1}$ and $\Phi_{i,i+1}$ is tangent to identity up to order $2k+1$ in $z^\beta$ along $E$, then 
	$$
	|\Phi_{i,i+1}(z)-z|\leq C e^{-\frac{M}{\vert z^\beta\vert^{\frac{k}{\tau+1}}}}.
	$$
\end{thm}
\begin{cor}\label{cor-gevrey}
	The formal normalizing transformation is $\frac{\tau+1}{k}$-Gevrey with respect to $z^{\beta}$.
	\end{cor}
This last statement is known in the case of 1-resonant vector fields \cite[Theorem 4.3]{stolo-braaksma}.

   	\subsection{Notations}
		Given $\alpha \in (0,\pi)$ and $R\in \R_{>0}$, we define the sectors:
		\[\Delta_+(\alpha,R) = \{z \in \C, - \alpha < \arg(z-R) < \alpha\} = \{ R + t e^{i\theta}, t > 0,-\alpha<  \theta < \alpha\},\]
		\[\Delta_-(\alpha,R) = \{z \in \C, \pi -\alpha < \arg(z+R) < \pi + \alpha\} = \{-R-t e^{i\theta}, t >0,-\alpha<  \theta < \alpha \}.\]

		For all $\rho >0$, we denote by $\D_\rho = \{z \in \C, \vert z \vert < \rho\}$ the open disk of radius $\rho$.

	For $Q=(q_1,\ldots, q_n)\in  \N^n$ and $x=(x_1,\ldots,x_n)\in \C^n$, we shall write
		$$
		|Q|:=q_1+\cdots +q_n,\quad x^Q:=x_1^{q_1}\cdots x_n^{q_n}.
		$$ 
		We shall denote $\N^n_2$, the set of $Q\in\N^n$ such that $|Q|\geq 2$.
		Denote for $z\in \C^n$
		\[ \Vert z \Vert := \sup_{i=1,...,n} \vert z_i\vert.\]

		\section{Preparation}

	According to \cite[Theorem 3.6]{Bracci}, if $\tilde F$ is non-degenerate, there exist complex numbers $b_1,\ldots,b_m$ such that this system is formally conjugate to 
	
	\begin{equation}\label{system4}
		\hat F:\left\{
		\begin{array}{l}
			z_1' =  z_1(\lambda_1 + a_1(z^{\beta})^k+b_1(z^{\beta})^{2k})  \\
			\vdots  \\
			z_m' =  z_m(\lambda_m + a_m(z^{\beta})^k+b_m(z^{\beta})^{2k})  \\
			z_{m+1}' =  z_{m+1}(\lambda_{m+1} + \hat P_{m+1}(z^{\beta})) \\
			\vdots  \\
			z_n' = z_n(\lambda_n + \hat P_n(z^{\beta}))
		\end{array}
		\right.
	\end{equation}
	where $\hat P_i$ are formal series.

		\begin{lemma}\label{changement}
			Under the assumptions of \rt{main-sect}, there exists a formal change of variables $\Psi=(\psi_1, ..., \psi_n) : \C^{n} \rightarrow \C^{n}$ defined by
			\[\psi_i:z_i \mapsto z_i,\; i=1,\ldots m, \quad \psi_j:z_j \mapsto z_j(1+\varphi_j(z^{\beta})),\;j=m+1,\ldots n\]
			where for each $j=m+1,...,n$, $\varphi_j$ is a formal power series with $\varphi_j(0)=0$, such that $\hat F$ is formally conjugate to
			\begin{equation}\label{system5}
				\hat G:\left\{
				\begin{array}{l}
					z_1' =  z_1(\lambda_1 + a_1(z^{\beta})^k+b_1(z^{\beta})^{2k})  \\
					\vdots  \\
					z_m' =  z_m(\lambda_m + a_m(z^{\beta})^k+b_m(z^{\beta})^{2k})  \\
					z_{m+1}' =  z_{m+1}(\lambda_{m+1} + Q_{m+1}(z^{\beta})) \\
					\vdots  \\
					z_n' = z_n(\lambda_n + Q_n(z^{\beta}))
				\end{array}
				\right.
			\end{equation}
			where the $Q_j$ are polynomials of degree $\leq 2k$ (which are truncations of formal series $\hat P_j$).
		\end{lemma}

		\begin{proof}
			We study the conjugacy equation on each coordinate $i=1,...,n$:
			\begin{equation}\label{eqi}
				(\hat F \circ \Psi)_i = (\Psi \circ \hat G)_i.
			\end{equation}
			We recall that $z^\beta$ depends only on $z_1,\ldots,z_m$.
			Since for $i=1,\ldots m$, $\psi_i = id$,  the $i$-equation is trivial. Let $m+1 \leq i \leq n$. For all $z\in \C^{n}$
			\[(\hat F \circ \Psi)_i(z) = \hat F_i (\Psi(z)) = \psi_i(z)(\lambda_i + \hat P_i (z^{\beta})) = z_i (1+ \varphi_i(z^{\beta}))(\lambda_i + \hat P_i(z^{\beta})),\]
			\[(\Psi \circ \hat G)_i(z) = \psi_i(\hat G(z)) = \hat G_i(z)(1+ \varphi_i( \hat G(z)^{\beta})) = z_i(\lambda_i + Q_i(z^{\beta}))(1+ \varphi_i( \hat G(z)^{{\beta}}))\]
			Then \eqref{eqi} becomes
			\[ (1+ \varphi_i(z^{\beta}))(\lambda_i + \hat P_i(z^{\beta})) = (\lambda_i + Q_i(z^{\beta}))\left(1 + \varphi_i(z^{\beta}) + [\varphi_i(\hat G(z)^{\beta}) - \varphi_i(z^{\beta})]\right)\]
			We have 
			\begin{align}
			\hat G^{\beta}(z)=z^{\beta}\prod_{i=1}^m \big(1+ \frac{a_i}{\lambda_i} z^{k \beta} + \frac{b_i}{\lambda_i} z^{2k\beta } \big)^{\beta_i}&
			=z^{\beta}(1+A (z^{\beta})^k + B(z^{\beta})^{2k} + P((z^{\beta})^k)),\label{gbeta}\\
			P(u^k) = \sum_{j=3}^{2\vert \beta \vert}c_ju^{kj}. \label{ecritureP}&
			\end{align}
			where $A= \sum_{j=1}^{m} \frac{\beta_j a_j}{\lambda_j}$.
			Write $\varphi_i(u) = \sum_{l\geq 1} \varphi_{i,l}u^l$, then
			\begin{align*}
				\varphi_i(\hat G(z)^{\beta})-\varphi_i(z^{\beta}) & = \sum_{l\geq 1}\varphi_{i,l}z^{l\beta}(1+A (z^{\beta})^k + B(z^{\beta})^{2k} + P((z^{\beta})^k))^l - \sum_{l\geq 1}\varphi_{i,l}(z^{\beta})^l \\
				& = \sum_{l \geq 1}\varphi_{i,l}z^{l\beta}\left((1+A (z^{\beta})^k + B(z^{\beta})^{2k} + P((z^{\beta})^k))^l-1\right) \\
				& = \sum_{l\geq 1} \varphi_{i,l}\big(lA(z^{\beta})^k + \varepsilon_{i,l}(z^{\beta}) \big)(z^{\beta})^{l}
			\end{align*} 
			where $\varepsilon_{i,l} = \OO(\vert z^{\beta} \vert^{2k})$. Finally, \eqref{eqi} can be written
			\begin{equation}
				(1+\varphi_i(z^{\beta}))(Q_i(z^{\beta})-\hat P_i(z^{\beta})) = (\lambda_i+ Q_i(z^{\beta}))\Big(\sum_{l\geq 1} \varphi_{i,l} \big[(lA(z^{\beta})^k + \varepsilon_{i,l}(z^{\beta}) \big](z^{\beta})^l\Big).
			\end{equation}
			This can be replaced by equation in a new variable $z_0$:
			\begin{equation}\label{equaphi}
				(1+\varphi_i(z_0))(Q_i(z_0)-\hat P_i(z_0)) = (\lambda_i+ Q_i(z_0))\Big(\sum_{l\geq 1} \varphi_{i,l} \big[(lAz_0^k + \varepsilon_{i,l}(z_0) \big]z_0^l\Big).
			\end{equation}
			Since $Q_i - \hat P_i$ is of order $2k+1$ in $0$, we can choose $\varphi_{i,l}=0$ for $l=1,...,k$ so that $\eqref{equaphi}$ is satisfied. Then $(Q_i - \hat P_i)\varphi_i$ is of order $\geq 3k+2$. 
			Recursively, for $l \ge k+1$, the coefficient of $z_0^{l+k}$ in $\varphi_i(\hat G(z)^{\beta}) - \varphi_i(z^\beta)$ is  
			\[
			\lambda_i lA \varphi_{i,l} = \mathcal{F}_{i,l},
			\]  
			where $\mathcal{F}_{i,l}:=\mathcal{F}_{i,l}( \varphi_{i,k+1}, ..., \varphi_{i,l-1})$ is a known polynomial depending on the previous coefficients $\varphi_{i,j}$, $j = k+1,...,l-1$, and on the polynomials $Q_i$ and $\hat P_i$. By non-degeneracy  $A \neq 0$, hence we define  
			\[\varphi_{i,l} = \frac{\mathcal{F}_{i,l}}{\lambda_i lA},\]  
			which determines all the coefficients $\varphi_{i,l}$ for all $l \geq k+1$.
	\end{proof}
		
Assume that $\{\lambda_1, ..., \lambda_n\}$ is Diophantine with respect to $\beta$, and that  $\tilde F$ in \re{system1}  is formally conjugate to  $\hat F$ in \re{system4} (its normal form). Then according to \cite{Sto15} (see also  \cite{Pos86} for a similar but more restricted problem), we can factorize by the ideal $\I = (x^\beta)$ (see Definition \ref{dio}) in the non linear part of $\tilde F$, by a holomorphic change of variables. Using \cite[Theorem 1.1]{Sto94} (the proof is written for vector fields but can be adapted to diffeomorphisms using \cite{Sto15}) there exists a holomorphic change of coordinates, tangent to the identity which conjugates \eqref{system1} to 
\begin{equation}\label{system2b}\check F:
	\left\{
	\begin{array}{c}
		z_1' =  z_1(\lambda_1 + a_1(z^{\beta})^k+b_1(z^{\beta})^{2k}) + z_1R_1^{\leq 3k+1}(z^{\beta})+ (z^{\beta})^{3k+2}f_1(z) \\
		\vdots  \\
		z_m' =  z_m(\lambda_m + a_m(z^{\beta})^k+b_m(z^{\beta})^{2k}) +  z_mR_m^{\leq 3k+1}(z^{\beta})+(z^{\beta})^{3k+2}f_m(z) \\
		z_{m+1}' =  z_{m+1}(\lambda_{m+1} + Q_{m+1}(z^{\beta})) + z_{m+1}R_{m+1}^{\leq 3k+1}(z^{\beta})+(z^{\beta})^{3k+2}f_{m+1}(z) \\
		\vdots  \\
		z_n' = z_n(\lambda_n +Q_n(z^{\beta})) + z_nR_n^{\leq 3k+1}(z^{\beta})+ (z^{\beta})^{3k+2}f_n(z)
	\end{array}
	\right.
\end{equation}
where $R_j^{\leq 3k+1}(u)$ are polynomials of degree $ \leq 3k+1$ and of order $ \geq 2k+1$. 
The terms $z_j R_j^{\leq 3k+1}(z^\beta)$ can be removed by polynomial changes of coordinates. For $j=1,...,m$, we use the changes of coordinates of the form $z_j(1+\varphi_j(z^\beta))$ from the proof of \cite[Theorem 3.6]{Bracci} (removing up to order $3k+1$ in $z^\beta$, there is a finite number of changes of variables). For $j=m+1,...,n$, we use Lemma \ref{changement} with truncation at order $3k+1$. The following Lemma \ref{tosystem3} will ensure that the remaining higher terms are divisible by $(z^\beta)^{3k+2}$ (that is, they belong to the ideal $\I^{3k+2} = ((z^\beta)^{3k+2})$), so $\tilde F$ is holomorphically conjugate to

\begin{equation}\label{system3}F:
	\left\{
	\begin{array}{c}
		z_1' =  z_1(\lambda_1 + a_1(z^{\beta})^k+b_1(z^{\beta})^{2k}) + (z^{\beta})^{3k+2}f_1(z) \\
		\vdots  \\
		z_m' =  z_m(\lambda_m + a_m(z^{\beta})^k+b_m(z^{\beta})^{2k}) + (z^{\beta})^{3k+2}f_m(z) \\
		z_{m+1}' =  z_{m+1}(\lambda_{m+1} + Q_{m+1}(z^{\beta})) + (z^{\beta})^{3k+2}f_{m+1}(z) \\
		\vdots  \\
		z_n' = z_n(\lambda_n +Q_n(z^{\beta})) + (z^{\beta})^{3k+2}f_n(z)
	\end{array}
	\right.
\end{equation}
where $f_i(z) = \OO (1)$.

\begin{lemma}\label{tosystem3}
    Let 
    \[F(z) = \big( z_i(\lambda_i+P_i(z^\beta)) + z_i R_i(z^\beta)\big) _{i=1,...,n} \mod \I^{3k+2},\]
    with $\deg P_i \leq 2k$ and $R_i$ is a polynomial of degree $\leq 3k+1$ and of order $\geq 2k+1$. Let $\Phi(z) = \big(z_i(1+\varphi_i(z^\beta)) \big)_{i=1,...,n}$. Then
    \[ \Phi^{-1}\circ F \circ \Phi(z) = \big( z_i(\lambda_i+\tilde P_i(z^\beta)) \big)_{i=1,...,n} \mod \I^{3k+2}.\]
\end{lemma}

\begin{proof}
    Let $C$ be the set of maps of the form $\big( z_i Q_i(z^\beta) \mod \I^{3k+2} \big)_{i=1,...,n}$ (since we work $\mod \I^{3k+2}$, $Q_i$ can be chosen to be a polynomial of degree $\leq 3k+1$). Then $F, \Phi, \Phi^{-1} \in C$. This is obvious for $F$ and $\Phi$ by construction. For $\Phi^{-1}$, if $v_i = \Phi_i(z)$, then $v^\beta = z^\beta(1+ \tilde K(z^\beta))$ where $1+ \tilde K(u) = \prod_{j=1}^m (1+\varphi_j(u))^{\beta_j}$. Then $z^\beta = M(v^\beta)$ for some holomorphic function $M$ by the implicit function theorem, and then $z_i = \Phi_i^{-1}(v) = v_i \big(1+\varphi_i(M(v^\beta)) \big)^{-1} = v_i Q_i(v^\beta)$ with $Q_i(u) = \big(1+\varphi_i(M(u)) \big)^{-1}$ and then $\Phi^{-1} \in C$.

    If $H, K \in C$, then $H \circ K \in C$. Indeed, writing $H_i(z) = z_i h_i(z^\beta)$ and $K_i(z) = z_i k_i(z^\beta)$ $\mod \I^{3k+2}$, we have
    \[H_i \circ K(z) = K_i(z)h_i(K(z)^\beta) + \I^{3k+2}\circ K(z)\]
    where $\I^{3k+2}\circ K$ denotes the composition of an element of $\I^{3k+2}$ with $K$, which remains in $\I^{3k+2}$ since $K(z)^\beta$ is divisible by $z^\beta$.
    Moreover $K(z)^\beta = z^\beta \tilde k(z^\beta) \mod \I^{3k+2}$ with $\tilde k(u) = \prod_{j=1}^m k_j(u)^{\beta_j}$, and then
    \begin{align*}
         H_i \circ K(z) &= z_i k_i(z^\beta) h_i(z^\beta \tilde k(z^\beta)) \mod \I^{3k+2} \\
         & = z_i r_i(z^\beta) \mod \I^{3k+2}
     \end{align*}
    with $r_i(u) = k_i(u)h_i(u \tilde k(u))$. Hence $\Phi^{-1} \circ F \circ \Phi \in C$. Notice that $r_i$ only depends on $h_i$ and the $k_j$ and not on the terms in $\I^{3k+2}$. Hence, if $\Phi$ is the change of coordinates which removes the resonant terms $z_i R_i(z^\beta)$, then $\Phi^{-1}\circ F \circ \Phi = NF \mod \I^{3k+2}$.
\end{proof}

We will show that $F$ in \eqref{system3} is formally conjugate to the normal form $\hat G$, by a formal transformation in $z^\beta$ with coefficients being holomorphic functions in $z$. This transformation is tangent to the identity up to order $2k+1$ in $z^\beta$.

		\begin{lemma}\label{phiholo}
			Assume that $\lambda$ is Diophantine with respect to $(z^{\beta})$ and let $G:=\hat G$ be the normal form \eqref{system5}. Then, for $1\leq i\leq n$, there exists $\varphi_{i,\geq (2k+1)}$ a formal series in $z^{\beta}$  with holomorphic coefficients in a common neighborhood of $0$ in $z = (z_1, ..., z_n)$, divisible by $(z^{\beta})^{2k+1}$, such that $G \circ \hat \Phi = \hat \Phi \circ F$ with $\hat \Phi = id + \varphi_{\geq (2k+1)}$, where $\varphi = (\varphi_i)_i$.
		\end{lemma}
		
		\begin{proof}
			We can write $f_i(z) = \sum_{l \geq 0}f_{i,l}(z)(z^{\beta})^l$ where the $f_{i,l}$'s are holomorphic functions in a common neighborhood of the origin in $\C^n$ whose Taylor expansions do not contain monomials divisible by $z^{\beta}$. We will use the convention $f_{i,l}(z) = 0$ for all $l < 0$. In what follows, $\{f(z)\}_l$ will denote the coefficient of $(z^{\beta})^l$ the Taylor expansion of which contains no term divisible by $z^{\beta}$. We then define $f_i(z_0,z)$ as $\sum_{l \geq 0}f_{i,l}(z)z_0^l$. We are looking for  $\varphi_i(z_0,z) = \sum_{l \geq (2k+1)} \varphi_{i,l}(z)z_0^l$ with the same properties and such that $z+\varphi(z^{\beta},z)$ conjugates $F$ to its normal form $G$.
			
			We emphasize that the product of $f(z)z_0^l$ and $g(z)z_0^k$ gives rise to terms of degree $\geq k+l$ as the product of a term of $f$ with a term of $g$ could be divisible by $z^{\beta}$, say $z^{q\beta+Q}$ with some $Q \in \N^n$ such that $z^Q$ is not divisible by $z^\beta$. Hence, the product will give rise to sums of terms of the form $h(z)z_0^{k+l}$ and $\tilde h(z)z_0^{k+l+q} \mod z_0-z^{\beta}$.
			Setting $\varphi_0(z^\beta,z) :=  \hat \Phi(z)^\beta -z^\beta$, we extend $G$, $\hat \Phi$ and $\varphi_0$ to $\C^{n+1}$ by setting \begin{align}\label{G0} G_0(z)&:=G^{\beta}(z)=z^{\beta}(1+A (z^{\beta})^k + B(z^{\beta})^{2k} + P((z^{\beta})^k)).\\
			G_0(z_0,z) &:=z_0(1+A (z_0)^k + B(z_0)^{2k} + P((z_0)^k)).\nonumber\\
			\hat \Phi_0(z)&:=\hat  \Phi(z)^{\beta}=z^{\beta}+\varphi_0(z^\beta,z)=z^{\beta}+(z^\beta)^{(2k+1)}\tilde\varphi_0(z^\beta,z)\nonumber\\
			 \hat \Phi_0(z_0,z)&=z_0+z_0^{(2k+1)}\tilde\varphi_0(z_0,z)\nonumber
			\end{align} 
			Notice that actually $G_0(z_0,z)$ does not depend on $z$ but we keep this notation for consistency. 

			According to \re{gbeta}, we have 
			\begin{align*}
				F_0(z)=F^{\beta}(z)=z^{\beta}\prod_{i=1}^m \big(1+ \frac{a_i}{\lambda_i} z^{k \beta} + \frac{b_i}{\lambda_i} z^{2k\beta } +(z^{\beta})^{3k+2}z_i^{-1} \lambda_i^{-1}f_i(z)\big)^{\beta_i}&\\
				=:z^{\beta}\big(1+A (z^{\beta})^k + B(z^{\beta})^{2k} + P((z^{\beta})^k)\big)+ (z^{\beta})^{3k+2}f_0(z)\\
				=G_0(z^{\beta},z)+ (z^{\beta})^{3k+2}f_0(z)&\\
				F_0(z_0,z):=G_0(z_0,z)+ z_0^{3k+2}f_0(z_0,z)&
		\end{align*}
	and we write $f_0(z_0,z)= \sum_{l\geq 0}f_{0,l}(z)z_0^l$ as above.
		We first study the conjugacy equation for the $0$ component: 
			\begin{align}\label{conjG0}
				(G\circ \hat \Phi)_0(z) = (\hat \Phi \circ F)_0(z) \Leftrightarrow G_0(\hat \Phi(z)) = F^{\beta}(z) + \varphi_0(F_0(z),F(z))= F_0(z) + \varphi_0(F_0(z),F(z)) &
			\end{align} 
			Since $G_0(\hat \Phi(z)) =G_0(\hat \Phi(z)^{\beta},z) =G_0(\hat\Phi_0(z^\beta,z),z)$, it is sufficient to solve
			$$
			G_0(\hat \Phi_0(z_0,z)) =  z_0\big(1+Az_0^k  + B z_0^{2k}+P(z_0^k)\big) + z_0^{(3k+2)}f_0(z_0,z) +\varphi_0(F_0(z_0,z),F(z_0,z)),
			$$
			and then to set $z_0=z^{\beta}$.
			By the Taylor formula in $0$, we have
			\begin{align}
				G_0(\hat \Phi_0 (z_0,z))-G_0(z_0) = G_0(z_0 + \varphi_0(z_0,z))-G_0(z_0)&\nonumber\\ 
				 = \sum_{l = 1}^{2|\beta|k+1}\frac{1}{l!}\frac{d^l}{dz_0^l}(z_0 + Az_0^{k+1}+B z_0^{2k+1}+z_0P(z_0^k)) \cdot \varphi_{0}(z_0,z) ^{l}\label{util40term}\\ \qquad 
				= (1 + (k+1)Az_0^{k}+(2k+1)Bz_0^{2k}+P(z_0^k)+kz_0^{k}P'(z_0^k))\varphi_0(z_0,z) + R(\varphi_0)&\nonumber
			\end{align}
			where $R(\varphi_0)$ is non linear in $\varphi_0$ and $P'$ denotes the derivative of $P$. Then, if $\varphi_0(z_0,z) = \sum_{m \geq 2k+1}\varphi_{0,m}( z)z_0^m$ and $\varphi_{0,l} = 0$ for all  $l\leq 2k$, the conjugacy equation becomes 
			\begin{align*}
				& (1 + (k+1)Az_0^{k}+(2k+1)B z_0^{2k}+P(z_0^k)+kz_0^{k}P'(z_0^k))\varphi_0(z_0,z) + R(\varphi_0) \\
				& = z_0^{3k+2}f_0(z_0,z) + \varphi_0(z_0,\lambda z) + \big(\varphi_0(F_0(z_0,z),F(z_0,z))-\varphi_0(z_0, \lambda  z) \big) \\
				\Leftrightarrow & (1 + (k+1)Az_0^{k}+(2k+1)B z_0^{2k}+P(z_0^k)+kz_0^{k}P'(z_0^k))\big(\sum_{m\geq 2k+1}\varphi_{0,m}(z)z_0^m\big) + R(\varphi_0)  \\
				&= z_0^{3k+2}\big(\sum_{m\geq 0}f_{0,m}(z)z_0^m\big) + \sum_{m\geq 2k+1}\varphi_{0,m}(\lambda z)z_0^m + \big(\varphi_0(F_0(z_0,z),F(z_0,z))-\varphi_0(z_0, \lambda z) \big) 
			\end{align*}
Let $l \geq 2k+1$. 
\begin{center}\label{Ij}
$(I_{l-1})$\hspace{1cm}Assume by induction that the $\varphi_{0,j}(z)$, $j=2k+1,...,l-1$, are known and holomorphic on $\D_{\rho_0}^n$, up to constant terms $\varphi_{0,j}(0)$ for $j = l-k, ..., l-1$, which are not determined yet (notice that for $l=2k+1$, this assumption is empty). 
\end{center}
Each $\varphi_{0,j}(0)$ will be chosen at step $j+k$ (indeed, as we will see below, equation \eqref{coho0} at step $j$ does not determine the constant term $\varphi_{0,j}(0)$, which is chosen so that the equation at step $j+k$ has a solution). In particular, if $l\leq 3k$, all the constants $\varphi_{0,j}(0)$, $j=2k+1,...,l-1$ are not determined yet. 
Let us set 
\begin{align}
	 h_l(z):=-\left\{\big((k+1)Az_0^{k}+(2k+1)B z_0^{2k}+P(z_0^k)+kz_0^{k}P'(z_0^k)\big)\left(\sum_{m\geq 2k+1}\varphi_{0,m}(z)z_0^m\right)\right\}_l\nonumber\\
	- \{R(\varphi_0)\}_l(z) +  f_{0,l-3k-2}(z) + \{ \varphi_0(F_0(z_0,z),F(z_0,z))-\varphi_0(z_0,\lambda  z)\}_l\label{h0}
\end{align} 
Recall the convention $f_{0,m}(z) = 0$ for $m < 0$.
 Let us explicit on which $\varphi_{0,j}$ the function $h_l$ depends. Recall that $\varphi_{0,j} = 0$ for $j \leq 2k$, and that a product between two factors can only increase the power of the monomial $z_0$. Notice that:
\begin{itemize}
    \item By \eqref{ecritureP}, $P(z_0^k) + k z_0^k P'(z_0^k) = \sum _{j=3}^{2\vert \beta \vert} (1+kj)c_j z_0^{kj}$. The polynomial $(k+1)Az_0^{k}+(2k+1)B z_0^{2k}+P(z_0^k)+kz_0^{k}P'(z_0^k)$ is then a combination of $z_0^{pk}$ with $p\geq 1$, and 
    \[\left\{\big((k+1)Az_0^{k}+(2k+1)B z_0^{2k}+P(z_0^k)+kz_0^{k}P'(z_0^k)\big)\left(\sum_{m\geq 2k+1}\varphi_{0,m}(z)z_0^m\right)\right\}_l\]
    \[ =(k+1)A \varphi_{0,l-k}(z) + (2k+1)B \varphi_{0,l-2k}(z) + \sum _{j=3}^{2\vert \beta \vert} (1+kj)c_j \varphi_{0,l-kj}(z)\]
    has only terms $\varphi_{0,j}$ with $j\leq l-k$.
    \item By \eqref{util40term}, $R(\varphi_0)$ gathers the terms of order $\geq 2$ in $\varphi_0$ of the Taylor formula, that is to say $R(\varphi_0) = \sum_{p\geq 2} \frac{1}{p!} G_0^{(p)}(z_0)\varphi_0^p (z_0,z)$. As $\varphi_0^{p}(z_0,z)$ is a sum of some products $\varphi_{0,j_1}(z)...\varphi_{0,j_p}(z) z_0^{j_1+...+j_p}$ which gives a $z_0^l$ only if $j_1+...+j_p \leq l$ (since multiplying by $G_0^{(p)}(z_0)$ cannot decrease the power of $z_0$, then if $j_1+...+j_p >l$ the product would give a power of $z_0$ which remains $>l$). As each $j_i \geq 2k+1$, we have $l \geq j_1 + ... + j_p \geq j_i + (p-1)(2k+1)\geq j_i + (2k+1)$ (since $p\geq 2$), which gives the condition $j_i \leq l - 2k -1$ for all $i$. Then $\{ R(\varphi_0)\}_l$ has only terms $\varphi_{0,j}$ with $j\leq l-2k-1$.
    \item Since $\varphi_0(z_0,z) = \sum_{j \geq 2k+1} \varphi_{0,j}(z) z_0^j$, we have
    \begin{align} \label{sommesurF} \varphi_0(F_0(z_0,z),F(z_0,z)) - \varphi_0(z_0,\lambda z) &\\
    	= \sum_{j\geq 2k+1} \varphi_{0,j}(F(z_0,z))F_0(z_0,z)^j - \sum_{j \geq 2k+1} \varphi_{0,j}(\lambda z) z_0^j&\nonumber\\
    	= \sum_{j\geq 2k+1} \big[(\varphi_{0,j}(F(z_0,z))-\varphi_{0,j}(\lambda z)) F_0(z_0,z)^j + \varphi_{0,j}(\lambda z)(F_0(z_0,z)^j-z_0^j) \big].&\nonumber\end{align}
  \end{itemize}
Furthermore, on the one hand, as $F_0(z_0,z) = z_0(1+A z_0^k + \OO(z_0^{k+1}))$,  we have $F_0(z_0,z)^j = z_0^j+Aj z_0^{j+k} + \OO(z_0^{j+k+1})$ so that 
\beq \label{F0j}
F_0(z_0,z)^j - z_0^j= jAz_0^{j+k} + \OO(z_0^{j+k+1})
\eeq is divisible by $z_0^{j+k}$. 

On the other hand, by \eqref{system3} and \eqref{H_0}, we have
\begin{align}
\varphi_{0,j}(F(z_0,z)) - \varphi_{0,j}(\lambda z) 
= \sum_{\vert \alpha \vert \geq 1} \frac{1}{\alpha !} \frac{\partial^{\vert \alpha \vert} \varphi_{0,j}}{\partial z^{\alpha}}(\lambda z).(F(z_0,z)-\lambda z)^\alpha&\label{phi0F}\\
 = \sum_{i=1}^n \frac{ \partial \varphi_{0,j}}{\partial z_i}(\lambda z) . (F_i(z_0,z)-\lambda_i z_i)+\cdots = \sum_{i=1}^n a_{i,1}z_i \frac{\partial \varphi_{0,j}}{\partial z_i}(\lambda z) z_0^k + \OO(z_0^{k+1})\nonumber&
\end{align}
    which is divisible by $z_0^k$ as the terms $\cdots$, corresponding to $\vert \alpha \vert \geq 2$, are divisible by $z_0^{2k}$. Since $F_0(z_0,z)^j$ is divisible by $z_0^j$, the product $(\varphi_{0,j}(F(z_0,z))-\varphi_{0,j}(\lambda z))F_0(z_0,z)^j$ is divisible by $z_0^{j+k}$.
    
    Finally, in the sum \eqref{sommesurF}, $\varphi_{0,j}(z)$ appears only in the terms divisible by $z_0^{j+k}$, so it contributes to the coefficient of $z_0^l$ only if $j+k \leq l$. Then, the coefficient of $z_0^l$ in $\varphi_0(F_0(z_0,z),F(z_0,z)) - \varphi_0(z_0, \lambda z)$ only involves the $\varphi_{0,j}$ with $j\leq l-k$.

Finally, we have proved
\begin{center}\label{Al}
$(A_l)$\hspace{1cm}$h_l(z)$ only involves $f_{0,l-3k-2}(z)$ and the $\varphi_{0,j}(z)$ with $j \leq l-k$ and is 0 whenever all these functions are zero. 
\end{center}
By the iteration assumption ($I_{l-1}$), these $\varphi_{0,j}(z)$ are known for $j \leq l-k$, except the constant term $\varphi_{0,l-k}(0)$.  
Identifying the coefficients of $z_0^l$ in the conjugation equation, we get
			\begin{equation} \label{coho0} \varphi_{0,l}( z) - \varphi_{0,l}(\lambda z) = h_l(z).\end{equation}
            
Write $\varphi_{0,l}(z) = \sum _Q \varphi_{0,l,Q}z^Q$ and $h_l(z) = \sum_Q h_{l,Q}z^Q$ where $z^Q$ is not divisible by $z^\beta$. Equation \eqref{coho0} reads $(1-\lambda^Q)\varphi_{0,l,Q} = h_{l,Q}$. For $Q \neq 0$, $\lambda^Q \neq 1$ by $1$-resonance, and $\varphi_{0,l,Q} = \frac{h_{l,Q}}{1 - \lambda^Q}$ is well determined. If $Q = 0$, the equation implies $h_{l,0} = h_{l}(0) = 0$ and the constant $\varphi_{0,l,0} = \varphi_{0,l}(0)$ is free.

Let us prove that we can satisfy this condition $h_{l}(0) = 0$ by choosing the constant $\varphi_{0,l-k}(0)$ which was left free at step $l-k$, in a similar way as we did in the proof of Lemma \ref{changement}.

Indeed, we first recall that $\varphi_{0,j}(z) = 0$ for $j \leq 2k$. Hence the constant $\varphi_{0,l-k}(0)$ can be chosen freely only if $l-k \geq 2k+1$, that is to say $l \geq 3k+1$. 
There are then two cases to consider to prove that the condition $h_l(0)=0$ holds: 

\begin{enumerate}
    \item \label{case1} If $2k+1 \leq l \leq 3k$, then $l-k \leq 2k$ so, according to ($A_l$), all the $\varphi_{0,j}(z)$ involved in $h_l(z)$ are $0$, and so is $f_{0,l-3k-2}(z)$. Hence $h_l(z) = 0$ and the condition $h_l(0) = 0$ holds, and \eqref{coho0} gives $\varphi_{0,l}(z) = \varphi_{0,l}(0)$ which is a constant which remains free. 
    \item If $l \geq 3k+1$, we compute $h_l(0)$ using \re{h0}. The constant $\varphi_{0,l-k}(0)$ appears only in two terms, namely:
    \begin{itemize}
        \item In $-\left\{\big((k+1)Az_0^{k}+...\big)\left(\sum_{m\geq 2k+1}\varphi_{0,m}(z)z_0^m\right)\right\}_l$ with the coefficient $-(k+1)A$,
        \item In the term of index $j = l-k$ of the sum $\sum _{j \geq 2k+1} \varphi_{0,j}(\lambda z)(F_0(z_0,z)^{j} - z_0^{j})$. 
        Indeed, according to \re{F0j}, the coefficient of $z_0^l$ of $\varphi_{0,l-k}(\lambda z)(F_0(z_0,z)^{l-k} - z_0^{l-k})$ is $(l-k)A \varphi_{0,l-k}(\lambda z)$, which gives $(l-k)A \varphi_{0,l-k}(0)$ at $z=0$.
    \end{itemize}
The term $ \sum_{i=1}^n a_{i,1} z_i \frac{ \partial \varphi_{0,l-k}}{\partial z_i}(\lambda z)$ (coming from the term of index $j=l-k$ of $\sum_{j \geq 2k+1}  \big(\varphi_{0,j}(F(z_0,z)) - \varphi_{0,j}(\lambda z)\big)F_0(z_0,z)^j$) vanishes at $z=0$, and all the other terms involve $f_{0,l-3k-2}$ and the $\varphi_{0,j}$ with $j < l-k$ which are known by induction assumption.
     Hence 
     \begin{equation}\label{equahl0}h_l(0) = A(l-2k-1)\varphi_{0,l-k}(0) + \mathcal{F}_{0,l}\end{equation} where $\mathcal{F}_{0,l}$ are  known terms. Since $A \neq 0$ and $l-2k-1 \geq k \geq 1$, there is a unique choice of $\varphi_{0,l-k}(0)$ such that $h_l(0)=0$. Then \eqref{coho0} gives a solution $\varphi_{0,l}(z)$ up to the constant term $\varphi_{0,l}(0)$ which remains free. 
\end{enumerate}

We now show that ($I_{l}$) holds.
By assumption ($I_{l-1}$), the constants $\varphi_{0,j}(0)$ are known for $2k+1 \leq j \leq l-k-1$. The constants $\varphi_{0,j}(0)$ for $l-k \leq j \leq l-1$ are not known yet.
Notice that, by the two cases (1) and (2) above, at step $l$:
\begin{enumerate}
    \item if $l \leq 3k$, then $l-k \leq 2k$ and by assumption on $\varphi$ ($\varphi_{0,j}(z)=0$ for all $j\leq 2k$), $\varphi_{0,l-k}(z)= \varphi_{0,l-k}(0)=0$;
    \item if $l \geq 3k+1$: $\varphi_{0,l-k}(0)$ is chosen such that $h_l(0)=0$ by \re{equahl0};
\end{enumerate}
and then the constant $\varphi_{0,l-k}(0)$ is determined at step $l$. Moreover, $\varphi_{0,l}(z)$ is determined, up to its constant term $\varphi_{0,l}(0)$ by \re{coho0}.
Hence, after step $l$, the constants $\varphi_{0,j}(0)$ are known for $2k+1 \leq j \leq l-k$, but they are not known yet for $l-k+1 \leq j \leq l$. This is exactly ($I_{l}$).
Finally, by induction, ($I_{l-1}$) holds for all $l \geq 2k+1$. Each $\varphi_{0,j}(z)$ is determined at step $j$, up to its constant term $\varphi_{0,j}(0)$, which is determined at step $j+k$. Therefore $\varphi_{0}(z_0,z)$ is completely determined.

Let us show that 
\beq \label{phi2k+1}
\varphi_{0,2k+1}(0)=\varphi_{0,2k+1}(z)=0,
\eeq
This will be useful below in the proof.
The case \ref{case1} above with $l=2k+1$ shows that $\varphi_{0,2k+1}(z) = \varphi_{0,2k+1}(0)$ is constant. 
Let now $l=3k+1$. By Equation \eqref{equahl0} of case (2), $h_{3k+1}(0) = A k \varphi_{0,2k+1}(0) + \mathcal{F}_{0,3k+1}$, with $\mathcal{F}_{0,3k+1}=0$ since it depends only on $f_{0,-1}(z)=0$ (by convention) and on the $\varphi_{0,j}(z)$ with $j < l-k = 2k+1$ (which are $=0$ by assumption). The dependence of $h_{3k+1}(0)$ on $\varphi_{0,2k+1}(0)$ is localized in the linear term.
Finally
\[ h_{3k+1}(0) = Ak \varphi_{0,2k+1}(0).\]
Since $A\neq 0$ and $k\geq 1$, the condition $h_{3k+1}(0)=0$ (required to have a solution of \eqref{coho0}) gives $\varphi_{0,2k+1}(0)=0$ and then $\varphi_{0,2k+1}(z)=0$.
          
			Let $\rho_l = \rho_0 \prod_{j=1}^l(1-2^{-j})$ for $l \geq 2k+1$. Let us assume that for all $j < l$ the coefficients $\varphi_{0,j}$ are holomorphic and {\it bounded} on $\D_{\rho_j}^n$. Then $h_l$ is holomorphic and bounded on $\D_{\rho_{l-1}}$. According to \cite[Lemma. 2.8.4]{Kat} and since $\lambda$ is Diophantine, $\varphi_{0,l}$ is holomorphic and bounded on $\D_{\rho_{l}}^n$ which proves the induction. 
			
			\bigskip
			We now study in the same way the coefficients of $\varphi_i$ for $i=1,...,n$. 
			We construct a complete conjugacy $\hat \Phi = (\hat \Phi_1,...,\hat\Phi_n)$, with each $\hat \Phi_i$ tangent to the identity such that
			\begin{equation} \hat\Phi \circ  F = G \circ\hat\Phi \label{conjuggf}\end{equation}
			which satisfies
			\begin{equation}\label{constraint} \hat\Phi(z)^\beta = \hat \Phi_0(z^\beta,z).\end{equation}
			We shall construct first $\hat \Phi_{i}$ for $i=1,...,m-1$, assuming the constraint \eqref{constraint} is satisfied. Then we will construct $\hat\Phi_{m}$ to ensure that condition \eqref{constraint} is indeed satisfied. Finally we shall construct $\hat\Phi_{i}$ for $i = m+1,...,n$. 
			For all $i=1,...,m-1$, let us set
			\[ \hat\Phi_{i}(z) = z_i + \varphi_{i}(z).\]
			For a given $1\leq i\leq m-1$, on the one hand, we have
			\[ \hat\Phi_{i}(F(z)) = F_{i}(z) + \varphi_{i}(F(z)) = G_i(z) + (z^{\beta})^{3k+2}f_i(z) + \varphi_{i}(F(z)). \]
			On the other hand, according to \eqref{constraint}, we have 
			\begin{align*}
				G_i(\hat\Phi(z)) & = \hat\Phi_{i}(z)\big(\lambda_i+a_i(\hat\Phi(z)^\beta)^k + b_i (\hat\Phi(z)^\beta)^{2k}\big) \\
				& = \hat\Phi_{i}(z)\big(\lambda_i+a_i(\hat\Phi_0(z^\beta,z))^k + b_i (\hat\Phi_0(z^\beta,z))^{2k}\big) \\
				& = (z_i + \varphi_{i}(z))\big(\lambda_i+a_i(\hat\Phi_0(z^\beta,z))^k + b_i(\hat\Phi_0(z^\beta,z))^{2k}\big).
			\end{align*}
			Hence, the $i$-th component of conjugacy equation \eqref{conjuggf} can be written as
			\begin{align} & G_i(z) +(z^{\beta})^{3k+2}f_i(z)  + \varphi_{i}(F(z)) = \label{equi}\\
				& (z_i + \varphi_{i}(z))\big(\lambda_i+a_i(\hat\Phi_0(z^\beta,z))^k + b_i (\hat\Phi_0(z^\beta,z))^{2k}\big).\nonumber
			\end{align}
			Note that
			\begin{align*}
				G_i(z) - z_i \big(\lambda_i+a_i(\hat \Phi_0(z^\beta,z))^k + b_i (\hat\Phi_0(z^\beta,z))^{2k}\big)
				& = z_ia_i((z^\beta)^k - \hat\Phi_0(z^\beta,z)^k) \\
				&+ z_ib_i ((z^\beta)^{2k}-\hat\Phi_0(z^\beta,z)^{2k}).
			\end{align*}
			Since $\hat\Phi_0 (z^{\beta},z)= z^{\beta} + \varphi_{0}(z^{\beta},z)=z^{\beta}+\left(\sum_{m\geq 2k+1}\varphi_{0,m}(z)(z^{\beta})^m\right)$ with $\varphi_{0,l}$ all holomorphic on a same domain, we obtain 
			\begin{align}\label{equif}
				\varphi_{i}(\lambda z)-\lambda_i\varphi_{i}(z)=\varphi_{i}(z)\big(a_i(\hat\Phi_0(z^\beta,z))^k + b_i (\hat\Phi_0(z^\beta,z))^{2k}\big)&\nonumber\\
			-z_i\left(a_i((z^\beta)^k - \hat\Phi_0(z^\beta,z)^k)+b_i((z^\beta)^{2k}-\hat\Phi_0(z^\beta,z)^{2k})\right)\nonumber\\
			-(z^{\beta})^{3k+2}f_i(z)+(\varphi_{i}(\lambda z)-\varphi_{i}(F(z))).
			\end{align}
			For $l\geq 2k+1$, the $l$-th component  of \re{equif} (i.e $\{\re{equif}\}_l$) reads
			\begin{equation} \label{cohoi} \lambda_i\varphi_{i,l}(z) - \varphi_{i,l}(\lambda z) = h_{i,l}(z),\end{equation}
            where $h_{i,l} := - \{ \text{right-hand side of \eqref{equif}} \}_l$.

            Write $\varphi_{i,l}(z) = \sum_Q \varphi_{i,l,Q} z^Q$ and $h_{i,l}(z) = \sum_Q h_{i,l,Q} z^Q$, where $z^Q$ is not divisible by $z^\beta$ (in particular, $\varphi_{i,l,E_i}$ and $h_{i,l,E_i}$ are the coefficients of $z_i$).
			For $Q \neq E_i$, we have by $1$-resonance $\lambda ^Q \neq \lambda_i$ so $\varphi_{i,l,Q} = \frac{h_{i,l,Q}}{\lambda_i - \lambda ^Q}$ is determined. For $Q = E_i$, we have $\lambda ^{E_i} = \lambda_i$ and the equation becomes $0 = h_{i,l,E_i}$. Then $\varphi_{i,l,E_i}$ is free, and the equation has a solution only if $h_{i,l,E_i} = 0$. As for $\varphi_0$, the coefficient $\varphi_{i,l,E_i}$ is not chosen at step $l$, but at step $l+k$, in order to have $h_{i,l+k,E_i} = 0$.

            By the same reasoning as for $\varphi_0$, $h_{i,l}(z)$ only involves $f_{i,l-3k-2}(z)$, $\hat \Phi_0(z^\beta, z)$ (already determined) and the $\varphi_{i,j}(z)$ with $j \leq l-k$. 
            
Notice that again,
\[ \varphi_i(F(z))-\varphi_i(\lambda z) = \sum_{j\geq 2k+1} (\varphi_{i,j}(F(z))-\varphi_{i,j}(\lambda z))F_0(z^\beta,z)^j +\sum_{j\geq 2k+1} \varphi_{i,j}(\lambda z) (F_0(z^\beta,z)^j-(z^\beta)^j). \]

Writing $\varphi_{i,j}(z) = \varphi_{i,j,E_i}z_i + \sum_{Q \neq E_i}\varphi_{i,j,Q}z^Q$, and
recalling that $h_{i,l} := - \{ \text{right-hand side of \eqref{equif}} \}_l$,  $\varphi_{i,l-k,E_i}$ only appears in the three terms of $h_{i,l,E_i}$:
            \begin{itemize}
            \item $-\varphi_i(z)a_i \hat \Phi_0(z^\beta,z)^k$ with coefficient $-a_i$,
                \item  $(\varphi_{i,l-k}(F(z)) - \varphi_{i,l-k}(\lambda z)) F_0(z^\beta,z)^{l-k}
                = \varphi_{i,l-k,E_i} (F_i(z)-\lambda_i z_i) F_0(z^\beta,z)^{l-k}+ \text{h.o.t}$. Since $F_i(z)-\lambda_i z_i = a_i z_i (z^\beta)^k + ...$ and $F_0(z^\beta,z)^{l-k} = (z^\beta)^{l-k} + ...$, we obtain  $$(\varphi_{i,l-k}(F(z)) - \varphi_{i,l-k}(\lambda z)) F_0(z^\beta,z)^{l-k}=a_i\varphi_{i,l-k,E_i} z_i (z^\beta)^l +\text{h.o.t}.$$
                \item \begin{align*}
                \varphi_{i,l-k}(\lambda z) (F_0(z^\beta,z)^{l-k}-(z^\beta)^{l-k})= \lambda_i \varphi_{i,l-k,E_i} z_i (F_0(z^\beta,z)^{l-k} - (z^\beta)^{l-k})+\text{h.o.t}&\\
                = \lambda_i \varphi_{i,l-k,E_i} z_i (l-k) A(z^\beta)^l+\text{h.o.t}&.\end{align*} 
            \end{itemize}
            Finally
\[ h_{i,l,E_i} = - a_i \varphi_{i,l-k,E_i} + a_i \varphi_{i, l-k, E_i} + \lambda_i(l-k)A \varphi_{i,l-k,E_i} + (\text{known terms})\] \[=  \lambda_i(l-k)A \varphi_{i,l-k,E_i} + (\text{known terms}).\]
 
There are, again, two cases to study to prove that the condition $h_{i,l,E_i}=0$ holds:
        \begin{enumerate}
        \item  If $l \geq 3k+1$, then $l-k \geq 2k+1$ and $\lambda_i (l-k) A \neq 0$, so there is a unique choice of $\varphi_{i,l-k,E_i}$ such that $h_{i,l,E_i} = 0$ by the previous equation.
        \item If $2k+1 \leq l \leq 3k$, then $l-k \leq 2k$ and all the $\varphi_{i,j}(z)$ involved in $h_{i,l}(z)$ are $=0$ and $f_{i,l-3k-2}(z)=0$ by convention (since $l-3k-2 <0$). Moreover, according to \re{phi2k+1}, $\varphi_{0,j}(z) = 0$ for $j \leq 2k+1$, so that $\varphi_0(z^\beta,z)$ is divisible by $(z^\beta)^{2k+2}$. It remains to consider terms of \eqref{equif} involving $\hat \Phi_0(z^\beta,z)$ namely:
        \begin{itemize}
        \item $-z_ia_i((z^\beta)^k - \hat \Phi_0(z^\beta,z)^k)$, which is divisible by $(z^\beta)^{k-1}\varphi_0(z^\beta,z)$, hence by $(z^\beta)^{3k+1}$,
        \item $-z_i b_i((z^\beta)^{2k} - \hat \Phi_0(z^\beta,z)^{2k})$, which is divisible by $(z^\beta)^{2k-1}\varphi_0(z^\beta,z)$, hence by $(z^\beta)^{4k+1}$
        \end{itemize}
        so these two terms have no coefficient of $(z^\beta)^l$ for $l \leq 3k$. Finally $h_{i,l}(z) = 0$. In particular, $h_{i,l,E_i}=0$ which is the required condition, and \eqref{cohoi} gives $\varphi_{i,l}(z) = \varphi_{i,l,E_i}z_i$.
        \end{enumerate}

In both cases, \eqref{cohoi} has a solution and $\varphi_{i,l}(z)$ is determined up to $\varphi_{i,l,E_i}$, which will be chosen at step $l+k$. Applying the same reasoning as for $\varphi_{0,l}(z)$, we prove by induction that the $\varphi_{i,l}(z)$ are all defined and holomorphic on a same polydisc.

            \textbf{Construction of $\hat\Phi_{m}$:} Let us set
			\beq\label{Phim}\hat\Phi_{m}(z) := \left(\frac{\hat\Phi_0(z^\beta,z)}{\hat \Phi_{1}(z)^{\beta_1}\cdots\hat \Phi_{m-1}(z)^{\beta_{m-1}}}\right)^\frac{1}{\beta_{m}}.\eeq
			
			We first prove that $\hat\Phi_{m}$ is well defined: notice that
			\[\prod_{i=1}^{m-1}\hat\Phi_{i}(z)^{\beta_i} = \prod_{i=1}^{m-1}(z_i + \varphi_{i}(z))^{\beta_i} = \frac{z^\beta}{z_m^{\beta_m}} + \hat\sigma(z)\]
            
			where $\hat\sigma$ is a formal power series in $z^{\beta}$ (of positive order), with coefficients holomorphic in a same polydisc.
    
            We have
			\begin{align}\label{phif} \frac{\hat\Phi_0(z^\beta,z)}{\hat\Phi_{1}(z)^{\beta_1}\cdots\hat\Phi_{m-1}(z)^{\beta_{m-1}}}
				= \frac{z^\beta + \varphi_{0}(z ^\beta,z)}{\frac{z^\beta}{z_m^{\beta_m}} +  \hat\sigma(z)} &\\
				 = z_m^{\beta_m} \left( \frac{1 + \frac{\varphi_{0}(z^\beta,z)}{z^\beta}}{1 + z_m^{\beta_m}\frac{\hat\sigma(z)}{z^\beta}}\right) = z_m^{\beta_m}(1 + \hat\varrho(z)) &\nonumber
			\end{align}
			where $\hat\varrho$ is power series in $z^{\beta}$, with no constant terms and with coefficients holomorphic on a same polydisc.

            Let us show that $\hat \varrho(z)$ is divisible by $(z^\beta)^{2k+1}$. It is enough to show it for $\frac{\varphi_0(z^\beta,z)}{z^\beta}$ and $z_m^{\beta_m} \frac{\hat \sigma(z)}{z^\beta}$. 
            
            Since $\varphi_{0,l}(z)=0$ for $l\leq 2k$ and also $\varphi_{0,2k+1}(z)=0$, we have
            \[\frac{\varphi_0(z^\beta,z)}{z^\beta} = \sum_{l\geq 2k+2} \varphi_{0,l}(z)(z^\beta)^{l-1}, \]
            which is divisible by $(z^\beta)^{2k+1}$.
     
            Moreover, since $\prod_{i=1}^{m-1} z_i^{\beta_i} = \frac{z^\beta}{z_m^{\beta_m}}$, we have 
            \[ z_m^{\beta_m} \frac{\hat \sigma(z)}{z^\beta} = \prod_{i=1}^{m-1} (1+ \frac{\varphi_i(z)}{z_i})^{\beta_i}-1,\]
            with, by definition of the $\varphi_i$, $\frac{\varphi_i(z)}{z_i} = \sum_{l\geq 2k+1} \frac{\varphi_{i,l}(z)}{z_i}(z^\beta)^l$. We have
            \begin{itemize}
            \item for $l=2k+1$, $\varphi_{i,2k+1}(z) = \varphi_{i,2k+1,E_i} z_i$, so that $\frac{\varphi_{i,2k+1}(z)}{z_i}(z^\beta)^{2k+1} = \varphi_{i,2k+1,E_i} (z^\beta)^{2k+1}$,
            \item for $l\geq 2k+2$, $\frac{\varphi_{i,l}(z)}{z_i}(z^\beta)^l = \varphi_{i,l}(z)\frac{z^\beta}{z_i} (z^\beta)^{l-1}$ is divisible by $(z^\beta)^{2k+1}$ (notice that since $\beta _i \geq 1$ for $i=1,...,m-1$, $\frac{z^\beta}{z_i} = z^{\beta- E_i}$ is a monomial). 
            \end{itemize}
            Hence, $z_m^{\beta_m} \frac{\hat \sigma(z)}{z^\beta}$ is divisible by $(z^\beta)^{2k+1}$. Taking the $\beta_m$-th root which is tangent to the identity, we get
\[ \hat \Phi_m(z) = z_m(1+ \hat \varrho(z))^{\frac{1}{\beta_m}} = z_m + \OO((z^\beta)^{2k+1}).\]

			Let us check that $\hat\Phi_{m}$ satisfies the conjugacy equation:
			\[\hat\Phi_{m}(F(z)) = G_m(\hat\Phi(z)).\]
			By definition, we have $\hat\Phi_i(F(z))=G_i(\hat\Phi(z))$, $i=1,\ldots m-1$ so that
			\[\hat\Phi_{m}(F(z))^{\beta_m} = \frac{\hat\Phi_0(F_{0}(z),F(z))}{\prod_{i=1}^{m-1}G_i(\hat\Phi(z))^{\beta_i}}.\]
			We have, on the one hand, by \re{conjG0} 			
            \[\prod_{i=1}^mG_i(\hat\Phi(z))^{\beta_i} = G_0(\hat\Phi(z)^\beta) = G_0(\hat\Phi_0(z^\beta,z),z) = \hat\Phi_0(F_0(z),F(z)) \]
			On the other hand
			\[ \prod_{i=1}^m\hat\Phi_{i}(F(z))^{\beta_i} = \hat\Phi_0(F(z)^\beta,F(z)) = \hat\Phi_0(F_{0}(z),F(z)).\]
			Combining these two equalities, we get
			\[\prod_{i=1}^mG_i(\hat\Phi(z))^{\beta_i} =\prod_{i=1}^m\hat\Phi_{i}(F(z))^{\beta_i}\] 
            We already know that $G_i(\hat\Phi(z))=\hat\Phi_{i}(F(z))$ for all $i=1,\ldots,m-1$, so that
            \[G_m(\hat\Phi(z))^{\beta_m} = \hat\Phi_{m}(F(z))^{\beta_m}.\]
            Both sides are of the form $G_m(z)^{\beta_m}(1+\tilde {\hat\varrho}(z))$ with $\tilde {\hat\varrho}(z)$ a formal power series in $z^\beta$ with no constant term, hence taking the $\beta_m$-th root which is tangent to the identity we get
\[G_m(\hat\Phi(z)) = \hat\Phi_{m}(F(z)).\]
		\textbf{Construction of $\hat\Phi_{m+1},\ldots,\hat\Phi_{n}$:} Let 
		us show that we can find $\hat\Phi_{i}(z) = z_i + \varphi_{i}(z)$ for all $i=m+1,\ldots,n$ such that
		\beq\label{equivi} \hat\Phi_{i}(F(z)) = G_i(\hat\Phi(z))\eeq
		with $G_i(z) = z_i(\lambda_i + Q_i(z^{ \beta}))$. We have
		\begin{align*}
			\hat\Phi_{i}(F(z)) & = F_{i}(z)+\varphi_{i}(F(z)) \\
			& = G_i(z) + (z^{\beta})^{3k+2}f_i(z) + \varphi_{i}(F(z)) \\
		G_i(\hat\Phi(z)) & = \hat\Phi_{i}(z)(\lambda_i+Q_i(\hat\Phi(z)^{\beta}))\\
			& = (z_i + \varphi_{i}(z))(\lambda_i+Q_i(\hat\Phi_0(z^\beta,z))) 
		\end{align*}
		since $\hat\Phi(z)^\beta = \hat\Phi_0(z^\beta,z)$. Hence, \re{equivi} is equivalent to
				\beq\label{equgenf}
			z_i(\lambda_i+Q_i(z^{\beta}))+(z^{\beta})^{3k+2}f_i(z) + \varphi_{i}(F(z)) = (z_i + \varphi_{i}(z))(\lambda_i+Q_i(\hat\Phi_0(z^\beta,z))). 
		\eeq
		
		Notice that $\hat\Phi_0(z^{\beta},z)=z^{\beta}+ \varphi_0(z^{\beta},z)$, with $\varphi_0(z^{\beta},z)=(z^{\beta})^{2k+1}\tilde\varphi_0(z^{\beta},z) $  (since $\varphi_{0,l}(z) = 0$ for $l\leq 2k$) formal with respect to $z^{\beta}$, uniformly in $z$, and that $\tilde \varphi_0(z^\beta,z)$ is divisible by $z^\beta$ since $\varphi_{0,2k+1}(z)=0$. Moreover, since $\ord _0 Q_i \geq k$ by \eqref{H_0} then $Q_i(z^\beta)-Q_i(\hat \Phi_0(z^\beta,z))$ is divisible by $(z^\beta)^{k-1}\varphi_0(z^\beta,z)$, hence by $(z^\beta)^{3k+1}$.
        Then we can write
		\[ z_i(\lambda_i+Q_i(z^{\beta})) -  z_i(\lambda_i+Q_i(\hat\Phi_0(z^{\beta},z))) =: z_i \varepsilon_i(z^{\beta},z)=:z_i(z^{\beta})^{3k+1}\tilde \varepsilon_i(z^{\beta},z)\]
		where $\varepsilon_i,\tilde \varepsilon_i $ are formal power series in $z^{\beta}$, with coefficients holomorphic on a same polydisc.
		Then equation \re{equivi} is equivalent to
		$$
		\varphi_i(z)(\lambda_i+Q_i(\hat\Phi_0(z^\beta,z)))= (z^{\beta})^{3k+2}f_i(z)+z_i(z^{\beta})^{3k+1}\tilde \varepsilon_i(z^\beta,z) + \varphi_{i}(F(z))
		$$
		This can be rewritten as
		\begin{align}\label{equisol}
		\lambda_i\varphi_i(z)-\varphi_i(\lambda z)&=(z^{\beta})^{3k+2}f_i(z)+z_i(z^{\beta})^{3k+1}\tilde \varepsilon_i(z^\beta,z)\\
		&-\varphi_i(z)Q_i(\hat\Phi_0(z^\beta,z))+(\varphi_{i}(F(z))-\varphi_i(\lambda z)).\nonumber\\
		&=: h_i(z)\nonumber
		\end{align}
		 As above, we expand the right hand side with respect to $z^{\beta}$ with coefficients in $z$ the Taylor expansions at $0$ of which do not contain terms divisible by $z^{\beta}$. 
         
         For $i=m+1,...,n$, we write $h_{i}(z) =\sum_{l\geq 2k+1}h_{i,l}(z)(z^{\beta})^l$ with the same convention as above, namely that the expansion at $0$ of $h_{i,l}(z)$ is not divisible by $z^{\beta}$. Since $\lambda^{\beta}=1$, the $l$-th component of $\eqref{equisol}$ reads
         \begin{equation}\label{mathcalF}\lambda_i\varphi_{i,l}(z)-\varphi_{i,l}(\lambda z)=h_{i,l}(z).\end{equation}
       As for $i=1,...,m-1$, since $\ord_0 Q_i \geq k$, $h_{i,l}(z)$ depends only on $f_{i,l-3k-2}(z)$, $\hat \Phi_0(z^\beta,z)$ and $\varphi_{i,j}(z)$ with $j \leq l-k$. Write again $\varphi_{i,l}(z) = \sum_{Q} \varphi_{i,l,Q}z^Q$, and $h_{i,l}(z) = \sum_Q h_{i,l,Q}z^Q$.
       For $Q \neq E_i$, we have $\lambda^Q \neq \lambda_i$ by $1$-resonance, and $\varphi_{i,l,Q} = \frac{h_{i,l,Q}}{\lambda_i-\lambda^Q}$. For $Q=E_i$, the equation reads $h_{i,l,E_i} = 0$, and $\varphi_{i,l,E_i}$ is free and will be chosen at step $l+k$.
       As for $i = 1,...,m-1$, the term $\varphi_{i,l-k,E_i}$ appears in $h_{i,l,E_i}$ in
       \begin{itemize}
           \item $-\varphi_i(z) Q_i(\hat \Phi_0(z^\beta,z))$ with coefficient $-a_{i,1}$,
           \item $(\varphi_{i,l-k}(F(z))-\varphi_{i,l-k}(\lambda z))F_0(z^\beta,z)^{l-k}$ with coefficient $+a_{i,1}$ (since $F_i(z) - \lambda_i z_i = a_{i,1}z_i(z^\beta)^k+...$),
           \item $\varphi_{i,l-k}(\lambda z)(F_0(z^\beta,z)^{l-k}-(z^\beta)^{l-k})$ with coefficient $\lambda_i(l-k)A$ (since $F_0(z^\beta,z)^{l-k} - (z^\beta)^{l-k} = (l-k)A (z^\beta)^l+...$),
       \end{itemize}
        and finally $h_{i,l,E_i} = \lambda_i (l-k)A \varphi_{i,l-k,E_i} + \text{(known terms)}$.
		The two cases to study again:
        \begin{itemize}
            \item If $l \geq 3k+1$, then $l-k \geq 2k+1$ and $\lambda_i(l-k)A \neq 0$, so there is a unique choice of $\varphi_{i,l-k,E_i}$ such that $h_{i,l,E_i} = 0$.
            \item If $2k+1 \leq l \leq 3k$, then all the $\varphi_{i,j}(z)$ involved in $h_{i,l}(z)$ are $=0$, and $f_{i,l-3k-2}(z) =0$ so the only remaining term of $h_{i,l}(z)$ comes from $z_i (z^\beta)^{3k+1} \tilde \varepsilon_i(z^\beta, z)$, which is divisible by $(z^\beta)^{3k+1}$ and then all the terms appearing in $h_{i,l}(z)$ are all $=0$. Hence $h_{i,l}(z) = 0$ and $\varphi_{i,l}(z) = \varphi_{i,l,E_i}z_i$.
        \end{itemize}

        In both cases, the equation \eqref{mathcalF} has a solution and $\varphi_{i,l}(z)$ is determined up to $\varphi_{i,l,E_i}$, which will be chosen at step $l+k$. Applying the same reasoning as for $\varphi_{0,l}(z)$, we prove by induction that the $\varphi_{i,l}(z)$ are all defined and holomorphic on a same polydisc.

        Finally, we have proved that there are formal power series in $z_0$, with coefficients holomorphic in a same polydisc in $z$, denoted $\hat \Phi_0(z_0,z), \hat \Phi_1(z_0,z),\ldots, \hat \Phi_n(z_0,z)$, such that the formal diffeomorphism $\hat\Phi(z):=(\hat\Phi_i(z))_{i=1,\ldots,n}$ with $\hat\Phi_i(z):=\hat \Phi_i(z^{\beta},z)$, and $\hat\Phi(z)^{\beta}=\hat \Phi_0(z^{\beta},z)$ solves the problem.
	\end{proof}

	\begin{prop}\label{centralisateur}
		The only formal diffeomorphism $\hat \Phi$ in $z^{\beta}$, tangent to identity at the origin and of the form $z\mapsto \left( z_i+\sum_{l \geq 2k+1} \varphi_{i,l}(z)(z^{\beta})^l\right)_{i=1,\ldots, n}$ (as the formal normalization of Lemma \ref{phiholo}) that commutes with $\hat G$ is the identity.
	\end{prop}
	\begin{proof}
		Recall $G=\hat G$. We follow the previous proof with $F= G$. Let us first consider \re{conjG0}. We solve by induction on $l=0,\ldots$, the degree of $z^{\beta}$. The 0-degree term of the left hand side of this equation is 
		$$
		 \varphi_{0,0}(z)+ A\varphi_{0,0}(z)^{k+1}+ B\varphi_{0,0}(z)^{2k+1}+\sum_{j=3}^{2|\beta|}c_j\varphi_{0,0}(z)^{kj+1}
		$$
		We have 
		\begin{align}
			\varphi_0(G_0(z_0,z),G(z_0,z))-\varphi_0(z_0,\lambda  z)= \sum_{l\geq 0}\underbrace{\left(\varphi_{0,l}(G(z_0,z))-\varphi_{0,l}(\lambda z)\right)G_0(z_0,z)^l}_{(D\varphi_{0,l}(\lambda z)).(zz_0^k+\cdots)(z_0+\cdots)^l}&\nonumber\\+\varphi_{0,l}(\lambda z)\underbrace{(G_0(z_0,z)^l-z_0^l)}_{z_0^{l-1}.(lAz_0^{k+1}+\cdots)}&\label{phiG}
		\end{align}
		We recall $\{h(z)\}$ denotes the part of the Taylor expansion of $h(z)$ at the origin containing only terms not divisible by $z^{\beta}$.
		Hence, the degree $0$-equation \re{h0} reads 
		\begin{equation}
\label{centrah0}		 \varphi_{0,0}( z) - \varphi_{0,0}(\lambda z) = h_0(z)=:-\{ A\varphi_{0,0}(z)^{k+1}+ B\varphi_{0,0}(z)^{2k+1}+\sum_{j=3}^{2|\beta|}c_j\varphi_{0,0}(z)^{kj+1}\}_0.
		\end{equation}
		Since $\hat \Phi$ is tangent to the identity, the order of $\varphi_{0,0}(z)$ at the origin with respect to $z$ is $\geq 2$. Since moreover $\varphi_{0,0}$ does not contains terms divisible by $z^{\beta}$, it does not contain resonant terms.
        
        Assume that $\varphi_{0,0}(z) \neq 0$ and let $\varphi_{0,0, Q}$ be a $\neq 0$ term of minimal degree $d=|Q|$. Identifying the coefficient of $z^Q$ in \eqref{centrah0}, we get $(1-\lambda^Q)\varphi_{0,0,Q}= h_{0,Q}$. Since, by \eqref{centrah0}, $h_0(z)$ only involves powers $\varphi_{0,0}(z)^p$ with $p \geq k+1$, it has order $\geq (k+1)d >d$, and then $h_{0,Q}=0$. Since $1-\lambda^Q \neq 0$, we get $\varphi_{0,0,Q}=0$, which is a contradiction.  So  $\varphi_{0,0}=0$.
		For $l=1$, according to \re{util40term} and \re{phiG}, we have $h_1=0$. Hence $\varphi_{0,1}(z)-\varphi_{0,1}(\lambda z)= h_1(z)=0$. Since, by assumption on the decomposition of $\hat \Phi$, $\varphi_{0,1}(z)$ has no term divisible by $z^\beta$, by $1$-resonance it is constant, and since $\hat \Phi$ is tangent to the identity, $\varphi_{0,1}(z)=0$.
        Assume by induction that $\varphi_{0,j}(z)=0$ for $j < l$. Then $h_l=0$, so that $\varphi_{0,l}(z)-\varphi_{0,l}(\lambda z)=0$. Since $\varphi_{0,l}(z)$ does not have terms divisible by $z^\beta$, and by $1$-resonance, the only solutions are the constants $\varphi_{0,l}(z)=c_l \in \C$.
        We will compare the coefficients of the monomial $(z^\beta)^{l+k}$ in $\hat \Phi_0 \circ G = G_0 \circ \hat \Phi_0$. We have 
\[ \hat \Phi_0(z) = z^\beta + c_l (z^\beta)^l + \sum_{j > l}\varphi_{0,j}(z)(z^\beta)^j.\]
Notice that $G(z)^Q$ is the product of $z^Q$ by a polynomial in $z^\beta$, and that $\varphi_{0,j}$ has no terms divisible by $z^\beta$, so that only the constant terms $\varphi_{0,j}(0)$ give powers of $z^\beta$. We have
\[ \hat \Phi_0 \circ G(z) = G_0(z^\beta) + c_l G_0(z^\beta)^l + \sum_{j >l}\varphi_{0,j}(G(z))G_0(z^\beta)^j\]
and $G_0(z^\beta)^j = (z^\beta)^j + j A (z^\beta)^{j+k}+ \OO((z^\beta)^{j+2k})$, so that the coefficient of $(z^\beta)^{l+k}$ in $\hat \Phi_0 \circ G(z) - G_0(z^\beta)$ is $lAc_l + \varphi_{0,l+k}(0)$. Moreover, by \eqref{util40term},
\[ G_0 \circ \hat \Phi_0(z) = G_0(z^\beta) + G_0'(z^\beta)\big(c_l (z^\beta)^l + \sum_{j >l}\varphi_{0,j}(z)(z^\beta)^j\big) + R(\varphi_0)\]
and $G_0'(z^\beta) = 1 + (k+1)A(z^\beta)^k + \OO((z^\beta)^{2k})$. Since, by iteration assumption, $\varphi_0(z^\beta,z)$ is divisible by $(z^\beta)^l$, and since $G_0^{(p)}$ has order $\geq k+1-p$, then $R(\varphi_0)$ is divisible by $(z^\beta)^{2l+k-1}$. Since $l \geq 2$, $R(\varphi_0)$ has no term $(z^\beta)^{l+k}$, and the coefficient of $(z^\beta)^{l+k}$ in $G_0 \circ \hat \Phi_0(z) - G_0(z^\beta)$ is $(k+1)Ac_l + \varphi_{0,l+k}(0)$. Comparing the coefficients of $(z^\beta)^{l+k}$ in both sides of $\hat \Phi_0 \circ G = G_0 \circ \hat \Phi_0$, we get 
\[lA c_l + \varphi_{0,l+k}(0) = (k+1)Ac_l + \varphi_{0,l+k}(0),\]
that is to say $A(k+1-l)c_{l} = 0$. Since $A\neq 0$ and $l \geq 2 k +1 > k+1$ by assumption (the case $l \leq 2k$ being trivial since $\varphi_{0,l}(z)=0$ for $l\leq 2k$ by assumption), we get $c_{l}=0$, which proves the induction. Finally, $\varphi_{0,l}(z)=0$ for all $l$.

        Hence, $\hat\Phi_0(z)=z^{\beta}$. Let us compute $\hat\Phi_i(z)=z_i+\varphi_i(z)$, $i=1,\ldots, m$. According to \re{equif}, we have 
		$$
		\varphi_i(\lambda z)-\lambda_i\varphi_i(z)=\varphi_{i}(z)\big(a_i(z^\beta)^k + b_i (z^\beta)^{2k}\big)+(\varphi_{i}(\lambda z)-\varphi_{i}(G(z)))=:h_i(z).
		$$
		Furthermore, we have
       \[
			\varphi_i(G(z))-\varphi_i(\lambda  z)= \sum_{l\geq 0}\underbrace{\left(\varphi_{i,l}(G(z))-\varphi_{i,l}(\lambda z)\right)G_0(z^\beta)^l}_{(D\varphi_{i,l}(\lambda z)).(z(z^\beta)^k+\cdots)(z^\beta+\cdots)^l}+\varphi_{i,l}(\lambda z)\underbrace{(G_0(z^\beta)^l-(z^\beta)^l)}_{(z^\beta)^{l-1}.(lA(z^\beta)^{k+1}+\cdots)}\]

    Assume by induction that $\varphi_{i,j}=0$ for $j<l$. Then $h_{i,l}=0$, and by $1$-resonance the solution of $\varphi_{i,l}(\lambda z)-\lambda_i\varphi_{i,l}(z) = 0$ without terms divisible by $z^\beta$ is $\varphi_{i,l}=\varphi_{i,l,E_i}z_i$ for some $\varphi_{i,l,E_i}\in\C$. Recall that $G_i(z) = z_i(\lambda_i + a_i (z^\beta)^k + b_i (z^\beta)^{2k})$, $G(z)^\beta=G_0(z^\beta)$ and $\hat \Phi(z)^\beta = z^\beta$. In the equation $\hat \Phi \circ G = G \circ \hat \Phi$, the monomial $z_i(z^\beta)^{l+k}$ only comes from the term $\varphi_{i,l,E_i} z_i(z^\beta)^l$ (the terms $\varphi_{i,j}(z)(z^\beta)^j$, $j>l$, contribute to $\lambda_i \varphi_{i,j}(z) - \varphi_{i,j}(\lambda z)$ at order $j$, which has no term $z_i$, and higher order terms). Since $G_0(z^\beta)^l = (z^\beta)^l (1+ lA(z^\beta)^k + ...)$, its coefficient is $\lambda_i \varphi_{i,l,E_i}lA$, so that $\varphi_{i,l,E_i}=0$ since $A \neq 0$ and $l\geq 1$. This proves the induction, and $\hat \Phi_i(z) = z_i$ for $1 \leq i \leq m$.
        
		Hence, $\hat\Phi_0(z)=z^{\beta}$ and $\hat\Phi_i(z)=z_i$, $i=1,\dots , m$. 
		For $i\geq m+1$, we use \eqref{equisol} which reads
		$$
		\lambda_i\varphi_i(z)-\varphi_i(\lambda z)=-\varphi_i(z)Q_i(z^\beta)+(\varphi_{i}(G(z))-\varphi_i(\lambda z))=:h_i(z)
		$$
		
       As above, assuming by induction that $\varphi_{i,j}=0$ for $j<l$, we have $h_{i,l}=0$, $i\geq m+1$, so that the (non-resonant, by assumption) solution of $\lambda_i\varphi_{i,l}( z)-\varphi_{i,l}( \lambda z) = h_{i,l}(z)$ is either $\varphi_{i,l}=0$ or $\varphi_{i,l}=\varphi_{i,l,E_i}z_i$ for some $\varphi_{i,l,E_i}\in \C$. If one of the $\varphi_{i,l,E_i}$ is not zero then the right-hand side of the previous equation yields a nonzero term of the form $z_i(z^{\beta})^{l+k}$ (with coefficient $\lambda_i \varphi_{i,l,E_i}lA$ as for $i\leq m$) which cannot be obtained as $(\lambda_i-\lambda^Q)\varphi_{i,Q}z^Q$. Hence $\varphi_{i,l,E_i}=0$ for all $l$, and $\hat \Phi_i(z)=z_i$ for all $m+1 \leq i \leq n$.

        Finally, we have proved that the identity is the only element of the formal centralizer of $\hat G$ of the form given in the proposition.
	\end{proof}

\begin{lemma} \label{centralisateur2}
       Let $\hat T$ be a formal diffeomorphism tangent to the identity which commutes with $\hat G$. Then,
       \[ \hat T(z)^\beta = z^\beta + c_{\hat T}(z^\beta)^{k+1} + \OO((z^\beta)^{k+2}),\]
       with $ c_{\hat T} \in \C$, and where $\OO((z^\beta)^{k+2})$ is a formal power series divisible by $(z^\beta)^{k+2}$. Moreover, if $\hat S$ is another formal diffeomorphism tangent to the identity which commutes with $\hat G$, then $c_{\hat S \circ \hat T} = c_{\hat S} + c_{\hat T}$. Finally, if $c_{\hat T} = 0$ then $\hat T = id$.
   \end{lemma}
\begin{proof}
    As in the proof of Proposition \ref{centralisateur}, we write
    \[ \hat T(z)^\beta = z^\beta + \sum_{l\geq 0} \varphi_{0,l}(z)(z^\beta)^l,\]
    \[ \hat T_i(z) = z_i + \sum_{l\geq 0} \varphi_{i,l}(z)(z^\beta)^l,\]
    where $\varphi_{0,l}$ and $\varphi_{i,l}$ are formal series which do not contain terms divisible by $z^\beta$. We will follow the proof of Proposition \ref{centralisateur} with $\hat T$ instead of $\hat \Phi$. The steps $\varphi_{0,0}=0$ and $\varphi_{0,1}=0$ are the same and only use that $\hat T$ is tangent to the identity. Then, for $l\geq 2$, if $\varphi_{0,j}=0$ for all $j < l$, the induction step gives $\varphi_{0,l}(z)=c_l \in \C$ with $A(k+1-l)c_l=0$. As $A\neq 0$, we get $c_l = 0$ for $2 \leq l\leq k$, so that
    \[\hat T(z)^\beta = z^\beta + c_{k+1}(z^\beta)^{k+1} + \sum_{l \geq k+2}\varphi_{0,l}(z)(z^\beta)^l,\]
    which proves the first part of the Lemma with $c_{\hat T} := c_{k+1}$.
    For the second part of the Lemma, let $u = z^\beta$ and $w := \hat T(z)$. Applying the previous result to $\hat S$ gives
    \[ \hat S(w)^\beta = w^\beta + c_{\hat S}(w^\beta)^{k+1} + \OO((w^\beta)^{k+2}).\]
    The previous result applied to $\hat T$ also gives $w^\beta = u + c_{\hat T}u^{k+1} + \OO(u^{k+2})$. Then $(w^\beta)^{k+1} = u^{k+1}+ \OO(u^{2k+1})$, and $(w^\beta)^l$ is divisible by $u^l$ for all $l$. Hence
    \[((\hat S \circ \hat T)(z))^\beta = u + (c_{\hat S}+c_{\hat T})u^{k+1} + \OO(u^{k+2}).\]
  To complete the proof, suppose now that $c_{\hat T}=0$. The iteration assumption can be applied at any order: for $l\geq k+2$, the equation $A(k+1-l)c_l =0$ gives $c_l=0$. Then $\hat T(z)^\beta = z^\beta$. We can continue to apply the proof of Proposition \ref{centralisateur}: for all $i$, we get $\varphi_{i,l}(z)=\varphi_{i,l,E_i}z_i$ and then $\varphi_{i,l,E_i}=0$ for $l\geq 1$ from equation $\lambda_i \varphi_{i,l,E_i}lA = 0$. Moreover, $\varphi_{i,0,E_i}= 0$ since $\hat T$ is tangent to the identity. Finally, $\hat T = id$.
\end{proof}

		We now apply Borel-Ritt theorem \cite{borel-ritt-ramis} with respect to $z_0$ with parameter $z$. Given a sector $S$ at $0$ in $z_0$, there exists $\Phi_{S}$ holomorphic on $S \times \D_{\rho_\infty}^{n}$, such that $\Phi_{S}=(\Phi_{S,0}(z_0,z),\Phi_{S,1}(z_0,z),\ldots,\Phi_{S,n}(z_0,z))$ is asymptotic to $(\hat\Phi_0(z_0,z),\hat \Phi_1(z_0,z),\ldots,\hat\Phi_n(z_0,z))$ in $z_0$ uniformly on $ z \in \D_{\rho_\infty}^n$. Denoting $\Phi_S = id +  \varphi_S$ with $ \varphi_S = ( \varphi_{S,0}, ...,  \varphi_{S,n})$, for any $\tilde S $ sub-sector of $S$, for all  $N \geq 2k+1$, there exists a constant $C_N>0$ such that, for all $i=0,...,n$, 
		\[\vert \varphi_{S,i}(z_0, z) - \sum_{l= 2k+1}^N  \varphi_{i,l}(z)z_0^l \vert \leq C_{N}\vert z_0\vert ^{N+1}, \quad \forall (z_0,z) \in \tilde S\times \D_{\rho_\infty}^n.\]

		Then $\Phi_S \circ (F_0(z_0,z),F(z_0,z)) =  (G_{\ast,0},G_{\ast}) \circ \Phi_S$ with $ G_{\ast} = G(z_0,z) + K$ (resp. $ G_{\ast,0} := G_{\ast}^{\beta}= G(z_0,z)^{\beta} +  K_0$) where $K=(K_1,...,K_n)$ and $ K_0$ are flat with respect to $z_0$ uniformly on $z$, that is to say, for all $i=0,...,n$
		\[ K_i(z_0,  z) = \OO_{z_0 \in \tilde S, z_0 \rightarrow 0}(\vert z_0 \vert^N), \quad \forall N \in \N\]
		uniformly for all $ z \in \D_{\rho_\infty}^n$.
		
		By restricting $z_0$ to $z^{\beta}$, $ \Phi_S(z^{\beta},z)$ is holomorphic in $\{(z^{\beta},z)\in S \times \D_{\rho_\infty}^{n}\}$ and asymptotic to $\hat\Phi(z^{\beta},z)$ there. Finally, it conjugates $F$ to 
		 $ G_{\ast}(z) = G(z^{\beta},z) + K(z^{\beta},z)$ where 
		$K(z^{\beta},z)$ is flat with respect to $z^{\beta}$ in $S$ uniformly in $z\in \D_{\rho_\infty}^{n}$.

        \begin{rem}
        		In general, polynomials $Q_i$, $i=m+1,\ldots,n$ of a normal form $\hat G$ might be of order $<k$.
        \end{rem}

 \begin{remark}\label{dilation}
        Up to the dilation $u\mapsto \delta u$ with $\delta^k= -\frac{1}{kA}$, which replaces $A$ by $A \delta^k$ and each $a_{i,1}$ by $a_{i,1}\delta^k$, the quantities $\frac{a_{i,1}}{A}$, and then the numbers $\nu_i$ of \eqref{H_1} are unchanged. \emph{We therefore assume from now on that $A=-\frac{1}{k}$}.
    \end{remark}
\section{Sectorial normalization}

In what follows, we will assume that the orders of the polynomials $Q_i$, $i=m+1,\ldots,n$ of a normal form $\hat G$ are equal to $k$ so that \eqref{H_0} is satisfied.

Let us consider the map $\pi:\C^n\rightarrow \C$ defined by $\pi(x)=x^{\beta}$ where $\beta$ is as above.
\subsection{Asymptotic expansion with respect to a monomial}\label{asymptmonom}
We follow and recall concepts and results of \cite[Chapter 2, Section 4]{Ram-Mart2}.

Let $\mathcal{B}$ be the sheaf of rings over the unit circle $S^1$ defined as follows: for each direction $d\in S^1$, an element of  $\mathcal{B}_d$ is a germ at $0\in \C\times\C^n$ of a function $f:V_d(\theta,r)\times \overline{\D_\rho^n}\rightarrow \C$ where $V_d$ is a closed sector at $0$ around direction $d$, of opening $2\theta$ and of radius $r$, $\ov{\D_\rho^n}$ is a closed polydisc at the origin of radius $\rho$. $f$ is analytic in the interior of $V_d(\theta,r)\times \overline{\D_\rho^n}$ and Whitney-smooth on the closed set $V_d(\theta,r)\times \overline{\D_\rho^n}$. Then the infinite jet of $f$ along $\{0\}\times \C^n$ is a formal power series $\hat f=\sum_{j\geq 0}f_j(x)u^j$ where all functions $f_j(x)$ are defined in the same polydisc of $\C^n$. The set of these infinite jets is denoted $\hat{\mathcal{B}}$. Let $\mathcal{B}^{\infty}$ be the subsheaf of $\mathcal{B}$ whose elements are flat along  $\{0\}\times \C^n$, that is with a zero infinite jet along $\{0\}\times \C^n$.

On the other hand, let $B_{\pi}$ be the sheaf of rings over the unit circle $S^1$ defined as follows: for each direction $d\in S^1$, an element of  ${B}_{\pi,d}$ is a germ at $0\in \C^n$ of a function $f:\tilde V_d(\theta,r)\cap \overline{\D_\rho^n}\rightarrow \C$ where $\tilde V_d(\theta,r)=\pi^{-1}(V_d(\theta,r))$ with $V_d(\theta,r)$ a closed sector at $0$ around direction $d$, of opening $2\theta$ and of radius $r$, $\ov{\D_\rho^n}$ is a closed polydisc at the origin of radius $\rho$. $f$ is analytic in the interior of $\tilde V_d(\theta,r)\cap \overline{\D_\rho^n}$ and Whitney-smooth on the closed set $\tilde V_d(\theta,r)\cap \overline{\D_\rho^n}$. 
Recall 
\begin{equation} \label{defE} E=\{x\in\C^n,\,x^{\beta}=0\}=\pi^{-1}(0).\end{equation} The set of infinite jets of such a $f$ along $E\cap \overline{\D_\rho^n}$ is denoted by $\hat B$. The latter is the formal completion of $\C\{x\}$ along $E$. $B_{\pi}^{\infty}$ denotes the sheaf of flat germs along $E\cap \overline{\D_\rho^n}$, these are elements of $B_{\pi}$, the infinite jet of which along $E\cap \overline{\D_\rho^n}$ is zero.

Let us define 
$$
\Sigma:=\{(u,x)\in \C^{n+1},\,u=x^{\beta}\}.
$$
\begin{lemma}[{\cite[Lemma 4.7]{Ram-Mart2}}]\label{lemRM}
Let $f$ be a flat section of $B_{\pi}$ over a closed sector $V$, considered as defined on a neighborhood of $0$ over $\Sigma\cap (V\times \C^n)$. Then $f$ admits a flat extension $\tilde f$ in $\mathcal{B}$ over $V$.
\end{lemma}
In other words, if $f$ is flat with respect to $x^{\beta}$ in some sector $S$, uniformly in $x$ in a polydisc, then there exists a function $\tilde f(u,x)$ flat with respect to $u$ in $S$, uniformly in $x$ in a polydisc, whose restriction to $\Sigma$ is $f$.

\subsection{Flat perturbations of 1-resonant normal form}
Let $k\geq 1$, $\lambda = (\lambda_1,...,\lambda_n)$ $1$-resonant with respect to $\beta$ as defined in \ref{1-res}, and recall that $\vert \lambda_i \vert = 1$ for all $i=1,...,n$. 

\begin{equation}\label{casgeneral}
			G_i(z)=\left\{
			\begin{array}{l}
				 z_i(\lambda_i+a_i z^{k \beta} + b_i  z^{2k\beta }), \quad i =1,...,m \\
				z_i(\lambda_i + P_i(z^{ \beta})) \quad i = m+1,...,n
			\end{array}
			\right.
		\end{equation}
 where
 \[ z^{k\beta} = z_1^{k\beta_1}...z_m^{k\beta_m},\]
and with $\ord_0 P_i =k$ so that condition \eqref{H_0} is satisfied.

\textbf{Monomial coordinate:} We have
 \begin{align*} 
G_1(z)^{\beta_1}\cdots G_m(z)^{\beta_m} &= \prod_{i=1}^m \big(z_i(\lambda_i+a_i z^{k \beta} + b_i  z^{2k\beta }) \big)^{\beta_i} \\
 & = z^\beta \prod_{i=1}^m \lambda_i^{\beta_i} \prod_{i=1}^m \big(1+ \frac{a_i}{\lambda_i} z^{k \beta} + \frac{b_i }{\lambda_i} z^{2k\beta } \big)^{\beta_i}.
 \end{align*}
Setting $u=z^\beta$, let us define
 \beq \label{defG0u}
G_0(u) := u \prod_{i=1}^m \big(1+ \frac{a_i}{\lambda_i} u^{k} + \frac{b_i }{\lambda_i} u^{2k} \big)^{\beta_i} = u(1+A u^k + Bu^{2k}) + uP(u^k)
\eeq
where $P(u^k) = \sum_{j=3}^{2\vert \beta \vert}c_ju^{kj}$ contains the other terms of the product, and $A= \sum_{j=1}^{m} \frac{\beta_j a_j}{\lambda_j}$.

Let $\alpha$ and $R$ define a sector at infinity $\Delta_+(\alpha,R)$. The change of variable $U=u^{-k}$ sends the following $k$ bounded sectors at the origin, centered on the directions $e^{2i\pi \frac{k_0}{k}}$, with $k_0 = 0,...,k-1$, into $\Delta_+(\alpha,R)$, where $\alpha$ and $\epsilon$ are as in Theorem \ref{main-sect}:
\begin{equation}\label{delta}
\delta_+^{[k_0]}(\alpha,r):=\left\{u\in \C, |\arg u- \frac{2\pi k_0}{k}|<\frac{\alpha}{k}-\epsilon,\;|u|<r\right\}.
\end{equation}
Notice that $\Delta_+(\alpha,R)$ is a sector centered at $R$, so that $U \in \Delta_+(\alpha,R)$ gives a condition on $\arg(U-R)$, while $\delta_+^{[k_0]}(\alpha,r)$ is centered at $0$ so implies a condition on $\arg U$. The correction $-\epsilon$ ensures that $u \in \delta_+^{[k_0]}(\alpha,r)$ does not give a $U$ outside $\Delta_+(\alpha,R)$. Indeed, if $r$ is small enough so that $r^k \leq \frac{\sin (k\epsilon)}{R}$, then $u \in \delta_+^{[k_0]}(\alpha,r)$ satisfies $\vert \arg U \vert < \alpha - k\epsilon$, and $\vert U \vert \geq r^{-k} \geq \frac{R}{\sin (k\epsilon)}$, so that $\vert \arg (U-R)-\arg U\vert \leq \arcsin \left( \frac{R}{\vert U \vert} \right) \leq k \epsilon$, hence $\vert \arg (U-R)\vert < \alpha$. 

$\delta_-^{[k_0]}(\alpha,r)$ is defined in the same way, centered on the directions $e^{i \pi \frac{2 k_0 +1}{k} }$, so that $U= u^{-k}$ is sent into $\Delta_-(\alpha,R)$.

After having applied a fibered Borel-Ritt lemma with respect to the monomial $z^{\beta}$, we can consider a flat perturbation $G_*=G+K$ with respect to $u=z^\beta$ over some sector at the origin $\delta_+^{[k_0]}(\alpha,r)$, uniformly in $z=(z_1,..., z_n)$ in $\bar \D_\rho^{n}$. Hence, we consider $K = (K_1,...,K_n)\in (B_{\pi}^{\infty})^n$. Let $\tilde K$ be an extension obtained from \rl{lemRM}.  Let us define $G_{\ast,0}(z) := \prod_{i=1}^{m}(G_{\ast,i}(z))^{\beta_i} = \prod_{i=1}^{m}(G_i(z)+K_i(z))^{\beta_i}$ as well as $ K_0(z) := G_{\ast,0}(z)-G_0(z^{\beta})$.

By construction, $K_0$ is flat with respect to $z^\beta$ uniformly in $z$. Let $\tilde G_{\ast,i} := G_i + \tilde K_i$ for $i=1,...,n$ and define $\tilde G_{\ast,0} := \tilde G_\ast^{\beta} = G_0(u)+\tilde K_0$ where $\tilde K_0 \in \B^\infty$ is flat in $u$ uniformly in $z$ and is an extension of $K_0$. 
 
Let, for all $(u,z) \in  \delta_{+}^{[k_0]}(\alpha,r) \times\D_\rho^n$ 
    \[\GG(u,z) := (\GG_0(u,z),\GG_1(u,z)):=(\tilde G_{\ast,0}(u,z),\tilde G_\ast (u,z))\]
with $\tilde G_\ast = G + \tilde K$ and $\tilde G_{\ast,0} := \tilde G_\ast^{\beta}$ which can be written $\tilde G_{\ast,0}= G_0 + \tilde K_0$, where $\tilde K_0 \in \B^\infty$ and $\tilde K = (\tilde K_1,..., \tilde K_n) \in (\B^\infty)^n$ (that is to say, each $\tilde K_i$ is flat in $u$ uniformly in $z$). Letting for all $i=1,...,n$

        \begin{equation}\label{defLambda}
			\Lambda_i(u,z)= \lambda_i+a_{i,1}u^{k} + \varepsilon_i(u,z)
		\end{equation}
(with $a_{i,1}:=a_i$ for $i=1,...,m$, and $a_{i,1}$ is the coefficient of $u^k$ in $P_i$ for $i=m+1,...,n$ as we supposed that $P_i$ has order $k$), with $\varepsilon_i(u,z) = \OO(u^{k+1})$, so that
\[ \tilde G_{\ast,i}(u,z) = z_i \Lambda_i(u,z) + \tilde K_i(u,z).\]

For $p\geq0$, we denote the iterates of $\GG$:
\beq \label{defproj}
\GG^{(p)}(u,z) = (\GG^{(p)}_0,\GG^{(p)}_1)=:(u_p,z_p)
\eeq
We also denote
	\[\tilde G_{\ast,0}^{(p)} := (\GG^{(p)})_0, \qquad \tilde G_{\ast}^{(p)} := (\GG^{(p)})_1.\]

\subsection{Stability of the sector under iteration}

Recall condition \eqref{H_1}:
\[    \nu_i  := \RE(a_{i,1}(\lambda_i A)^{-1}) >0, \quad \forall i = 1,...,n. \]
\begin{lemma}\label{stability} Under assumptions \eqref{H_0} (with $P_i$ here for $Q_i$) and \eqref{H_1}, for $R$ large enough and $\frac{\pi}{2}<\alpha < \frac{3\pi}{4}$, there exists $0<\rho_0<\rho$ such that 
in the variable $U = u^{-k}$ (that is, for $u$ such that $u^{-k}\in \Delta_+(\alpha,R)$), for all $(U,z) \in \Delta_+(\alpha,R)\times \D_{\rho_0}^n$ and all $l\geq0$,
\[ \GG^{(l)}(U,z) \in \Delta_+(\alpha,R) \times \D_{\rho}^n.\]
In particular, by \eqref{delta}, this holds for all $(u,z) \in \delta_+^{[k_0]}(\alpha,r)\times \D_{\rho_0}^n$.
\end{lemma}

\begin{proof}
In the proof, we write $U= u^{-k}$ and work with $U\in \Delta_+(\alpha,R)$, which includes the case $u \in \delta_+^{[k_0]}(\alpha,r)$ by \eqref{delta}.  For all $i=1,...,n$, writing $\tilde G_{\ast,i}(u,z) = z_i \Lambda_i(u,z) + \tilde K_i(u,z)$, we obtain that for all $p \geq 0$,
\begin{equation}\label{defGiast}
  \tilde G_{\ast,i}^{(p)} (u,z)=  z_i \prod_{l=0}^{p-1}\Lambda_i(\GG^{(l)}(u,z)) + \sum _{j=0}^{p-1} \tilde K_i(\GG^{(j)}(u,z))\prod_{l=j+1}^{p-1}\Lambda_i(\GG^{(l)}(u,z))
\end{equation}
for all $(U,z) \in \Delta_+(\alpha,R)\times \D_{\rho}^n$.

We have to prove that, for all $l\geq 0$, we have $U_l \in \Delta_+(\alpha,R)$ and $z_l \in \D_\rho^n$. We will proceed by induction on $l$.

If $l=0$, the result is true by definition of $(U_0,z_0) = (U,z) \in \Delta_+(\alpha,R) \times \D_{\rho_0}^n\subset \Delta_+(\alpha,R) \times \D_{\rho}^n$. Suppose that for all $l \leq p-1$, $(U_l,z_l)\in \Delta_+(\alpha,R) \times \D_\rho^n$.

\textbf{Estimates of the flat terms:} Since by induction assumption we have $z_l \in \D_\rho^n$, we can write
\begin{equation}\label{ul}u_{l+1} = u_l + Au_l^{k+1} + \OO(u_l^{2k+1})\end{equation}
(the flat term is "absorbed" by the big $\OO$). Recall that $A=-\frac{1}{k}$ (see Remark \ref{dilation}).

From \eqref{ul} 
\begin{align*}
    U_{l+1}=u_{l+1}^{-k}&  = u_l^{-k}(1-\frac{1}{k}u_l^{k}+\OO(u_l^{2k}))^{-k} \\
    & = u_l^{-k}(1+u_l^{k}+\OO(u_l^{2k})) \\
    & = u_l^{-k} + 1 + \OO(u_l^{k}) \\
    & = U_l + 1 + \OO(U_l^{-1})
\end{align*}

We would like that for all $l\geq 0$, $U_l \in \Delta_+(\alpha,R) \Rightarrow U_{l+1} \in \Delta_+(\alpha,R)$ and
\begin{equation} \label{estimO}
    \vert \OO(\vert U_l \vert^{-1})\vert \leq \frac{1}{4}.
\end{equation}

\begin{figure} 
    \begin{tikzpicture}[x=0.75pt,y=0.75pt,yscale=-1,xscale=1]

\draw [draw opacity=0][fill={rgb, 255:red, 80; green, 227; blue, 194 }, fill opacity=0.29] 
  (390,150) -- (420,150) arc (0:-142.43:30) -- cycle;
\draw (420,150) arc (0:-142.43:30);
\draw [shift={(366.73,131.05)}, rotate = 320.77] [fill={rgb, 255:red, 0; green, 0; blue, 0}][line width=0.08] [draw opacity=0] (10.72,-5.15) -- (0,0) -- (10.72,5.15) -- (7.12,0) -- cycle;

\draw [draw opacity=0][fill={rgb, 255:red, 189; green, 16; blue, 224 }, fill opacity=0.32] 
  (390,150) -- (366.56,131.27) arc (-142.43:-180:30) -- cycle;

\draw (366.56,131.27) arc (-142.43:-180:30);
\draw [shift={(360,149.63)}, rotate = 282.15] [fill={rgb, 255:red, 0; green, 0; blue, 0}][line width=0.08] [draw opacity=0] (8.93,-4.29) -- (0,0) -- (8.93,4.29) -- cycle ;

\draw (160,150) -- (467,150) ;
\draw [shift={(470,150)}, rotate = 180] [fill={rgb, 255:red, 0; green, 0; blue, 0}][line width=0.08] [draw opacity=0] (8.93,-4.29) -- (0,0) -- (8.93,4.29) -- cycle ;

\draw (300,40) -- (300,150) -- (300,250) ;
\draw [shift={(300,40)}, rotate = 90] [fill={rgb, 255:red, 0; green, 0; blue, 0}][line width=0.08] [draw opacity=0] (8.93,-4.29) -- (0,0) -- (8.93,4.29) -- cycle ;

\draw (260,50) -- (339.02,110.79) -- (390,150) ;
\draw [shift={(390,150)}, rotate = 37.57] [color={rgb, 255:red, 0; green, 0; blue, 0}][fill={rgb, 255:red, 0; green, 0; blue, 0}][line width=0.75] (0, 0) circle [x radius= 3.35, y radius= 3.35] ;

\draw (260,250) -- (390,150) ;

\draw [fill={rgb, 255:red, 255; green, 255; blue, 255}, fill opacity=1] (334.49,107.2) -- (343.56,114.37) -- (336.39,123.44) -- (327.32,116.27) -- cycle ;
\draw (334.49,107.2) -- (299.02,150.79) ;
\draw [shift={(299.02,150.79)}, rotate = 129.14] [color={rgb, 255:red, 0; green, 0; blue, 0}][fill={rgb, 255:red, 0; green, 0; blue, 0}][line width=0.75] (0, 0) circle [x radius= 3.35, y radius= 3.35] ;
\draw [shift={(334.49,107.2)}, rotate = 129.14] [color={rgb, 255:red, 0; green, 0; blue, 0}][fill={rgb, 255:red, 0; green, 0; blue, 0}][line width=0.75] (0, 0) circle [x radius= 3.35, y radius= 3.35] ;

\draw (231.5,107) .. controls (265.63,100.18) and (287.87,108.56) .. (311.67,117.33) ;
\draw [shift={(313.5,118)}, rotate = 200.17] [fill={rgb, 255:red, 0; green, 0; blue, 0}][line width=0.08] [draw opacity=0] (12,-3) -- (0,0) -- (12,3) -- cycle ;

\draw (401,102.4) node [anchor=north west][inner sep=0.75pt] {$\alpha$};
\draw (386,162.4) node [anchor=north west][inner sep=0.75pt] {$R$};
\draw (283,152.4) node [anchor=north west][inner sep=0.75pt] {$0$};
\draw (328.39,133.84) node [anchor=north west][inner sep=0.75pt] [font=\scriptsize] {$\pi -\alpha$};
\draw (127,110.4) node [anchor=north west][inner sep=0.75pt] {$|U_{min}| = R\sin \alpha$};

\end{tikzpicture}
    \caption{Definition of minimal distance to the origin}
    \label{dessin1}
\end{figure}

Given $l \geq 0$, notice that \eqref{estimO} holds if and only if $\frac{C}{\vert U_l\vert} \leq \frac{1}{4}$, that is to say, if and only if $4C \leq \vert U_l\vert$. But if $ U\in \Delta_+(\alpha,R)$, then $\vert U \vert \geq R \sin \alpha$ (see Figure \ref{dessin1}), so condition \eqref{estimO} is satisfied if $R \geq \frac{4C}{\sin \alpha}$. 
\begin{figure} 
    \tikzset{every picture/.style={line width=0.75pt}} 

\begin{tikzpicture}[x=0.75pt,y=0.75pt,yscale=-1,xscale=1]

\draw [draw opacity=0][fill={rgb, 255:red, 80; green, 227; blue, 194 }, fill opacity=0.29] 
  (404,207.24) -- (434,207.24) arc (0:-153.53:30) -- cycle;

\draw (378.81,190.93) .. controls (381.57,186.66) and (385.47,183.02) .. (390.32,180.54) .. controls (405.07,172.98) and (423.15,178.81) .. (430.7,193.56) .. controls (432.95,197.95) and (434.01,202.63) .. (434.01,207.24) ;  \draw [shift={(377.31,193.54)}, rotate = 308.66] [fill={rgb, 255:red, 0; green, 0; blue, 0 }  ][line width=0.08]  [draw opacity=0] (10.72,-5.15) -- (0,0) -- (10.72,5.15) -- (7.12,0) -- cycle    ;

\draw    (200,205.57) -- (481,207.22) ;
\draw [shift={(484,207.24)}, rotate = 180.34] [fill={rgb, 255:red, 0; green, 0; blue, 0 }  ][line width=0.08]  [draw opacity=0] (8.93,-4.29) -- (0,0) -- (8.93,4.29) -- cycle    ;
\draw    (200,105.57) -- (240,125.57) -- (404,207.24) ;
\draw [shift={(404,207.24)}, rotate = 26.47] [color={rgb, 255:red, 0; green, 0; blue, 0 }  ][fill={rgb, 255:red, 0; green, 0; blue, 0 }  ][line width=0.75]      (0, 0) circle [x radius= 3.35, y radius= 3.35]   ;
\draw    (340,245.57) -- (404,207.24) ;
\draw [shift={(404,207.24)}, rotate = 329.08] [color={rgb, 255:red, 0; green, 0; blue, 0 }  ][fill={rgb, 255:red, 0; green, 0; blue, 0 }  ][line width=0.75]      (0, 0) circle [x radius= 3.35, y radius= 3.35]   ;
\draw  [dash pattern={on 4.5pt off 4.5pt}]  (240,125.57) -- (290,125.57) ;
\draw [shift={(290,125.57)}, rotate = 0] [color={rgb, 255:red, 0; green, 0; blue, 0 }  ][fill={rgb, 255:red, 0; green, 0; blue, 0 }  ][line width=0.75]      (0, 0) circle [x radius= 3.35, y radius= 3.35]   ;
\draw  [dash pattern={on 4.5pt off 4.5pt}]  (280,145.57) -- (290,125.57) ;
\draw  [fill={rgb, 255:red, 255; green, 255; blue, 255 }  ,fill opacity=1 ] (277.88,139.37) -- (281.71,141.5) -- (279.57,145.33) -- (275.75,143.2) -- cycle ;
\draw  [fill={rgb, 255:red, 248; green, 231; blue, 28 }  ,fill opacity=0.35 ] (263.26,121.44) .. controls (263.26,107.15) and (274.84,95.57) .. (289.13,95.57) .. controls (303.42,95.57) and (315,107.15) .. (315,121.44) .. controls (315,135.72) and (303.42,147.3) .. (289.13,147.3) .. controls (274.84,147.3) and (263.26,135.72) .. (263.26,121.44) -- cycle ;
\draw    (317.81,110.91) .. controls (344.58,102.59) and (331.97,101.74) .. (360,105.57) ;
\draw [shift={(315.67,111.57)}, rotate = 343.21] [fill={rgb, 255:red, 0; green, 0; blue, 0 }  ][line width=0.08]  [draw opacity=0] (12,-3) -- (0,0) -- (12,3) -- cycle    ;

\draw [draw opacity=0][fill={rgb, 255:red, 189; green, 16; blue, 224 }, fill opacity=0.32] 
  (240,125.57) -- (260,125.57) arc (0:26.47:20) -- cycle;

\draw (259.75,128.74) .. controls (259.39,130.97) and (258.67,133.08) .. (257.64,135) ;  \draw [shift={(260,125.74)}, rotate = 107.69] [fill={rgb, 255:red, 0; green, 0; blue, 0 }  ][line width=0.08]  [draw opacity=0] (3.57,-1.72) -- (0,0) -- (3.57,1.72) -- cycle    ;

\draw    (249.51,137.7) .. controls (246.23,152.23) and (246.4,154.55) .. (250,165.57) ;
\draw [shift={(250,135.57)}, rotate = 102.91] [fill={rgb, 255:red, 0; green, 0; blue, 0 }  ][line width=0.08]  [draw opacity=0] (12,-3) -- (0,0) -- (12,3) -- cycle    ;

\draw (415,159.64) node [anchor=north west][inner sep=0.75pt]    {$\alpha $};
\draw (400,219.64) node [anchor=north west][inner sep=0.75pt]    {$R$};
\draw (231,107.97) node [anchor=north west][inner sep=0.75pt]  [font=\footnotesize]  {$U_{l}$};
\draw (271,107.97) node [anchor=north west][inner sep=0.75pt]  [font=\footnotesize]  {$U_{l} +1$};
\draw (288,132.97) node [anchor=north west][inner sep=0.75pt]  [font=\footnotesize]  {$d$};
\draw (371,97.97) node [anchor=north west][inner sep=0.75pt]  [font=\footnotesize]  {$ \begin{array}{l}
U_{l}{}_{+}{}_{1} \in D\left( U_{l} +1,\frac{1}{4}\right) \subset \Delta _{+}( \alpha ,R)\\
\Rightarrow d=\ \sin \alpha  >\frac{1}{4}
\end{array}$};
\draw (236,167.97) node [anchor=north west][inner sep=0.75pt]  [font=\footnotesize]  {$\pi -\alpha $};

\end{tikzpicture}
    \caption{Condition of stability}
    \label{dessin2}
\end{figure}
Then the critical situation for $U_l$ is to be on the boundary of the sector, and $U_{l+1}$ is in the disc of center $U_l+1$ and of radius $d \leq \frac{1}{4}$ (see Figure \ref{dessin2}). Then $U_{l+1} \in \Delta_+(\alpha,R)$ if $\sin(\pi-\alpha) = \sin \alpha > \frac{1}{4}$.

Let $p(U) = \min (l \geq 0, \RE(U_l)> 0)$. In what follows, we shall study the orbit of $U = U_0 \in \Delta_+(\alpha,R)$ on the upper-half plane ($\Im U >0$). The case $\Im U < 0$ is similar. From \eqref{estimO} we have that for all $l\geq 0$,
\begin{equation}\label{estimreelle}
\RE U_l +\frac{3}{4}\leq \RE U_{l+1} \leq \RE U_l + \frac{5}{4},
\end{equation}
\begin{equation}\label{estimimaginaire}
\Im U_l -\frac{1}{4}\leq \Im U_{l+1} \leq \Im U_l + \frac{1}{4}.
\end{equation}
If $\RE U > 0$ then $p(U)=0$. If $\RE U \leq 0$, then
\[ \RE U_{p(U)} \geq \RE U + \frac{3p(U)}{4} \Rightarrow p(U) \leq \frac{4}{3}(\RE U_{p(U)} - \RE U).\]

\begin{figure} 
    \tikzset{every picture/.style={line width=0.75pt}} 

\begin{tikzpicture}[x=0.75pt,y=0.75pt,yscale=-1,xscale=1]

\draw [draw opacity=0][fill={rgb, 255:red, 80; green, 227; blue, 194 }, fill opacity=0.29] 
  (460,191.67) -- (490.01,191.67) .. controls (490.01,187.06) and (488.95,182.38) .. (486.7,178) .. controls (479.15,163.25) and (461.07,157.42) .. (446.32,164.97) .. controls (439.82,168.3) and (435.06,173.67) .. (432.4,179.91) -- cycle;

\draw    (433.74,177.16) .. controls (436.53,172.11) and (440.8,167.8) .. (446.32,164.97) .. controls (461.07,157.42) and (479.15,163.25) .. (486.7,178) .. controls (488.95,182.38) and (490.01,191.67) .. (490.01,191.67) ;  \draw [shift={(432.4,179.91)}, rotate = 304.52] [fill={rgb, 255:red, 0; green, 0; blue, 0 }  ][line width=0.08]  [draw opacity=0] (10.72,-5.15) -- (0,0) -- (10.72,5.15) -- (7.12,0) -- cycle    ;

\draw    (190,190) -- (507,190) ;
\draw [shift={(510,190)}, rotate = 180] [fill={rgb, 255:red, 0; green, 0; blue, 0 }  ][line width=0.08]  [draw opacity=0] (8.93,-4.29) -- (0,0) -- (8.93,4.29) -- cycle    ;
\draw    (190,80) -- (460,191.67) ;
\draw    (426,210) -- (460,191.67) ;

\draw [shift={(460,191.67)}, rotate = 331.67] [color={rgb, 255:red, 0; green, 0; blue, 0 }  ][fill={rgb, 255:red, 0; green, 0; blue, 0 }  ][line width=0.75]      (0, 0) circle [x radius= 3.35, y radius= 3.35]   ;

\draw    (390,63) -- (390,190) -- (390,210) ;
\draw [shift={(390,60)}, rotate = 90] [fill={rgb, 255:red, 0; green, 0; blue, 0 }  ][line width=0.08]  [draw opacity=0] (8.93,-4.29) -- (0,0) -- (8.93,4.29) -- cycle    ;
\draw    (250,90) ;
\draw [shift={(250,90)}, rotate = 0] [color={rgb, 255:red, 0; green, 0; blue, 0 }  ][fill={rgb, 255:red, 0; green, 0; blue, 0 }  ][line width=0.75]      (0, 0) circle [x radius= 3.35, y radius= 3.35]   ;
\draw  [dash pattern={on 4.5pt off 4.5pt}]  (250,90) -- (390,90) ;
\draw  [dash pattern={on 4.5pt off 4.5pt}]  (250,90) -- (250,190) ;
\draw [color={rgb, 255:red, 144; green, 19; blue, 254 }  ,draw opacity=1 ]   (250,90) -- (390,190) ;

\draw [draw opacity=0][fill={rgb, 255:red, 189; green, 16; blue, 224 }, fill opacity=0.21] 
  (390,190) -- (389.95,156.61) .. controls (383.94,154.05) and (377.89,153.79) .. (372.79,156.4) .. controls (367.82,158.95) and (364.53,163.83) .. (363.04,169.99) -- cycle;

\draw [color={rgb, 255:red, 144; green, 19; blue, 254 }  ,draw opacity=1 ]   (363.93,167.05) .. controls (365.66,162.28) and (368.64,158.53) .. (372.79,156.4) .. controls (377.89,153.79) and (383.94,154.05) .. (389.95,156.61) ;  \draw [shift={(363.04,169.99)}, rotate = 297.12] [fill={rgb, 255:red, 144; green, 19; blue, 254 }  ,fill opacity=1 ][line width=0.08]  [draw opacity=0] (8.93,-4.29) -- (0,0) -- (8.93,4.29) -- cycle    ;

\draw (471,142.4) node [anchor=north west][inner sep=0.75pt]    {$\alpha $};
\draw (456,204.07) node [anchor=north west][inner sep=0.75pt]    {$R$};
\draw (234,62.4) node [anchor=north west][inner sep=0.75pt]    {$U$};
\draw (301,72.4) node [anchor=north west][inner sep=0.75pt]  [font=\small]  {$-\Re U$};
\draw (221,142.4) node [anchor=north west][inner sep=0.75pt]  [font=\small]  {$\Im U$};
\draw (371,132.4) node [anchor=north west][inner sep=0.75pt]  [color={rgb, 255:red, 189; green, 16; blue, 224 }  ,opacity=1 ]    {$\gamma $};
\draw (301,152.4) node [anchor=north west][inner sep=0.75pt]  [font=\small,color={rgb, 255:red, 189; green, 16; blue, 224 }  ,opacity=1 ]    {$|U|$};
\draw (420,82.4) node [anchor=north west][inner sep=0.75pt]    {$ \begin{array}{l}
-\Re U=\vert U \vert \sin \gamma \leq \vert U \vert \cos(\pi-\alpha)\\
\ \  
\end{array}$};
\draw (428,102.4) node [anchor=north west][inner sep=0.75pt]    {$\ \Im U = \vert U \vert \cos \gamma \geq \vert U \vert \sin \alpha$};
\draw (371,192.4) node [anchor=north west][inner sep=0.75pt]    {$0$};

\end{tikzpicture}
    \caption{Estimates of real and imaginary parts}
    \label{dessin3}
\end{figure}

Notice that by definition of $p(U)$, we have $\RE U_{p(U)-1} \leq 0$ so by \eqref{estimreelle}, $\RE U_{p(U)} \leq \frac{5}{4}$. Moreover, since $U \in \Delta_+(\alpha,R)$, denoting by $\gamma$ the angle between $U$ and the positive imaginary axis, we have $-\RE U = \vert U \vert \sin \gamma \leq \vert U \vert \cos(\pi-\alpha)$ (see Figure \ref{dessin3}), then 
\[ p(U) \leq \frac{4}{3}(\frac{5}{4} + \vert U \vert\cos(\pi-\alpha))\]
and finally for $R$ large enough,
\begin{equation}\label{p(U)}
    p(U) \leq \max (0, 4 c_\alpha\vert U \vert )
\end{equation}
where $c_\alpha =  \cos(\pi-\alpha)$.

Let $N$ be  a large integer to be defined below, we would like to estimate 
\[\sum_{l=0}^{p-1} \vert u_l\vert^{N} = \sum_{l=0}^{p-1} \vert U_l\vert^{-\frac{N}{k}} \leq \sum_{l=0}^{p(U)} \vert U_l\vert^{-\frac{N}{k}} + \vert U_{p(U)+1}\vert^{-\frac{N}{k}} +\sum_{l=p(U)+2}^{\infty} \vert U_l\vert^{-\frac{N}{k}}.\]

We have, for all $l \leq p(U)+1$,
\[ \vert U_l \vert \geq  \vert \Im U_l\vert \geq \Im U - \frac{l}{4} \geq \Im U - \frac{p(U)+1}{4}. \]
We have $\Im U = \vert U \vert \cos \gamma$ and since $0 < \gamma < \alpha - \frac{\pi}{2}$ (see Figure \ref{dessin3}, as we are in the case $\RE U \leq 0$) we get $\Im U \geq \vert U \vert \sin \alpha$ and then 
\[ \vert U_l \vert \geq \vert U \vert \sin\alpha - c_\alpha \vert U \vert  \geq c'_\alpha\vert U \vert\]
with $c'_\alpha = \sin \alpha - \cos (\pi - \alpha) > 0$ if and only if $\sin \alpha + \cos\alpha > 0$ which is true if $\alpha < \frac{3\pi}{4}$, then

\[\sum_{l=0}^{p(U)} \vert U_l\vert^{-\frac{N}{k}} + \vert U_{p(U)+1}\vert^{-\frac{N}{k}} \leq (p(U)+2) (c_\alpha' \vert U\vert)^{-\frac{N}{k}} \leq (4c_\alpha \vert U \vert +2)(c_\alpha' \vert U\vert)^{-\frac{N}{k}} \leq C_1\]
if $N \geq 2k$.

On the other hand, from \eqref{estimreelle} we have for all $l \geq p(U)+1$,
\[ \vert U_l \vert \geq \RE U_l \geq \RE U_{p(U)} + \frac{3}{4}(l-p(U)) \geq \frac{3}{4}(l-p(U))\]
since $\RE U_{p(U)} > 0$. Then
\[ \sum_{l=p(U)+2}^{\infty} \vert U_l\vert^{-\frac{N}{k}} \leq \sum_{l=p(U)+2}^\infty (\frac{4}{3})^{\frac{N}{k}} \frac{1}{(l-p(U))^{\frac{N}{k}}} = \sum_{l=2}^{\infty} (\frac{4}{3})^{\frac{N}{k}} \frac{1}{l^{\frac{N}{k}}} \leq C_2\]
if $N \geq 2k$. Finally, letting $C_3=C_1+C_2$ which depends only on $\alpha, R, N$ and $k$, we obtain
\begin{equation}\label{estimul}
\sum_{l=0}^{p-1} \vert u_l\vert^{N} \leq C_3.
\end{equation}

\textbf{Norm of the product of $\vert \Lambda_i \vert$: }by induction assumption, for $l\leq p-1$, we have $z_l \in \D_\rho^n$ and then
\[ \Lambda_i(u_l,z_l) = \lambda_i + a_{i,1}u_l^{k} + \OO(u_l^{k+1}).\]
 Then for all $l\leq p-1$,

\begin{align*}
    \ln \vert \Lambda_i(u_l,z_l) \vert & = \ln \vert \lambda_i + a_{i,1}u_l^{k} + \OO(u_l^{k+1}) \vert \\
    & = \ln (\vert \lambda_i \vert\vert1+\frac{a_{i,1}}{\lambda_i}u_l^k + \OO(u_l^{k+1})\vert) \\
    & = \frac{1}{2}\ln \big( 1+(\overline{a_{i,1}\lambda_i^{-1}u_l^k} + a_{i,1}\lambda_i^{-1}u_l^k) + \OO(u_l^{k+1}) \big) \\
    & = \frac{1}{2}\ln \big(  1 + 2\RE(a_{i,1}\lambda_i^{-1}u_l^k) + \OO(u_l^{k+1})\big) \\
    & = \frac{1}{2}( 2\RE(a_{i,1}\lambda_i^{-1}u_l^k) +  \OO(u_l^{k+1}))\\
    & = \RE(a_{i,1}\lambda_i^{-1}u_l^k) +  \OO(u_l^{k+1})
\end{align*}

By Leau-Fatou theorem, on the attracting petal, 
which corresponds to $\Delta_+(\alpha,R)$ in the variable $U$ (see for example \cite[Lemma 10.1]{milnor}), we have
\[ u_l^k \underset{l\rightarrow +\infty}{\sim} -\frac{1}{kAl},\]
then for all $i=1,...,n$,
\[ \ln\vert \Lambda_i(u_l,z_l) \vert= -\frac{1}{l}\RE(a_{i,1}\lambda_i^{-1} A^{-1}k^{-1}) + \OO(l^{-1-\frac{1}{k}}) = -\frac{1}{kl} \nu_i  + \OO(l^{-1-\frac{1}{k}}) \]
where 
\[ \nu_i  := \RE(a_{i,1}(\lambda_i A)^{-1}).\]

We have to estimate $\prod_{l=1}^{p-1}\Lambda_i(u_l,z_l)$ for $i=1,...,n$, since $\sum_{l=1}^{p-1}\frac{1}{l} = \ln(p) + \OO(1)$, and since $\sum_{l=1}^{p-1}\OO(l^{-1-\frac{1}{k}})$ converges absolutely for all $k\geq 1$, we have
    \[ \sum_{l=1}^{p-1}\ln \vert \Lambda_i(u_l,z_l)\vert = -\frac{1}{k}\nu_i \ln(p) + \OO(1),\]
    and then since by assumption \eqref{H_1} $\nu_i > 0$,
    \[ \prod_{l=1}^{p-1} \vert \Lambda_i(u_l,z_l)\vert \leq \tilde Cp^{-\frac{\nu_i}{k}} \leq C'.\]

Then for all $i=1,...,n$,
\begin{equation}\label{estimproduit} \prod_{l=1}^{p-1} \vert \Lambda_i(\GG^{(l)}(u,z))\vert \leq C'.\end{equation}

\textbf{Estimate of $\tilde G_{\ast,i}$:} Since $\tilde K_i$ is flat and since by assumption $z_j \in \D_\rho^n$ for all $j \leq p-1$, there exists a positive constant $C_{4k}$ such that
\[  \vert \tilde K_i(\GG^{(j)}(u,z))\vert \leq C_{4k} \vert u_j\vert^{4k}. \]
Since for all $l\leq p-1$, $U_l \in \Delta_+(\alpha,R)$, then we have $\vert U_l \vert \geq R \sin \alpha$, so that \[ \sum_{l=0}^{p-1}\vert u_l \vert^{4k} = \sum_{l=0}^{p-1}\vert U_l\vert^{-4} \leq \frac{1}{(R \sin \alpha)^2} \sum_{l=0}^{p-1} \vert U_l\vert^{-2} \leq \frac{C_3}{(R\sin \alpha)^2},\]
and then using previous estimates, we have,
\[ \left\vert \sum _{j=0}^{p-1} \tilde K_i(\GG^{(j)}(u,z))\left(\prod_{l=j+1}^{p-1}\Lambda_i(\GG^{(l)}(u,z))\right)\right\vert \leq  C_{4k}\sum _{j=0}^{p-1} \vert u_j\vert^{4k}\prod_{l=j+1}^{p-1}\vert\Lambda_i(u_l,z_l)\vert \leq  \frac{C_{4k}C_3C'}{(R\sin \alpha)^2} \]
where the constants are independent of $p$. Finally, since $\Vert z \Vert \leq \rho_0$, by \re{defGiast} we get
\[  \vert \tilde G_{\ast,i}^{(p)} (u,z)\vert \leq  \rho_0 C' + \frac{C_{4k}C_3C'}{(R\sin \alpha)^2}.\]
Choosing $R$ large enough so that $\frac{C_{4k}C_3C'}{(R\sin \alpha)^2} \leq \frac{\rho}{2}$ (which is possible since $C_3$ has a uniform upper bound for $R\geq R_0$), and letting $\rho_0 := \frac{\rho}{2C'}$, we get $\vert \tilde G_{\ast,i}^{(p)} \vert \leq \rho$ which proves the induction.

\end{proof}
\begin{remark}\label{remarquestability}
    Since $G_{\ast,0}$ is a perturbation of $G_0$, the previous result also implies that, working at infinity, the iterates of $G_{0}$ also remain in $\Delta_+(\alpha,R)$.
\end{remark}

\subsection{Solutions for some functional equations}
The following proposition will be useful to solve the conjugacy equation, in the extended case. We recall that $\GG(u,z)=(\tilde G_{\ast,0}(u,z),\tilde G_\ast(u,z))$.
\begin{prop}\label{Propositionconjug}
We use the notations previously introduced and the conditions of Lemma \ref{stability}. Let $\tilde L \in (\B^\infty)^{n+1}$ be flat in $u \in \delta_+^{[k_0]}(\alpha,r)$ uniformly in $z \in \D_{\rho}^n$. Let $\Lambda = (\Lambda_0, \Lambda_1,..., \Lambda_n)$ with $\Lambda_0 \equiv 1$ and the $\Lambda_i$ defined in \eqref{defLambda} for $i=1,...,n$. We consider the equations in the unknown $\tilde \varphi =(\tilde \varphi_0, \tilde \varphi_1,...,\tilde \varphi_n)$:
\begin{equation}\label{equaphi2}
\tilde \varphi_i(\GG(u,z)) - \tilde \varphi_i(u,z)\Lambda_i(u,z) = \tilde L_i(u,z)
\end{equation}
for all $i=0,...,n$. 
Then $\eqref{equaphi2}$ admits a unique solution $\tilde \varphi \in (\B^\infty)^{n+1}$ (that is to say each $\tilde \varphi_i$ is flat in $u$ uniformly in $z$). Moreover, on $\Sigma = \{u=z^\beta\}$, we have $\tilde G_\ast(z^\beta,z)=G_\ast(z)$ and $\tilde G_{\ast,0}(z^\beta,z)=G_{\ast,0}(z)=G_\ast(z)^\beta$, so that letting $\varphi(z) := \tilde  \varphi (z^\beta,z)$, \[
\tilde\varphi_i(\tilde G_{\ast,0}(z^\beta,z), \tilde G_\ast(z^\beta,z)) = \tilde\varphi_i(G_\ast(z)^\beta, G_\ast(z)) = \varphi_i(G_\ast(z)),
\]
and $\varphi$ satisfies the equation
\begin{equation}\label{equaphisigma}
\varphi_i(G_\ast(z)) - \varphi_i(z)\Lambda_i(z^\beta,z) = \tilde L_i(z^\beta,z), \quad i=0,...,n
\end{equation}
for all $z \in \D_\rho^n$ such that $z^\beta \in \delta_+^{[k_0]}(\alpha,r)$. Moreover, $\varphi$ is the unique solution of \eqref{equaphisigma} which is flat with respect to $z^\beta$.
\end{prop}

The proof will be inspired by \cite[Chapter 6]{Loday} and will use the following results:
		
		\begin{lemma}[{see \cite[Lemma 6.2.1]{Loday}}]\label{somme}
			Let $0 < \alpha < \pi$ and $R > 1$. For all $m \in \N \setminus\{0\}$, there exists a constant $A_m$ which depends on $\alpha$ and $m$ (but not on $R$) such that, for all $U \in \Delta_+(\alpha, R)$,
			\begin{equation}\label{eq:somme}
				\sum_{p \geq 0} \frac{1}{\vert U + p \vert^{m+1}} \leq \frac{A_m}{\vert U \vert^m}.
			\end{equation}
		\end{lemma}

        \begin{lemma}\label{estimationdusinus}
        Let $(U,z) \in \Delta_+(\alpha,R) \times \D_{\rho}^n$, $(\tilde U,\tilde z) \in \C \times \D_{\rho}^n$, and suppose that $\vert U- \tilde U \vert  \leq \sin \alpha$. If $K_0$ is flat in $U$ at infinity uniformly in $z \in \D_\rho^n$, then 
        \begin{equation}\label{flat2}
				\displaystyle\sup_{\substack{\vert U- \tilde U|\leq \sin \alpha \\(U,z) \in \Delta_+(\alpha,R)\times \D_{\rho}^n}}  \vert K_0(\tilde U,\tilde z)\vert \leq \frac{C}{\vert U \vert^{m+1}}
			\end{equation}
            where $C$ depends on $m$.
            If $P: \C\rightarrow \C$ is a polynomial of order $\geq 1$, we have
            \begin{equation}\label{polynomialestimate}
				\displaystyle\sup_{\substack{\vert U- \tilde U|\leq \sin \alpha \\U \in \Delta_+(\alpha,R)}}  \vert P(\frac{1}{\tilde U})\vert \leq \frac{C'_P}{\vert U \vert}
			\end{equation}
            for some $C'_P$ depending on $P$.
        \end{lemma}

   \begin{proof}     
        Condition $\vert U- \tilde U \vert  \leq \sin \alpha$ implies that $\tilde U \in \Delta_+(\alpha,R-1)$ and $\vert U \vert \geq R \sin \alpha$ (see Figures \ref{sinalpha} and \ref{dessin1}). Then for all $m \in \N$
			\[ \vert K_0(\tilde U,\tilde z) \vert\leq \frac{c}{\vert U\vert^{m+1}} \leq \frac{c}{(\vert U \vert - \sin \alpha)^{m+1}}.\]
\begin{figure} 
    \tikzset{every picture/.style={line width=0.75pt}} 

\begin{tikzpicture}[x=0.75pt,y=0.75pt,yscale=-1,xscale=1]

\draw  [draw opacity=0][fill={rgb, 255:red, 184; green, 233; blue, 134 }  ,fill opacity=0.36 ] (450,60) -- (450,240) -- (300,240) -- (370,150) -- (370,150) -- (300,60) -- cycle ;
\draw    (160,150) -- (467,150) ;
\draw [shift={(470,150)}, rotate = 180] [fill={rgb, 255:red, 0; green, 0; blue, 0 }  ][line width=0.08]  [draw opacity=0] (8.93,-4.29) -- (0,0) -- (8.93,4.29) -- cycle    ;
\draw    (212,150) ;
\draw [shift={(212,150)}, rotate = 0] [color={rgb, 255:red, 0; green, 0; blue, 0 }  ][fill={rgb, 255:red, 0; green, 0; blue, 0 }  ][line width=0.75]      (0, 0) circle [x radius= 3.35, y radius= 3.35]   ;
\draw [shift={(212,150)}, rotate = 0] [color={rgb, 255:red, 0; green, 0; blue, 0 }  ][fill={rgb, 255:red, 0; green, 0; blue, 0 }  ][line width=0.75]      (0, 0) circle [x radius= 3.35, y radius= 3.35]   ;
\draw    (240,60) -- (310,150) ;
\draw [shift={(310,150)}, rotate = 52.13] [color={rgb, 255:red, 0; green, 0; blue, 0 }  ][fill={rgb, 255:red, 0; green, 0; blue, 0 }  ][line width=0.75]      (0, 0) circle [x radius= 3.35, y radius= 3.35]   ;
\draw    (240,240) -- (310,150) ;
\draw [shift={(310,150)}, rotate = 307.87] [color={rgb, 255:red, 0; green, 0; blue, 0 }  ][fill={rgb, 255:red, 0; green, 0; blue, 0 }  ][line width=0.75]      (0, 0) circle [x radius= 3.35, y radius= 3.35]   ;
\draw    (300,60) -- (370,150) ;
\draw [shift={(370,150)}, rotate = 52.13] [color={rgb, 255:red, 0; green, 0; blue, 0 }  ][fill={rgb, 255:red, 0; green, 0; blue, 0 }  ][line width=0.75]      (0, 0) circle [x radius= 3.35, y radius= 3.35]   ;
\draw    (300,240) -- (370,150) ;
\draw [shift={(370,150)}, rotate = 307.87] [color={rgb, 255:red, 0; green, 0; blue, 0 }  ][fill={rgb, 255:red, 0; green, 0; blue, 0 }  ][line width=0.75]      (0, 0) circle [x radius= 3.35, y radius= 3.35]   ;
\draw  [draw opacity=0] (351.53,126.58) .. controls (356.07,123.02) and (361.69,120.73) .. (367.88,120.29) .. controls (384.41,119.12) and (398.76,131.57) .. (399.93,148.1) .. controls (399.98,148.79) and (400,149.48) .. (400,150.16) -- (370,150.22) -- cycle ; \draw    (354.01,124.83) .. controls (358.06,122.28) and (362.77,120.66) .. (367.88,120.29) .. controls (384.41,119.12) and (398.76,131.57) .. (399.93,148.1) .. controls (399.98,148.79) and (400,149.48) .. (400,150.16) ;  \draw [shift={(351.53,126.58)}, rotate = 333.44] [fill={rgb, 255:red, 0; green, 0; blue, 0 }  ][line width=0.08]  [draw opacity=0] (10.72,-5.15) -- (0,0) -- (10.72,5.15) -- (7.12,0) -- cycle    ;
\draw  [draw opacity=0] (400,150.16) .. controls (399.98,155.93) and (398.29,161.76) .. (394.79,166.89) .. controls (385.47,180.58) and (366.81,184.12) .. (353.11,174.79) .. controls (352.54,174.41) and (351.99,174) .. (351.45,173.58) -- (370,150) -- cycle ; \draw    (400,150.16) .. controls (399.98,155.93) and (398.29,161.76) .. (394.79,166.89) .. controls (385.47,180.58) and (366.81,184.12) .. (353.11,174.79) ; \draw [shift={(351.45,173.58)}, rotate = 26.7] [fill={rgb, 255:red, 0; green, 0; blue, 0 }  ][line width=0.08]  [draw opacity=0] (10.72,-5.15) -- (0,0) -- (10.72,5.15) -- (7.12,0) -- cycle    ; 
\draw    (289.2,120.1) -- (323.38,95.18) ;
\draw [shift={(325,94)}, rotate = 143.9] [color={rgb, 255:red, 0; green, 0; blue, 0 }  ][line width=0.75]    (10.93,-4.9) .. controls (6.95,-2.3) and (3.31,-0.67) .. (0,0) .. controls (3.31,0.67) and (6.95,2.3) .. (10.93,4.9)   ;
\draw [shift={(287.59,121.28)}, rotate = 323.9] [color={rgb, 255:red, 0; green, 0; blue, 0 }  ][line width=0.75]    (10.93,-4.9) .. controls (6.95,-2.3) and (3.31,-0.67) .. (0,0) .. controls (3.31,0.67) and (6.95,2.3) .. (10.93,4.9)   ;

\draw (372,153.62) node [anchor=north west][inner sep=0.75pt]  [font=\footnotesize]  {$R$};
\draw (203,154.4) node [anchor=north west][inner sep=0.75pt]  [font=\footnotesize]  {$0$};
\draw (312,153.4) node [anchor=north west][inner sep=0.75pt]  [font=\footnotesize]  {$R-1$};
\draw (397,116.4) node [anchor=north west][inner sep=0.75pt]  [font=\footnotesize]  {$\alpha $};
\draw (391,173.4) node [anchor=north west][inner sep=0.75pt]  [font=\footnotesize]  {$-\alpha $};
\draw (386,62.4) node [anchor=north west][inner sep=0.75pt]  [color={rgb, 255:red, 65; green, 117; blue, 5 }  ,opacity=1 ]  {$ \begin{array}{l}
\Delta _{+}( \alpha ,R)\\
\end{array}$};
\draw (278.65,99.52) node [anchor=north west][inner sep=0.75pt]  [font=\small,rotate=-324.13]  {$\sin \alpha $};

\end{tikzpicture}
    \caption{Sector at infinity}
    \label{sinalpha}
\end{figure}
			Now write
			\[ \frac{c}{(\vert U \vert - \sin \alpha)^{m+1}} = \frac{1}{\vert U \vert ^{m+1}} \frac{c \vert U \vert^{m+1}}{(\vert U \vert - \sin \alpha)^{m+1}}.\]
			Notice that $t \mapsto \frac{t}{t-\sin \alpha}$ is strictly decreasing for all $t > \sin \alpha$, then
			\[ \sup_{U \in \Delta_+(\alpha,R)} \frac{\vert U \vert}{\vert U \vert - \sin \alpha} \leq \sup_{\vert U \vert \geq R \sin \alpha} \frac{\vert U \vert}{\vert U \vert - \sin \alpha}  = \frac{R}{R-1}.\]
			Then defining $C = c \displaystyle \sup_{U \in \Delta_+(\alpha,R)} \left( \frac{\vert U\vert}{\vert U \vert - \sin \alpha}\right)^{m+1} = c \left( \frac{R}{R-1} \right)^{m+1}$ 
			we obtain estimate \eqref{flat2}.
				
			If now $P$ is a polynomial in one variable of order $\geq 1$, then for all $\tilde U  \in \C^{\ast}$ with $\vert \tilde U\vert$ large enough (which will be satisfied for $\tilde U$ in the sector at infinity),
			\[\vert P\left(\frac{1}{\tilde U}\right)\vert \leq \frac{C_P}{\vert \tilde U \vert} \]
			and in the same way as above, we get \eqref{polynomialestimate},
			where $C'_P = CC_P$.
\end{proof}

\begin{proof}(of Proposition \ref{Propositionconjug})

The change of variable $U=u^{-k}$ sends $\delta_+^{[k_0]}(\alpha,r)$ to $\Delta_+(\alpha,R)$. Flat functions in $u$ at the origin become flat functions in $U$ at infinity. Abusing notation, we keep the same name for the functions (for instance, we write $\tilde L(U,z)$ for $\tilde L(U^{-\frac{1}{k}},z)$). The dynamics at infinity then become:
\begin{align}
    G_0(U) &= G_0(U^{-1/k})^{-k} = U + 1 + E(\frac{1}{U}), \label{G0inf}\\
    \tilde G_{\ast,0}(U,z) &= G_0(U) + \tilde K_0(U,z),\label{G*0}
\end{align}
where $E$ is holomorphic in a neighborhood of $0$, $E(0) = 0$ and $\vert E(1/U) \vert \leq \frac{C_E}{\vert U \vert}$ for all $U\in \Delta_+(\alpha,R)$.

\textbf{Some preliminary estimates: } \\

\textbf{(*)} By Lemma \ref{somme}, 
        \begin{equation} \label{sommeLoday}\sum_{p\geq 0} \frac{1}{\vert U+p \vert^{m+1} } \leq \frac{A_m}{\vert U \vert^m}, \quad \forall m \in \N^\ast, \forall U \in \Delta_{+}(\alpha,R), \end{equation}
        with $A_m$ depending on $\alpha, m$;

\textbf{(*)}	For $U \in \Delta_{+}(\alpha,R)$, write $U = R + t e^{i\theta}$, with $\vert \theta \vert \leq \alpha$ and $t>0$. For all $j'>0$, $U+j' = (R+j') + t e^{i\theta}\in \Delta_+(\alpha,R+j')$ and then $\vert U +j' \vert \geq (R+j') \sin \alpha$ (see Figure \ref{dessin1}). Since $R > 1$, we get
			\begin{equation}\label{minz02}
			\vert U +j' \vert \geq C_\alpha(R+j') \geq C_\alpha(1+j')
			\end{equation}
			where $C_\alpha = \sin \alpha >0$. Moreover
            \begin{align}
			 \sum_{j'=0}^j \frac{1}{\vert U +j'\vert} \leq \frac{1}{C_\alpha}\sum_{j'=0}^j \frac{1}{1+j'} = \frac{1}{C_\alpha} + \frac{1}{C_\alpha}\sum_{j'=1}^j \frac{1}{1+j'}&\nonumber\\
			  \leq \frac{1}{C_\alpha} + \frac{1}{C_\alpha}\int_{0}^j \frac{1}{1+s}ds = \frac{1}{C_\alpha}+ \frac{1}{C_\alpha}\ln(1+j) \leq M(1+\ln(1+j))  \label{sum}
			\end{align}
			where $M=C_{\alpha}^{-1}$.

\textbf{(*)} 			Let $(U,z) \in \Delta_{+}(\alpha,R) \times \D_{\rho}^n$, and $(\tilde U, \tilde z) \in \C \times \D_{\rho}^n$, and suppose that $\vert U- \tilde U \vert  \leq \sin \alpha$. 
			In exactly the same way as we obtained \eqref{flat2}, we obtain the estimate:
			\begin{equation}\label{flat2bis}
				\displaystyle\sup_{\substack{\vert U- \tilde U|\leq \sin \alpha \\(U,z) \in \Delta_{+}(\alpha,R)\times \D_{\rho}^n}}  \vert \tilde K_0(\tilde U,z)\vert \leq \frac{C}{\vert U \vert^{m+1}}.
			\end{equation}

\textbf{(*)} Let $\rho_E > 0$ be such that $E$ is holomorphic on $D_{\rho_E} = \{ \vert w \vert \leq \rho_E\}$. Since $E(0)=0$, there exists $\eta >0$ such that
\[\vert E(w)\vert \leq \eta\vert w\vert \]
and
\[\vert E(w_1)-E(w_2)\vert \leq \eta \vert w_1-w_2\vert\]
for all $w,w_1,w_2 \in D_{\frac{\rho_E}{2}}$. By definition, we have for all $U$,
\[ G_0^{(p)}(U) = U + p + \sum_{j=0}^{p-1}E\left(\frac{1}{G_0^{(j)}(U)}\right)\]
then for all $p \geq 0$,
\[ \vert G_0^{(p)}(U)-(U+p)\vert = \left\vert \sum_{q=0}^{p-1}E\left(\frac{1}{G_0^{(q)}(U)}\right)\right\vert.\]
By Lemma \ref{stability}, since $G_0^{(q)}(U) \in \Delta_+(\alpha,R)$ for all $q$, $\vert \frac{1}{G_0^{(q)}(U)}\vert \leq \frac{1}{R\sin \alpha}$, then
\[ \left\vert E\left(\frac{1}{G_0^{(q)}(U)}\right)\right\vert \leq \eta  \left\vert \frac{1}{G_0^{(q)}(U)} \right\vert \leq \frac{\eta}{R\sin \alpha}\]
if $R \geq \frac{2}{\rho_E \sin\alpha}$ so that $ \frac{1}{G_0^{(q)}(U)} \in D_{\frac{\rho_E}{2}}$, and then
\[ \vert G_0^{(p)}(U)-(U+p)\vert  \leq p \frac{\eta}{R\sin \alpha}. \]

Since $U+p\in \Delta_+(\alpha, R+p)$, $\vert U+p \vert \geq (R+p)\sin \alpha$ and
\[\vert G_0^{(p)}(U) \vert \geq \vert U+p \vert  - \vert G_0^{(p)}(U)-(U+p)\vert \geq \vert U+p \vert - p \frac{\eta}{R\sin \alpha}. \]
Now
\[\frac{p \frac{\eta}{R\sin \alpha}}{(R+p)\sin\alpha}  = \frac{p \eta}{R(R+p)\sin^2\alpha} \leq \frac{1}{4}\]
if $R \geq \frac{4\eta}{\sin^2\alpha}$
which gives
\begin{equation}\label{estimU+p}
    \vert G_0^{(p)}(U) \vert \geq \vert U+p\vert - \frac{1}{4}\vert U +p\vert \geq \frac{3}{4}\vert U+p\vert.
\end{equation}

\textbf{Control of the iterates:}

Recalling notations \eqref{defproj}, we shall prove by induction on $p\geq 0$, that for all $(U,z) \in \Delta_{+}(\alpha,R)\times \D_{\rho}^n$,
\begin{equation}\label{controleG2} \varepsilon_p(U,z) :=\vert \tilde G_{\ast,0}^{(p)}(U,z)-G_0^{(p)}(U) \vert \leq \sin \alpha\end{equation}

    \underline{p=0,1:} $\varepsilon_0(U,z) =0 \leq \sin \alpha$ trivially. By definition, since $\tilde K_0(U,z)$ is flat at infinity, $\vert \tilde G_{\ast,0}(U,z) - G_0(U) \vert = \vert \tilde K_0(U,z)\vert \leq \frac{C}{\vert U \vert^{m+1}} \leq \frac{C}{(R\sin \alpha)^{m+1}} \leq \sin \alpha$ for $R$ large enough;
    
	 \underline{Iteration:} By \eqref{sommeLoday}, we have
     \[  C\sum_{q=0}^{p-1} \frac{1}{\vert U+q\vert^{m+1}} \leq \frac{A_mC}{\vert U\vert^m} \leq \frac{A_mC}{(R\sin\alpha)^m} \leq \sin \alpha \]
 for $R$ large enough. Assuming that \eqref{controleG2} holds for $l \leq p$ for some $p \geq 0$, then by \eqref{estimU+p} and since for $R \geq 4$, $\vert U +p\vert \geq R \sin \alpha \geq 4 \sin \alpha$,
       \begin{equation}\label{estimG}\vert \tilde G_{\ast,0}^{(p)}(U,z) \vert \geq \vert G_0^{(p)}(U) \vert - \sin \alpha \geq \frac{3}{4}\vert U+p \vert - \frac{1}{4}\vert U +p\vert = \frac{1}{2}\vert U +p\vert.\end{equation}
    From Lemma \ref{stability}, under conditions \eqref{H_0} and \eqref{H_1} we have $(\tilde G_{\ast,0}^{(p)}(U,z),G_\ast^{(p)}(U,z)) \in \Delta_{+}(\alpha,R) \times \D_\rho^n$, then applying \eqref{flat2bis} with $\tilde G_{\ast,0}^{(p)}(U,z)$ for $\tilde U$ and $G_0^{(p)}(U)$ for $U$, we get
    \[ \vert \tilde K_0(\tilde G_{\ast,0}^{(p)}(U,z),G_\ast^{(p)}(U,z)) \vert \leq \frac{C'}{\vert G_0^{(p)}(U)\vert^{m+1}} \leq  \frac{C}{\vert U+p \vert^{m+1}}.\]
  Using that for all $w_1,w_2\in D_{\frac{\rho_E}{2}}$, $\vert E(w_1)-E(w_2) \vert\leq \eta\vert w_1-w_2\vert$, and from previous estimates and iteration assumption, 
   \begin{align*} \left\vert E\left(\frac{1}{\tilde G_{\ast,0}^{(p)}(U,z)}\right)-E\left(\frac{1}{G_0^{(p)}(U)}\right) \right\vert & \leq \eta \left\vert \frac{1}{\tilde G_{\ast,0}^{(p)}(U,z)}-\frac{1}{G_0^{(p)}(U)} \right\vert = \eta \left\vert\frac{\tilde G_{\ast,0}^{(p)}(U,z)-G_0^{(p)}(U)}{\tilde G_{\ast,0}^{(p)}(U,z) G_0^{(p)}(U)} \right\vert \\
   &\leq \eta \frac{8}{3} \frac{1}{\vert U+p\vert^2}\vert \tilde G_{\ast,0}^{(p)}(U,z)- G_0^{(p)}(U)\vert = \eta \frac{8}{3} \frac{1}{\vert U+p\vert^2} \varepsilon_p(U,z).
   \end{align*}
  Since $\vert G_0^{(p)}(U)\vert\geq\frac{3}{4}\vert U+p\vert$ by \eqref{estimU+p} and $\vert \tilde G_{\ast,0}^{(p)}(U,z)\vert\geq\frac{1}{2}\vert U+p\vert$ by \eqref{estimG}, both being $\geq\frac{1}{2}R\sin\alpha$, then $\frac{1}{\tilde G_{\ast,0}^{(p)}(U,z)}$ and $\frac{1}{G_0^{(p)}(U)}$ are in $D_{\frac{\rho_E}{2}}$ for $R$ large enough. Moreover, by Lemma \ref{stability} and Remark \ref{remarquestability}, $\tilde G_{\ast,0}^{(p)}(U,z)\in\Delta_+(\alpha,R)$ and $G_0^{(p)}(U)\in\Delta_+(\alpha,R)$, then
  \[ G_0^{(p+1)}(U)= G_0(G_0^{(p)}(U)) = 1 + G_0^{(p)}(U)+E(\frac{1}{G_0^{(p)}(U)}),\]
  and
  \begin{align*} \tilde G_{\ast,0}^{(p+1)}(U,z) &= G_0(\tilde G_{\ast,0}^{(p)}(U,z)) + \tilde K_0(\GG^{(p)}(U,z)) \\
  & = 1 + \tilde G_{\ast,0}^{(p)}(U,z) + E(\frac{1}{\tilde G_{\ast,0}^{(p)}(U,z)}) + \tilde K_0(\GG^{(p)}(U,z))
  \end{align*}
  which gives, by estimates above and applying \eqref{flat2bis},
  \begin{align*}
      \varepsilon_{p+1}(U,z) & =  \vert \tilde G_{\ast,0}^{(p+1)}(U,z) - G_0^{(p+1)}(U) \vert \\
      & = \left\vert [\tilde G_{\ast,0}^{(p)}(U,z)-G_0^{(p)}(U)] + \left[E\left(\frac{1}{\tilde G_{\ast,0}^{(p)}(U,z)}\right) -E\left(\frac{1}{G_0^{(p)}(U)}\right)\right] + \tilde K_0(\GG^{(p)}(U,z)) \right\vert \\
      &  \leq \varepsilon_p(U,z) + \eta \frac{8}{3}\frac{1}{\vert U +p \vert^2}\varepsilon_p(U,z) + \frac{C}{\vert U+p\vert^{m+1}}
  \end{align*}
  which reads 
   \[ \varepsilon_{p+1}(U,z) \leq  (1+\gamma_p(U))\varepsilon_p(U,z) + \frac{C}{\vert U+p\vert^{m+1}} \]
with $\gamma_p(U) := \frac{8\eta}{3\vert U+p\vert^2}$. Since $\varepsilon_0(U,z) = 0$, we obtain by iterating:
\[  \varepsilon_{p+1}(U,z)  \leq C \sum_{q=0}^{p}\big(\prod_{l=q+1}^{p}(1+\gamma_l(U))\big)\frac{1}{\vert U+q\vert^{m+1}}. \]
Since 
\[ \prod_{l=q+1}^{p}(1+\gamma_l(U)) \leq e^{\ln (\prod_{l=0}^{\infty}(1+\gamma_l(U)))} \leq e^{\sum_{l\geq 0}\gamma_l(U)}\] and by Lemma \ref{somme} with $m=1$, 
\[\sum_{l\geq 0}\gamma_l(U) = \frac{8\eta}{3}\sum_{l\geq 0}\frac{1}{\vert U+l\vert^{2}} \leq \frac{8\eta A_1}{3\vert U\vert},\]
so $\prod_{l=0}^{\infty}(1+\gamma_l(U)) \leq e^{\frac{8\eta A_1}{3\vert  U \vert}} \leq e^{\frac{8 \eta A_1}{3R\sin \alpha}} =: T$, which is independent of $U$ and $p$, and finally,
\[  \varepsilon_{p+1}(U,z) \leq CT\sum_{q=0}^{p}\frac{1}{\vert U+q\vert^{m+1}} \leq  A_mCT \frac{1}{\vert U \vert^m} \leq \frac{A_mCT}{(R\sin \alpha)^m} \leq \sin \alpha \] for $R$ large enough, which proves \eqref{controleG2}.

\textbf{Tentative solutions:} Equation \eqref{equaphi2} at infinity reads, for component $i=1,...,n$,
\[ \tilde \varphi_i(U,z) = \frac{\tilde \varphi_i(\GG(U,z))}{\Lambda_i(U,z)} -\frac{\tilde L_i(U,z)}{\Lambda_i(U,z)}\]
and writing recursively, its solution would satisfy
\begin{equation}\label{candidattotal}\tilde \varphi_i(U,z) = \frac{\tilde \varphi_i(\GG^{(p)}(U,z))}{\prod_{l=0}^{p-1}\Lambda_i(\GG^{(l)}(U,z))} - \sum_{j=0}^{p-1}\frac{\tilde L_i(\GG^{(j)}(U,z))}{\prod_{l=0}^{j}\Lambda_i(\GG^{(l)}(U,z))}.\end{equation}
Hence, the tentative solution is given by
\begin{equation}\label{solutioncandidatei} \tilde \varphi_i(U,z) = - \sum_{j\geq 0}\frac{\tilde L_i(\GG^{(j)}(U,z))}{\prod_{l=0}^{j}\Lambda_i(\GG^{(l)}(U,z))}.\end{equation}
For the component $i=0$, the tentative solution is simply 
\begin{equation} \tilde \varphi_0 (U,z)= - \sum_{j\geq 0} \tilde L_0(\GG^{(j)}(U,z)) \label{solutioncandidate0}.\end{equation}
We shall prove that the infinite sums \eqref{solutioncandidatei} and \eqref{solutioncandidate0} define holomorphic functions in suitable domains. \\

\textbf{Convergence of the sum: }
Let us first investigate the case $i=0$. Since $\tilde L_0 \circ \GG^{(j)}$ is holomorphic for all $j \geq 0$, so is the sum $\sum_{j\geq 0} \tilde L_0 \circ \GG^{(j)}$ on $\Delta_{+}(\alpha,R)\times \D_\rho^n$. Moreover, since $\tilde L_0$ is flat, and applying \eqref{flat2bis} with $\tilde G_{\ast,0}^{(j)}(U,z)$ for $\tilde U$ and $G_0^{(j)}(U)$ for $U$ (thanks to \eqref{controleG2}), from Lemma \ref{stability} and according to Lemma \ref{somme} we get 
\begin{equation} \vert \sum_{j\geq 0}\tilde L_0(\GG^{(j)}(U,z)) \vert 
\leq \sum_{j\geq 0}\frac{C}{\vert U+j\vert^{m+1}}\leq \frac{A_mC}{\vert U \vert^m}\label{estimflat}\end{equation}
and then $\tilde \varphi_0$ is flat in $U$ uniformly in $z$.

We now investigate the case of component $i=1,...,n$. Write $\Lambda_i(U,z) = \lambda_i + d_i(U,z)$. Then 
\[\prod_{l=0}^j \Lambda_i(\GG^{(l)}(U,z)) = \lambda_i^{j+1}\prod_{l=0}^j(1+a_l(U,z))\]
where $a_l(U,z) = \lambda_i^{-1}d_i(\GG^{(l)}(U,z))$. Then from \eqref{estimG} and under condition \eqref{H_0},  $\vert a_l(U,z) \vert \leq \frac{C}{\vert \tilde G_{\ast,0}^{(l)}(U,z)  \vert}$ and we can assume that $\vert a_l(U,z)\vert \leq \frac{1}{2}$ so that $\vert 1 + a_l(U,z)\vert > 0$. Then
		\begin{align}
				\vert \prod_{l=0}^j (1+a_{l}(U,z))\vert & = \vert e^{\ln \prod _{l=0}^j(1+a_{l}(U,z))} \vert = \vert e^{\sum _{l=0}^j\ln (1+a_{l}(U,z))}  \vert = e^{\RE \big( \sum _{l=0}^j\ln (1+a_{l}(U,z))\big)} \nonumber\\
				& > e^{- \vert \sum _{l=0}^j\ln (1+a_{l}(U,z))\vert} 
				 > e^{-2 \sum_{l=0}^j \vert a_{l}(U,z)\vert } 
				 > e^{-2 \sum_{l=0}^j \frac{C'}{\vert U +l\vert} }\label{prod-min}
			\end{align} 
			where $C'=2C$ coming from \eqref{estimG}.
		Therefore, using \eqref{minz02} we obtain
			\[ 	\vert \prod_{l=0}^j (1+a_{l}(U,z))\vert  > e^{\frac{-2C'}{C_\alpha} \sum_{l=0}^j \frac{1}{1+l}}\]
          \[> e^{- \frac{4C'}{C_\alpha} \ln(1+j)} = \frac{1}{(1+j)^{M'}}\]
			with $M'=\frac{4C'}{C_\alpha}$, therefore 
            \begin{equation}\label{estimdeno} \vert\prod_{l=0}^j \Lambda_i(\GG^{(l)}(U,z)) \vert \geq (1+j)^{-M'}.\end{equation}

            Since $\tilde L_i$ is flat 
            , thanks to Lemma \ref{stability} and in the same way as \eqref{estimflat} we get for all $N\geq 0$
            \[ \vert \tilde L_i (\GG^{(j)}(U,z))\vert \leq \frac{C}{\vert \tilde G_{\ast,0}^{(j)}(U,z)\vert^{N+1}} \leq \frac{C}{\vert U+j \vert^{N+1}},\]
and moreover
\[\vert \sum_{j\geq 0}\frac{\tilde L_i(\GG^{(j)}(U,z))}{\prod_{l=0}^{j}\Lambda_i(\GG^{(l)}(U,z))} \vert \leq \sum_{j\geq 0}\frac{C}{\vert U+j\vert^N}(1+j)^{M'}  \leq \sum_{j\geq 0} \frac{\tilde C'}{\vert U+j\vert^{N-M'}} \]
where $\tilde C' = \frac{C}{C_\alpha ^{M'}}$. For any $m\geq 1$, choose $N\geq m + M' +1$ so that 
\[\vert \sum_{j\geq 0}\frac{\tilde L_i(\GG^{(j)}(U,z))}{\prod_{l=0}^{j}\Lambda_i(\GG^{(l)}(U,z))} \vert \leq \sum_{j\geq 0} \frac{\tilde C'}{\vert U+j\vert^{m+1}} \]
which proves that the series converges uniformly on compact sets of $\Delta_{+}(\alpha,R)\times \D_\rho^n$ and then $\tilde \varphi_i$ is holomorphic and from $\eqref{sommeLoday}$, $\vert \tilde \varphi_i(U,z) \vert \leq \frac{A_m \tilde C'}{\vert U \vert^{m}}$. As $m$ is chosen arbitrary, this proves that $\tilde \varphi_i \in \B^\infty$. 
Finally, let us check that $\tilde \varphi$ satisfies \eqref{equaphi2}. Since $\GG^{(j)}(\GG(U,z)) = \GG^{(j+1)}(U,z)$, we have
\[ \tilde \varphi_i(\GG(U,z)) = - \sum_{j\geq 1}\frac{\tilde L_i(\GG^{(j)}(U,z))}{\prod_{l=1}^{j}\Lambda_i(\GG^{(l)}(U,z))},\]
hence
\[ \frac{\tilde \varphi_i(\GG(U,z))}{\Lambda_i(U,z)} = - \sum_{j\geq 1}\frac{\tilde L_i(\GG^{(j)}(U,z))}{\prod_{l=0}^{j}\Lambda_i(\GG^{(l)}(U,z))} = \tilde \varphi_i(U,z) + \frac{\tilde L_i(U,z)}{\Lambda_i(U,z)}.\]
 Multiplying by $\Lambda_i(U,z)$ gives \eqref{equaphi2}. The case $i=0$ is identical with $\Lambda_0 \equiv 1$.
Finally, $\tilde \varphi = (\tilde \varphi_0,...,\tilde \varphi_n)$  satisfies conjugacy equation \eqref{equaphi2} (in the variable $U$ at infinity).

\textbf{Uniqueness:} Let $\psi \in (\B^\infty)^{n+1}$, $\psi = (\psi_0,...,\psi_n)$ be a solution of
\[\psi (\GG(U,z)) = \psi(U,z)\Lambda(U,z),\]
then by iteration, for all $i=0,...,n$
\[ \psi_i(U,z) = \frac{\psi_i(\GG^{(p)}(U,z))}{\prod_{l=0}^{p-1}\Lambda_i(\GG^{(l)}(U,z))}.\]
Since $\psi_i$ is flat and from estimate \eqref{estimG}, using again \eqref{flat2bis} and Lemma \ref{stability}, we get for all $N$ 
\[ \vert \psi_i(\GG^{(p)}(U,z))\vert = \vert \psi_i(\tilde G^{(p)}_{\ast,0}(U,z),\tilde G^{(p)}_{\ast}(U,z))\vert 
\leq \frac{C_N}{\vert U+p\vert^N}\]
and by \eqref{estimdeno}, 
\[\vert\prod_{l=0}^{p-1} \Lambda_i(\GG^{(l)}(U,z)) \vert \geq p^{-M'}.\]
Then for all $(U,z) \in \Delta_{+}(\alpha,R) \times \D_\rho^n$, and choosing $N$ large enough,
\[ \vert \psi_i(U,z) \vert \leq \frac{C''}{\vert U +p\vert^{N-M' }}\underset{p \rightarrow\infty}{\rightarrow} 0\]
so that $\psi \equiv0$ in the same domain. Hence, if $\tilde \varphi^{(1)}$ and $\tilde \varphi^{(2)}$ are two solutions of \eqref{equaphi2} in $(\B^\infty)^{n+1}$, their difference is a flat solution of the homogeneous equation, hence vanishes identically, which proves uniqueness.

\textbf{Restriction to $\Sigma$: } Coming back to the variable $u=U^{-\frac{1}{k}}$ we constructed a unique $\tilde \varphi$ satisfying conjugacy equation  \eqref{equaphi2}, flat at the origin in the variable $u\in \delta_+^{[k_0]}(\alpha,r)$ uniformly in the variable $z\in \D_\rho^n$. 
By construction, $\tilde G_\ast(z^\beta,z) = G_\ast(z)$ and $\tilde G_{\ast,0}(z^\beta,z)=G_\ast(z)^\beta = G_{\ast,0}(z)$, so we have, writing $\varphi(z) := \tilde \varphi(z^\beta,z)$:
\[\tilde \varphi(\tilde G_{\ast,0}(z^\beta,z),\tilde G_\ast(z^\beta,z)) = \tilde\varphi(G_\ast(z)^\beta,G_\ast(z))=\varphi(G_\ast(z))\]
and then restricting 
\[ \tilde \varphi(\tilde G_{\ast,0}(u,z),\tilde G_\ast (u,z)) - \tilde \varphi(u,z)\Lambda(u,z) = \tilde L(u,z)\]
to $\Sigma$ gives exactly
\[\varphi(G_\ast(z))-\varphi(z)\Lambda(z^\beta,z) = \tilde L(z^\beta,z)\]
for all $z\in \D_\rho^n$ such that $z^\beta \in \delta_+^{[k_0]}(\alpha,r)$.

Finally, this restricted solution is unique. Indeed, since $\tilde G_\ast(z^\beta,z)=G_\ast(z)$ and $\tilde G_{\ast,0}(z^\beta,z)=G_{\ast,0}(z)$, we have $\GG^{(l)}(z^\beta,z)=(G_{\ast,0}^{(l)}(z),G_\ast^{(l)}(z))$ for all $l\geq 0$, and the iterates of $\GG$ are in $\Sigma$ so that the uniqueness is proved the same way as above.
\end{proof}

\begin{remark}
    \label{kplusgrand}
Theorem \ref{main-sect} still holds if one only assumes $\ord_0 Q_i \geq k$ for $i=m+1,\ldots,n$ ($\ord_0 P_i \geq k$ in this section). Indeed, if $a_{i,1}=0$, then $\Lambda_i(u,z) = \lambda_i + \OO(u^{k+1})$, so that
\[\ln \vert \Lambda_i(u_l,z_l)\vert = \OO(l^{-1-\frac{1}{k}}),\]
and the series $\sum_l \ln\vert \Lambda_i(u_l,z_l)\vert$ converges absolutely so that estimate \eqref{estimproduit} is satisfied and no condition on $\nu_i$ is needed for such an $i$. Moreover, the product $\prod_{l=1}^{p-1}\vert \Lambda_i(u_l,z_l)\vert$ is bounded from below by some positive constant uniformly in $p$. Then, the estimates \eqref{estimdeno} and the ones following are simplified and the result still holds.
\end{remark}

\subsection{Sectorial conjugacy}
We will construct in this section the conjugacy to the normal form, on each sector. \\

\textbf{Resonant Monomial conjugacy.}
Let $r,R,\alpha$ and $k_0$ be as defined before. Let
\[\s_+^{[k_0]} = \{(u,z)=(u,z_1,...,z_n) \in \delta^{[k_0]}_+(\alpha, r)\times (\D_\rho\setminus\{0\})^{m}\times  \D_\rho^{n-m}\}.\]

We want a holomorphic map $\phi_+ : \s_+^{[k_0]} \rightarrow \C$ of the form $\phi_+(u,z) = u +  \varphi_{0,+}(u,z)$ with $ \varphi_{0,+} \in \B^\infty$ (so that $\phi_+$ is a flat perturbation of the identity at the origin in $u$), satisfying the conjugacy equation
\[ \phi_+(\tilde G_{\ast,0}(u,z), \tilde G_\ast(u,z)) = G_0(\phi_+(u,z)), \]
which is equivalent to
\[ \varphi_{0,+}(\tilde G_{\ast,0}(u,z), \tilde G_\ast(u,z)) - \varphi_{0,+}(u,z) = -\tilde K_0(u,z) + Q(u,\varphi_{0,+}(u,z)), \]
where $Q(u,v) = \sum_{i\geq 0,j\geq1, i+j \leq 2k\vert \beta\vert +1 }r_{i,j}u^iv^j$ and we cannot apply Proposition \ref{Propositionconjug} since the right side of the equation depends on $\varphi_{0,+}$. Going back to infinity, up to a dilation, we continue abusing  notations and use the definitions of $G_0$ and $\tilde G_{\ast,0}$ introduced in \eqref{G0inf}, \eqref{G*0}. We want $\phi_+(U,z) = U + \varphi_{0,+}(U,z)$ defined on $\Delta_+(\alpha,R)\times \D_\rho^{n}$, with $\varphi_{0,+}$ flat at infinity in $U$ uniformly in $z$, satisfying

\begin{equation}\label{relphi+} \phi_+ (\tilde G_{\ast,0}(U,z),\tilde G_\ast(U,z)) = G_0 (\phi_+(U,z)).\end{equation}
This equation is equivalent to 
\begin{equation} \varphi_{0,+}(\GG(U,z))-\varphi_{0,+}(U,z) = - \tilde K_0(U,z) + Q(U,\varphi_{0,+}(U,z)) \label{equationmonom}\end{equation}
where more precisely $Q(U,v) := E((U+v)^{-1})-E(U^{-1})$. We will use a fixed point result to solve this equation. 

\textbf{(*)} For $s\geq 2$, define
\[ E_s := \{ f \in \mathcal{H}(\Delta_+(\alpha,R)\times \D_\rho^n), \Vert f \Vert_s:= \sup_{(U,z)\in \Delta_+(\alpha,R)\times \D_\rho^n} \vert f(U,z)\vert \cdot\vert U \vert^s < \infty  \}\]
(where $\mathcal{H}$ stands for holomorphic functions). Then $(E_s, \Vert \cdot\Vert_s)$ is a Banach space, and since $ \vert U \vert \geq R \sin \alpha$ on $\Delta_+(\alpha,R)$, the inclusion $\iota : E_{s+1} \rightarrow E_s$ is continuous with $\Vert f \Vert_s \leq \frac{1}{R \sin \alpha}\Vert f \Vert_{s+1}$. Notice that $f$ is flat at infinity if and only if $f \in E_s$ for all $s \geq 2$.

\textbf{(*)} $Q(U,v) = E((U+v)^{-1})-E(U^{-1})$ satisfies $Q(U,0)=0$. Given $U$, define $g: v \mapsto E((U+v)^{-1})$. Then $g'(v) = - \frac{E'((U+v)^{-1})}{(U+v)^2}$ and if $\vert v \vert \leq \frac{\vert U \vert}{2}$, $\vert U + v \vert \geq \frac{\vert U \vert}{2}$ and then $\frac{1}{\vert U + v\vert^2} \leq \frac{4}{\vert U \vert^2}$. Moreover, $\vert (U+v)^{-1}\vert\leq \frac{2}{\vert U \vert} \leq \frac{\rho_E}{2}$ for $R$ large enough, so that $\vert E'((U+v)^{-1})\vert$ is bounded by some constant depending only on $E$. Then for all $v_1, v_2 \in \{v, \vert v\vert \leq \frac{\vert U\vert}{2}\}$, integrating $g'$ along $[v_1,v_2]$, we have
\[ \vert Q(U,v_1)-Q(U,v_2)\vert = \vert g(v_1)-g(v_2)\vert \leq \frac{c_E}{\vert U \vert^2}\vert v_1 - v_2 \vert \]
where $c_E$ depends on $E$. In particular,
\begin{equation}\label{estimQ} \vert Q(U,v) \vert \leq \frac{c_E}{\vert U \vert^2}\vert v\vert 
\end{equation}

\textbf{(*)} Proposition \ref{Propositionconjug} with $i=0$ and $\Lambda_0 \equiv 1$ gives tentative solution \eqref{solutioncandidate0}. Define the operator $S: E_{s+1} \rightarrow E_s$ by
\[ S(L)(U,z) := - \sum_{j\geq 0}L(\GG^{(j)}(U,z))\]
which for a function $L \in E_{s+1}$ gives a unique solution $S(L)\in E_s$ of the equation
\[S(L)(\GG(U,z)) - S(L)(U,z) = L(U,z).\]
Given $L \in E_{s+1}$, $\vert L(U,z)\vert\leq \frac{\Vert L \Vert_{s+1}}{\vert U \vert^{s+1}}$ and from estimate \eqref{estimG},
\[ \vert L(\GG^{(j)}(U,z))\vert \leq \frac{\Vert L \Vert_{s+1}}{\vert \tilde G_{\ast,0}^{(j)}(U,z) \vert^{s+1}} \leq 2^{s+1}\frac{\Vert L \Vert_{s+1}}{\vert U+j \vert^{s+1}} \]
then
\[\vert S(L)(U,z) \vert \leq \sum_{j\geq 0}2^{s+1}\frac{\Vert L \Vert_{s+1}}{\vert U+j \vert^{s+1}} \leq 2^{s+1}A_s \frac{\Vert L \Vert_{s+1}}{\vert U \vert^{s}}\]
where $A_s$ is given by Lemma \ref{somme}. Finally \begin{equation}\Vert S(L)\Vert_{s} \leq K_s \Vert L \Vert_{s+1} \label{defK_s}\end{equation} where $K_s = 2^{s+1}A_s$.

\textbf{(*)} Given $\psi \in E_s$ such that $\Vert \psi \Vert_s \leq \frac{1}{2}(R\sin \alpha)^{s+1}$ then for all $(U,z)\in \Delta_+(\alpha,R)\times \D_{\rho}^n$, $\vert \psi(U,z)\vert  \leq \frac{1}{2}\vert U\vert$. Indeed, if $\Vert \psi \Vert_s \leq \frac{(R\sin \alpha)^{s+1}}{2} \leq \frac{\vert U \vert^{s+1}}{2}$ then $\frac{\Vert \psi \Vert_s}{\vert U \vert^s} \leq \frac{\vert U \vert}{2} \Rightarrow \vert \psi(U,z)\vert \leq \frac{\vert U \vert}{2}$. Applying \eqref{estimQ}, we have
\[ \vert Q(U,\psi(U,z))\vert \leq \frac{c_E}{\vert U \vert^2} \vert \psi(U,z)\vert \leq \frac{c_E}{\vert U\vert^{s+2}} \Vert \psi\Vert_s.\]
Then $Q(\cdot, \psi) \in E_{s+2}$ with
\begin{equation}\Vert Q (\cdot, \psi)\Vert_{s+2} \leq c_E \Vert \psi \Vert_s.\label{QdansEs+2}\end{equation}
Finally, given $b \leq \frac{1}{2}(R\sin \alpha)^{s+1}$, we define 
\[  \bar B_b = \{ \psi \in E_s, \Vert \psi \Vert_s \leq b \}.\]
By \re{estimQ}, we have for all $\psi_1, \psi_2 \in \bar B_b \subseteq E_s$
\begin{equation*} \Vert Q(\cdot, \psi_1) - Q(\cdot, \psi_2)\Vert_{s+2} \leq c_E \Vert \psi_1 - \psi_2 \Vert _s. \end{equation*}
More generally, by \re{estimQ}, for all $\psi_1,\psi_2\in E_s$ satisfying $\vert\psi_j(U,z)\vert\leq\frac{1}{2}\vert U\vert$ on $\Delta_+(\alpha,R)\times\D_\rho^n$ (which is in particular the case if $\psi_1,\psi_2\in\bar B_b$), we have
\begin{equation}\label{estimQpsi} \Vert Q(\cdot, \psi_1) - Q(\cdot, \psi_2)\Vert_{s+2} \leq c_E \Vert \psi_1 - \psi_2 \Vert _s. \end{equation}

\textbf{(*)} Let's define the operator $T$ defined on $E_s$ by~:
\[ T(\psi) := S(-\tilde K_0 + Q(\cdot, \psi)),\quad\forall \psi \in E_s.\]
Since $\psi \in E_s$, we have $Q(\cdot, \psi) \in E_{s+2}$ and then $S(-\tilde K_0 + Q(\cdot, \psi)) \in E_{s+1}$. From the inclusion $\iota: E_{s+1} \rightarrow E_s$, we can view $T$ as an operator taking values in $E_s$. Notice that a fixed point of $T$ in $E_s$ satisfies \eqref{equationmonom}. 

Given $s = 2$, let $c_0 := K_2 \Vert \tilde K_0 \Vert_{3}$ (where $K_2$ is defined in \eqref{defK_s}, and $\Vert \tilde K_0 \Vert_{3} < \infty$ since $\tilde K_0$ is flat), and let $b:=2 c_0$. For $\psi \in \bar B_{b} \subseteq E_2$,
\[ \Vert T(\psi) \Vert_2 = \Vert S(-\tilde K_0 + Q(\cdot, \psi)) \Vert_2\leq \Vert S(-\tilde K_0)\Vert_2 + \Vert S(Q(\cdot, \psi)) \Vert_2.\]
From previous estimates,
\[ \Vert S(-\tilde K_0)\Vert_2 \leq K_2 \Vert \tilde K_0 \Vert_{3} =c_0\]
and
\[ \Vert S(Q(\cdot, \psi))\Vert_2 \leq \frac{1}{R \sin \alpha} \Vert S(Q(\cdot,\psi))\Vert_{3} \leq \frac{K_{3}}{R\sin\alpha} \Vert Q(\cdot, \psi)\Vert_{4} \leq \frac{K_{3}c_E}{R\sin \alpha} \Vert \psi \Vert_{2} \leq \frac{K_{3} c_E}{R \sin \alpha}b.\]

For $R$ large enough so that $\frac{K_{3} c_E}{R \sin \alpha} \leq \frac{1}{2}$ and $b \leq \frac{(R\sin \alpha)^3}{2} $, 
\[ \Vert T(\psi)\Vert_2 \leq c_0 + \frac{b}{2} = c_0 + c_0 = b\]
and then $T(\bar B_{b}) \subseteq \bar B_{b}$. Moreover, by a similar reasoning, for all $\psi_1,\psi_2 \in \bar B_b$,
\[ \Vert T(\psi_1) - T(\psi_2)\Vert_{2} = \Vert S(Q(\cdot, \psi_1))-S(Q(\cdot, \psi_2))\Vert_2 \leq \frac{K_3c_E}{R\sin \alpha} \Vert \psi_1 - \psi_2\Vert_2\leq \frac{1}{2}\Vert \psi_1 - \psi_2\Vert_2\]
so that $T$ is $\frac{1}{2}$-contracting on $\bar B_b$. Then, by the Banach fixed-point theorem on $\bar B_b$, $T$ has a unique fixed point $\psi \in \bar B_b \subseteq E_2$, hence $\psi$ satisfies
\[ \psi(\GG(U,z))-\psi(U,z) = -\tilde K_0(U,z)+Q(U,\psi(U,z)).\]

\textbf{(*)} We now show by induction on $s\geq 2$ that $\psi \in E_s$. We already proved that $\psi \in E_2$. Let us suppose that $\psi \in E_s$ for some $s\geq 2$. Then since $Q(\cdot,\psi) \in E_{s+2}$ from estimate \eqref{QdansEs+2},  and since $\tilde K_0 \in E_{s+2}$ as $K_0$ is flat by assumption, we have $-\tilde K_0 + Q(\cdot, \psi) \in E_{s+2}$. Moreover, $S: E_{s+2} \rightarrow E_{s+1}$ gives $\psi = T(\psi) \in E_{s+1}$.

By induction, $\psi \in E_s$ for all $s\geq 2$ which means that $\psi$ is flat at infinity in $U$ uniformly in $z$.

Finally, we get a solution $\varphi_{0,+}$ to \eqref{equationmonom} on $\Delta_+(\alpha,R)\times \D_\rho^n$. Coming back to the variable $u$, the function $\varphi_{0,+}$ is in $\B^{\infty}$ and then, the diffeomorphism $\phi_+(u,z) = u+\varphi_{0,+}(u,z)$ satisfies
\[ \phi_+(\tilde G_{\ast,0}(u,z), \tilde G_\ast(u,z)) = G_0(\phi_+(u,z)). \]

\textbf{(*)} It remains to show the uniqueness in $\B^\infty$. Indeed, the Banach fixed-point theorem provided a solution in $\bar B_b$ and then the uniqueness is only in $\bar B_b$. \\
Let $\psi_1$ and $\psi_2$ in $\B^\infty$ two solutions of \eqref{equationmonom},
\[ \psi_1(\GG(U,z)) - \psi_1(U,z) = -\tilde K_0(U,z) + Q(U,\psi_1(U,z)),\]
\[ \psi_2(\GG(U,z)) - \psi_2(U,z) = -\tilde K_0(U,z) + Q(U,\psi_2(U,z)),\]
and then, $\chi := \psi_1-\psi_2$ 
\[ \chi(\GG(U,z)) -\chi(U,z) = Q(U,\psi_1(U,z))-Q(U,\psi_2(U,z)),\]
which can be written, letting $L(U,z) :=Q(U,\psi_1(U,z))-Q(U,\psi_2(U,z))$,
\[ \chi(\GG(U,z)) -\chi(U,z) = L(U,z).\]

We would like to apply estimate \eqref{estimQpsi} on $\psi_1,\psi_2$, which requires for $j=1,2$, $\vert \psi_j(U,z)\vert \leq \frac{\vert U \vert}{2}$. Since $\psi_1$ and $\psi_2$ are flat, they are in $E_2$, then for $j=1,2$, $\vert \psi_j(U,z)\vert \leq \frac{\Vert \psi_j\Vert_2}{\vert U \vert^2}$, which is $\leq \frac{1}{2}\vert U \vert$ if $\vert U \vert \geq R' \geq 2^\frac{1}{3}\max\big(\Vert\psi_1\Vert_2^{\frac{1}{3}},\Vert \psi_2\Vert_2^{\frac{1}{3}} \big)$. If $R'' = \max(R,\frac{R'}{\sin \alpha})$, $U \in \Delta_+(\alpha,R'')$ satisfies $\vert U \vert \geq R'' \sin \alpha \geq R'$ so we can apply \eqref{estimQpsi} on $\Delta_+(\alpha,R'')\times \D_\rho^n$. Since $\psi_1,\psi_2 \in E_2$, then $L \in E_4$ and then from \eqref{estimQpsi}, $\Vert L \Vert_4 \leq c_E \Vert \chi\Vert_2$ on $\Delta_+(\alpha,R'')\times \D_\rho^n$. Since $\chi$ is flat and satisfies $\chi \circ \GG - \chi = L$, the uniqueness of $\chi = S(L)$ follows from Proposition \ref{Propositionconjug} with $\Lambda_0 \equiv1$. From estimates \eqref{defK_s} and using the inclusion $\iota: E_3 \rightarrow E_2$,
\begin{equation}\Vert \chi \Vert_2 \leq \frac{1}{R''\sin \alpha} \Vert S(L)\Vert_3 \leq \frac{K_3}{R''\sin \alpha} \Vert L \Vert_4 \leq \frac{K_3 c_E}{R''\sin \alpha} \Vert \chi \Vert_2. \label{estimchi}\end{equation}
Since $R'' \geq R$ and since $R$ has been chosen so that $\frac{K_3 c_E}{R \sin \alpha}\leq \frac{1}{2}$, we get $\Vert \chi \Vert_2 = 0$ and then $\chi \equiv 0$ on $\Delta_+(\alpha,R'')\times \D_\rho^n$. Since $\chi \equiv 0$ on the open subset $\Delta_+(\alpha,R'')\times \D_\rho^n$ of $\Delta_+(\alpha,R)\times \D_\rho^n$ we also have that $\chi \equiv 0$ on $\Delta_+(\alpha,R)\times \D_\rho^n$. Finally, this proves that $\varphi_{0,+}$ is the unique solution of \eqref{equationmonom} in $\B^\infty$.

Restricting to $\Sigma=\{u=z^{\beta}\}$, we get a solution of
\begin{equation}\label{equ0}
\phi_+ (\tilde G_{\ast,0}(z^{\beta},z),\tilde G_\ast(z^\beta,z)) = \phi_+ ( G_{\ast,0}(z),G_\ast(z)) =G_0 (\phi_+(z^{\beta},z)).
\end{equation}

 Let us show the uniqueness of this solution is also true on $\Sigma$. Since $\tilde G_\ast(z^\beta,z)=G_\ast(z)$ and $\tilde G_{\ast,0}(z^\beta,z)=G_{\ast,0}(z)$, we have $\GG^{(j)}(z^\beta,z)=(G_{\ast,0}^{(j)}(z),G_\ast^{(j)}(z))$ for all $j\geq 0$. Then estimate \eqref{estimG} still holds (with $U=(z^\beta)^{-k}$, modulo the dilation) so that estimate \eqref{defK_s} also holds, the spaces $E_s$ and the operator $S$ being now considered on $\s_+^{[k_0]}$, that is, for functions of $z$ alone, with the same constants. \\
Let $\phi_1(z)  = z^\beta + \psi_1(z)$ and $\phi_2(z)  = z^\beta + \psi_2(z)$ two solutions of
\begin{equation}\label{equ0sigma}\phi(G_\ast(z))=G_0(\phi(z))\end{equation}
on $\s_+^{[k_0]}$, with $\psi_1$ and $\psi_2$ flat with respect to $z^\beta$. Then as in \eqref{equationmonom}, we have
\[ \psi_1(G_\ast(z))-\psi_1(z) = -\tilde K_0(z^\beta,z)+Q(U,\psi_1(z)),\]
\[ \psi_2(G_\ast(z))-\psi_2(z) = -\tilde K_0(z^\beta,z)+Q(U,\psi_2(z)),\]
and letting $\chi = \psi_1-\psi_2$ we get
\[\chi(G_\ast(z)) - \chi(z) = Q(U,\psi_1(z)) - Q(U,\psi_2(z)).\]
Letting $L(z) := Q(U,\psi_1(z))-Q(U,\psi_2(z))$, we have $\chi \circ G_\ast - \chi = L$. Since $\psi_1$ and $\psi_2$ are flat, they are in $E_2$, so that $\vert \psi_j(z) \vert \leq \frac{1}{2} \vert U\vert$ for $\vert U\vert\geq R'$ as above, and moreover estimate \eqref{estimQpsi} gives $L\in E_4$ with $\Vert L \Vert_4 \leq  c_E \Vert \chi \Vert_2$. As $\chi$ is flat and satisfies $\chi \circ G_\ast - \chi = L$, Proposition \ref{Propositionconjug} with $\Lambda_0 \equiv 1$ gives $\chi  =S(L)$, and then by \eqref{defK_s} and the inclusion $\iota : E_ 3\rightarrow E_2$, Equation \eqref{estimchi} still holds. Hence here again $\Vert\chi\Vert_2=0$, so that $\chi \equiv 0$ and $\phi_1 = \phi_2 $ where $\vert U \vert\geq R'$, so that $\phi_1 = \phi_2 $ on $\s_+^{[k_0]}$ by analytic continuation.

\textbf{Complete conjugacy: }We would now like to construct a complete conjugacy $\Phi_+ = (\Phi_{1,+},...,\Phi_{n,+})$, with each $\Phi_{i,+}$ tangent to the identity such that
\begin{equation} \Phi_+ \circ G_\ast = G \circ\Phi_+ \label{conjugg}\end{equation}
which satisfies
\begin{equation}\label{condphi} \Phi_+(z)^\beta = \phi_+(z^\beta,z).\end{equation}
We shall construct first $\Phi_{i,+}$ for $i=1,...,m-1$, assuming the constraint \eqref{condphi} is satisfied. Then we will construct $\Phi_{m,+}$ to ensure that condition \eqref{condphi} is indeed satisfied. Finally we shall construct $\Phi_{i,+}$ for $i = m+1,...,n$. 

For all $i=1,...,m-1$, let us set
\[ \Phi_{i,+}(z) = z_i + \varphi_{i,+}(z).\]
For a given $1\leq i\leq m-1$, on the one hand, we have
\[ \Phi_{i,+}(G_\ast(z)) = G_{\ast,i}(z) + \varphi_{i,+}(G_\ast(z)) = G_i(z) + K_i(z) + \varphi_{i,+}(G_\ast(z)). \]
On the other hand, according to \re{condphi}, we have 
\begin{align*}
     G_i(\Phi_+(z)) & = \Phi_{i,+}(z)\big(\lambda_i+a_i(\Phi_+(z)^\beta)^k + b_i  (\Phi_+(z)^\beta)^{2k}\big) \\
     & = \Phi_{i,+}(z)\big(\lambda_i+a_i(\phi_+(z^\beta,z))^k + b_i  (\phi_+(z^\beta,z))^{2k}\big) \\
     & = (z_i + \varphi_{i,+}(z))\big(\lambda_i+a_i(\phi_+(z^\beta,z))^k + b_i  (\phi_+(z^\beta,z))^{2k}\big).
     \end{align*}
Hence, the $i$-th component of conjugacy equation \eqref{conjugg} can be written as
\begin{align} & G_i(z) + K_i(z) + \varphi_{i,+}(G_\ast(z)) = \label{equisect}\\
	& z_i \big(\lambda_i+a_i(\phi_+(z^\beta,z))^k + b_i  (\phi_+(z^\beta,z))^{2k}\big) + \varphi_{i,+}(z)\big(\lambda_i+a_i(\phi_+(z^\beta,z))^k + b_i  (\phi_+(z^\beta,z))^{2k}\big).\nonumber
\end{align}
Note that 
\begin{align*}
G_i(z) - z_i \big(\lambda_i+a_i(\phi_+(z^\beta,z))^k + b_i  (\phi_+(z^\beta,z))^{2k}\big) & = z_i (\lambda_i + a_i (z^{\beta})^k + b_i  (z^{\beta})^{2k})\\
& - z_i(\lambda_i + a_i \phi_+(z^\beta,z)^k + b_i  \phi_+(z^\beta,z)^{2k}) \\
& = z_ia_i((z^\beta)^k - \phi_+(z^\beta,z)^k) \\
&+ z_ib_i  ((z^\beta)^{2k}-\phi_+(z^\beta,z)^{2k}).
\end{align*}
but since $\phi_+ (u,z)= u + \varphi_{0,+}(u,z)$ with $\varphi_{0,+}\in \B^{\infty}$ flat, we obtain that
\[ G_i(z) - z_i (\lambda_i + a_i \phi_+(z^\beta,z)^k + b_i  \phi_+(z^\beta,z)^{2k}) = \varepsilon_i(z)\in B_{\pi}^{\infty} \]
is flat with respect to $z^{\beta}$ uniformly in $z$. Finally, the $i$th component of \eqref{conjugg} can be written as, for $i=1,...,m-1$:
\begin{equation}\label{mainequi}
 \varepsilon_i(z) + K_i(z) + \varphi_{i,+}(G_\ast(z)) = \varphi_{i,+}(z)(\lambda_i + a_i \phi_+(z^\beta,z)^k + b_i  \phi_+(z^\beta,z)^{2k}). 
     \end{equation}

 Let $\tilde K_i(u,z)\in \B^\infty$ be an extension of $-\varepsilon_i(z)-K_i(z)$ and $\Lambda_i(u,z)=\lambda_i + a_i \phi_+(u,z)^k+b_i  \phi_+(u,z)^{2k}$ (since $\phi_+(u,z)=u+\varphi_{0,+}(u,z)$ with $\varphi_{0,+}$ flat, this $\Lambda_i$ is of the form \eqref{defLambda} up to a flat term, hence Proposition \ref{Propositionconjug} applies) and let us consider the equation in the unknown $\tilde\varphi_{i,+}(u,z)$:

\begin{equation}\label{unfoldequ}
 \tilde \varphi_{i,+}(\tilde G_{\ast,0}(u,z),\tilde G_\ast(u,z)) - \tilde\varphi_{i,+}(u,z)\Lambda_i(u,z) = \tilde K_i(u,z).
\end{equation}
According to Proposition \ref{Propositionconjug}, the previous equation has a solution  $\tilde\varphi_{i,+}(u,z)\in \B^{\infty}$. Then, the  restriction of \re{unfoldequ} to $\Sigma$ reads, for all $z$ such that $z^\beta \in \delta_+^{[k_0]}(\alpha,r)$,
\begin{equation}\label{unfoldequ2}
 \tilde \varphi_{i,+}(\tilde G_{\ast,0}(z^\beta,z),\tilde G_\ast(z^{\beta},z)) - \tilde\varphi_{i,+}(z^{\beta},z)\Lambda_i(z^{\beta},z) = \tilde K_i(z^\beta,z).
\end{equation}
By definition, we have $\tilde G_\ast(z^{\beta},z)=G_\ast(z)$, and $\tilde G_\ast(z^{\beta},z)^{\beta}=G_\ast(z)^{\beta}=G_{\ast,0} (z)$. Hence, let us define $\varphi_{i,+}(z):=\tilde\varphi_{i,+}(z^{\beta},z)\in B_{\pi}^{\infty}$,
then
\[ \tilde \varphi_{i,+}(\tilde G_{\ast,0}(z^\beta,z),\tilde G_\ast(z^\beta,z)) = \tilde \varphi_{i,+}(G_\ast(z)^\beta,G_\ast(z)) = \varphi_{i,+}(G_\ast(z))\]
and equation \eqref{unfoldequ2} becomes
\[ \varphi_{i,+}(G_\ast(z)) - \varphi_{i,+}(z)\Lambda_i(z^\beta,z) = \tilde K_i(z^\beta,z).\]
Since $\tilde K_i(z^\beta,z) = -\varepsilon_i(z)-K_i(z)$ and $\Lambda_i(z^\beta,z)=\lambda_i + a_i \phi_+(z^\beta,z)^k+b_i  \phi_+(z^\beta,z)^{2k}$, equation \eqref{unfoldequ2} becomes exactly \re{mainequi} and $\varphi_{i,+}$ is indeed a solution.

Applying Proposition \ref{Propositionconjug} for each $i=1,\ldots, m-1$, we obtain that $\varphi_{i,+}$ is holomorphic on $\s_+^{[k_0]}$, flat in $z^{\beta}$ uniformly in $z$ and unique, which defines holomorphic maps $\Phi_{1,+},..., \Phi_{m-1,+}$ tangent to the identity.

\textbf{Construction of $\Phi_{m,+}$: } Let us set
\[\Phi_{m,+}(z) := \left(\frac{\phi_+(z^\beta,z)}{\Phi_{1,+}(z)^{\beta_1}...\Phi_{m-1,+}(z)^{\beta_{m-1}}}\right)^\frac{1}{\beta_{m}}\]
so that equation \eqref{condphi} is automatically satisfied, that is to say
\begin{equation}\label{conjugstable}\prod_{i=1}^m\Phi_{i,+}(z)^{\beta_i} = \phi_+(z^\beta,z).\end{equation}

We first prove that $\Phi_{m,+}$ is well defined: notice that
\[\prod_{i=1}^{m-1}\Phi_{i,+}(z)^{\beta_i} = \prod_{i=1}^{m-1}(z_i + \varphi_{i,+}(z))^{\beta_i} = \frac{z^\beta}{z_m^{\beta_m}} + \sigma(z)\]
where $\sigma$ is flat with respect to $z^{\beta}$, then

\begin{equation}\label{phiplat} \frac{\phi_+(z^\beta,z)}{\Phi_{1,+}(z)^{\beta_1}...\Phi_{m-1,+}(z)^{\beta_{m-1}}} = \frac{z^\beta + \varphi_{0,+}(z ^\beta,z)}{\frac{z^\beta}{z_m^{\beta_m}} +  \sigma(z)} = z_m^{\beta_m} \left( \frac{1 + \frac{\varphi_{0,+}(z ^\beta,z)}{z^\beta}}{1 + z_m^{\beta_m}\frac{\sigma(z)}{z^\beta}}\right) = z_m^{\beta_m}(1 + o(1)). \end{equation}
Indeed, $\frac{\varphi_{0,+}(z^\beta,z)}{z^\beta}  
 \rightarrow 0$ and $\frac{z_m^{\beta_m} \sigma(z)}{z^\beta} 
 \rightarrow 0$ as $z\in \s_+^{[k_0]} \rightarrow 0 $ as $\varphi_{0,+}(z^\beta, z)$ and $\sigma(z)$ are flat with respect to $z^{\beta}$.
In the sectorial domain $\s_+^{[k_0]}$, $z_m$ does not vanish. We choose the ${\frac{1}{\beta_m}}$-root so that $\Phi_{m,+}(z)=z_m(1+o(1))$. $\Phi_{m,+}$ is well defined (recall that $\beta_m \neq 0$ by assumption) and $\Phi_{m,+}(z)-z_m\in B_{\pi}^{\infty}$. We set $\varphi_{m,+}(z) := \Phi_{m,+}(z)-z_m$.

 We now need to check that $\Phi_{m,+}$ satisfies the conjugacy equation:
 \[\Phi_{m,+}(G_\ast (z)) = G_m(\Phi_+(z)).\]
We have, on the one hand, by \eqref{conjugstable} and \eqref{equ0},
\[\prod_{i=1}^mG_i(\Phi_+(z))^{\beta_i} = G_0(\Phi_+(z)^\beta) = G_0(\phi_+(z^\beta,z)) = \phi_+(G_{\ast,0}(z),G_\ast(z)) \]
(the last equality comes from \eqref{equ0}).
On the other hand
\[ \prod_{i=1}^m\Phi_{i,+}(G_\ast(z))^{\beta_i} = \phi_+(G_\ast(z)^\beta,G_\ast(z)) = \phi_+(G_{\ast,0}(z),G_\ast(z)).\]
Combining these two equalities, we get
\begin{equation}\label{eqphi} \prod_{i=1}^m G_i(\Phi_+(z))^{\beta_i} = \prod_{i=1}^m\Phi_{i,+}(G_\ast(z))^{\beta_i}.\end{equation}
We already know that, for all $i=1,...,m-1$,
\[ G_i(\Phi_+(z))^{\beta_i} = \Phi_{i,+}(G_\ast(z))^{\beta_i}\]
so \eqref{eqphi} becomes
\[G_m(\Phi_+(z))^{\beta_m} = \Phi_{m,+}(G_\ast(z))^{\beta_m}.\]
Since on $\s_+^{[k_0]}$
\[ \Phi_{m,+}(G_\ast(z)) = G_{\ast,m}(z)+\varphi_{m,+}(G_\ast(z)) = G_m(z)+K_m(z)+\varphi_{m,+}(G_\ast(z)) = G_m(z)(1+o(1))\]
and
\[ G_{m}(\Phi_+(z)) = G_m(z)(1+o(1)),\]
then both sides of the equation are of the form $G_m(z)^{\beta_m}(1+o(1))$ and since $G_m(z) \neq 0$ on $\s_+^{[k_0]}$, we can take the $\beta_m$-root  which gives
\[G_m(\Phi_+(z)) = \Phi_{m,+}(G_\ast(z)).\]

\textbf{Construction of $\Phi_{m+1,+},\ldots,\Phi_{n,+}$:} Let 
$\Phi_{i,+}(z) = z_i + \varphi_{i,+}(z)$ for all $i=m+1,\ldots,n$. Recall that coordinate $i=m+1,...,n$ of the conjugacy equation is
\[ \Phi_{i,+}(G_\ast(z)) = G_i(\Phi_{+}(z))\]
with $G_i(z) = z_i(\lambda_i + P_i(z^{\beta}))$. Therefore, we have
\begin{align*}
\Phi_{i,+}(G_\ast(z)) & = G_{\ast,i}(z)+\varphi_{i,+}(G_\ast(z)) = G_i(z) + K_i(z) + \varphi_{i,+}(G_\ast(z)) \\
& = z_i(\lambda_i+P_i(z^{\beta}))+K_i(z) + \varphi_{i,+}(G_\ast(z)).
\end{align*}
Since $\Phi_+(z)^\beta = \phi_+(z^\beta,z)$,
$$
    G_i(\Phi_+(z)) = \Phi_{i,+}(z)(\lambda_i+P_i(\Phi_+(z)^{\beta})) = (z_i + \varphi_{i,+}(z))(\lambda_i+P_i(\phi_+(z^\beta,z))),
$$
then
\begin{align}\label{equgen}
  \Phi_{i,+}(G_\ast(z)) & = G_i(\Phi_{+}(z)) \Longleftrightarrow \\
  z_i(\lambda_i+P_i(z^{\beta}))+K_i(z) + \varphi_{i,+}(G_\ast(z)) &= (z_i + \varphi_{i,+}(z))(\lambda_i+P_i(\phi_+(z^\beta,z))). \nonumber
\end{align}
We have
\[   z_i(\lambda_i+P_i(u)) -  z_i(\lambda_i+P_i(\phi_+(u,z))) = z_i(P_i(u)-P_i(\phi_+(u,z))),  \]
 and then we can write
\[ z_i(\lambda_i+P_i(u)) -  z_i(\lambda_i+P_i(\phi_+(u,z))) = z_i \varepsilon_i(u,z),\]
where $\varepsilon_i\in \B^{\infty}$ is flat with respect to $u$, uniformly in $z$.
Let $\tilde K_i\in \B^{\infty}$ be an extension of $K_i\in B_{\pi}^{\infty}$ from \rl{lemRM}. Let us consider the functional equation of unknown $\tilde \varphi_{i,+}(u,z)$, whose restriction to $\Sigma$ will give the $ \varphi_{i,+}(z)$:
\begin{align*}
 \varepsilon_i(u,z) + \tilde K_i(u,z) + \tilde \varphi_{i,+}(\tilde G_{\ast,0}(u,z),\tilde G_\ast(u,z)) = \tilde \varphi_{i,+}(u,z)(\lambda_i + P_i(\phi_+(u,z))).
\end{align*}
It can be written as
\begin{equation}\label{equagen2} \tilde \varphi_{i,+}(\tilde G_{\ast,0}(u,z),\tilde G_\ast(u,z))-\tilde\varphi_{i,+}(u,z)\Lambda_i(u,z) = \tilde L_i(u,z)\end{equation}
with $\Lambda_i(u,z) = \lambda_i + P_i(\phi_+(u,z))$ and $\tilde L_i(u,z) = -\varepsilon_i(u,z)-\tilde K_i(u,z)$ is flat in $u$ uniformly in $z$. Notice that, since $\phi_+(u,z)=u+\varphi_{0,+}(u,z)$ with $\varphi_{0,+}$ flat, we have $\Lambda_i(u,z) = \lambda_i + a_{i,1}u^k + \OO(u^{k+1}) + \tilde\varepsilon_i(u,z)$ where $\tilde\varepsilon_i$ is flat in $u$ uniformly in $z$. As $\tilde\varepsilon_i$ is flat, it is bounded by $C\vert u\vert^{k+1}$, so this $\Lambda_i$ is of the form \eqref{defLambda}. Applying Proposition \ref{Propositionconjug}, we obtain the solutions. As above, we define function $\varphi_{i,+}(z):= \tilde\varphi_{i,+}(z^{\beta},z)$ as the solution $\tilde\varphi_{i,+}$ restricted to $\Sigma$. It is a solution of \re{equgen}. Indeed, we have
\begin{align*}
\tilde \varphi_{i,+}(\tilde G_{\ast,0}(z^\beta,z),\tilde G_\ast(z^{\beta},z))-\tilde\varphi_{i,+}(z^{\beta},z)\Lambda_i(z^{\beta},z) &= \tilde L_i(z^{\beta},z)= -\varepsilon_i(z^{\beta},z)- K_i(z)
\end{align*}

Since $\tilde G_{\ast}(z^\beta,z) = G_\ast(z)$ and $\tilde G_{\ast,0}(z^\beta,z) = G_{\ast,0}(z)$, we get
\[ \tilde \varphi_{i,+}(\tilde G_{\ast,0}(z^\beta,z), \tilde G_{\ast}(z^\beta,z)) = \tilde \varphi_{i,+}(G_{\ast}(z)^\beta,G_\ast(z)) = \varphi_{i,+}(G_\ast(z))\]
so equation \eqref{equagen2} becomes
\[ \varphi_{i,+}(G_\ast(z)) - \varphi_{i,+}(z)\Lambda_i(z^\beta,z) = \tilde L_i(z^\beta,z)\]
and  using $\Lambda_i(z^\beta,z) = \lambda_i + P_i(\phi_+(z^\beta,z))$ and $\tilde L_i (z^\beta,z) = - \varepsilon_i(z^\beta,z)-K_i(z)$, we obtain equation \eqref{equgen}.

\textbf{Uniqueness of the sectorial conjugacy: } Let $\tilde \Phi = (\tilde \Phi_1,...,\tilde \Phi_n)$ be another conjugacy from $G_\ast$ to $G$ on $\s_+^{[k_0]}$, that is to say $\tilde \Phi \circ G_\ast = G \circ \tilde \Phi$, such that $\tilde \Phi_i(z) = z_i + \tilde \varphi_i(z) $ with each $\tilde \varphi_i$ flat at $0$ with respect to $z^\beta$.  Moreover, since $G(z)^\beta = G_{0}(z^\beta)$ (see \eqref{defG0u}), conjugacy equation at power $\beta$ gives 
\[ \tilde \Phi ( G_\ast(z))^\beta = G(\tilde \Phi(z))^\beta = G_0(\tilde \Phi(z)^\beta) \]
so that letting $\tilde \phi = \tilde \Phi^\beta$, we have $\tilde \phi(z) = z^\beta +flat$ is a solution of \eqref{equ0sigma} and by the uniqueness in the construction of $\Phi_+$ on $\Sigma$ proved above, we get $\tilde \phi(z) = \phi_+(z^\beta,z)$, that is to say $\tilde \Phi(z)^\beta = \phi_+(z^\beta,z)$ and $\tilde \Phi$ satisfies \eqref{condphi}. Hence the construction of the conjugacy used in \eqref{mainequi} for $i=1,...,m-1$ and in \eqref{equgen} for $i=m+1,...,n$ apply in the same way to $\tilde \Phi$, and each $\tilde \varphi_i$ satisfies the same equation as $\varphi_{i,+}$ and is flat at $0$ with respect to $z^\beta$, so that $\tilde \Phi_i \equiv \Phi_{i,+}$ for all $i=1,...,m-1,m+1,...,n$. 

Finally, since \eqref{condphi} is satisfied for $\tilde \Phi$ and $\Phi_+$, we get that on $\s_+^{[k_0]}$
\[ \prod_{i=1}^m \tilde \Phi_i(z)^{\beta_i} = \phi_+(z^\beta,z)= \prod_{i=1}^m \Phi_{i,+}(z)^{\beta_i}\]
and since $\tilde \Phi_i = \Phi_{i,+}$ for $i=1,...,m-1$ do not vanish on $\s_+^{[k_0]}$, this gives
\[ \tilde \Phi_m(z)^{\beta_m} = \Phi_{m,+}(z)^{\beta_m}.\]
These two quantities are of the form $z_m^{\beta_m}(1+o(1))$ so that taking the $\beta_m$-root as above, we get $\tilde \Phi_m = \Phi_{m,+}$. Finally $\tilde \Phi = \Phi_+$ which concludes the proof.

The above construction is for a fixed $k_0 \in \{0, \dots,k-1\}$ and for the sectors $\delta_+^{[k_0]}(\alpha,r)$. This then gives $k$ sectorial conjugacies $\Phi_+^{[k_0]}$. The same argument applied to the sectors $\delta_-^{[k_0]}(\alpha,r)$ (on $G^{-1}$ and $G_\ast^{-1}$) produces $k$ sectorial conjugacies $\Phi_-^{[k_0]}$. With the notation \eqref{sector}, we have $S_{2k_0} = \delta_+^{[k_0]}(\alpha,r)$ and $S_{2k_0+1} = \delta_-^{[k_0]}(\alpha,r)$ for $k_0=0,...,k-1$, so that $\Phi_+^{[k_0]}$ and $\Phi_-^{[k_0]}$ give $2k$ conjugacies $\Psi_i$ of Theorem \ref{main-sect}. In total, we obtain $2k$ sectorial domains covering $\D_\rho^n \setminus E$ (where $E$ was defined in \eqref{defE}). This completes the proof of Theorem \ref{main-sect}.

\section{Holomorphic classification}
In this section, we prove \rt{classif}. The argument is standard (e.g. \cite{Malgrange-bourbaki, Ram-Mart1, Ram-Mart2, Stolo-classif}). 

\subsection{Proof of Theorem \ref{classif}}
Let $F_1,F_2$ be two germs of biholomorphisms of the form \re{system3} having the same normal form $\hat G$ as in \re{system5}. Without loss of generality we can assume that $F_1$ and $F_2$ are defined on a same neighborhood of the origin and that they can be normalized to $\hat G$ on the same domains ${\mathcal D}_i:= \{z\in\C^n,\,z^{\beta}\in S_i\}\cap \D^n_\rho$,  with sectors $S_i$ as defined in \re{sector}. Let $\{\Psi_{j,i}\}_i$ be the associated collection of normalization transformations given by \rt{main-sect}, $j=1,2$: $\Psi_{j,i}\circ F_j=\hat G\circ \Psi_{j,i}$. We denote by $\hat \Phi_j$ the common asymptotic expansion of the $\Psi_{j,i}$ along $E$ (see Lemma \ref{phiholo}).
Let us consider the associated cocycle $C_j:=\{\Psi_{j,i}\circ \Psi_{j,i+1}^{-1}\}=:\{C_{j,i}\}$. Each $C_{j,i}$ is tangent to the identity and commutes with $\hat G$ on ${\mathcal D}_i\cap {\mathcal D}_{i+1}$ as 
$$
\hat G\circ \Psi_{j,i}\circ\Psi_{j,i+1}^{-1}= \Psi_{j,i}\circ F_j\circ \Psi_{j,i+1}^{-1}= \Psi_{j,i}\circ \Psi_{j,i+1}^{-1}\circ \hat G.
$$
Assume first that there exist $T_i$ as in Theorem \ref{classif}. The asymptotic expansion $\hat T$ commutes with $\hat G$, so that by Lemma \ref{centralisateur2}, $\hat T(z)^\beta$ is divisible by $z^\beta$. Notice that $\hat \Phi_j(z)^\beta$ is also divisible by $z^\beta$, since $\hat \Phi_j(z) = z + \OO((z^\beta)^{2k +1})$. Then $\Psi_{1,i}(z)^\beta$, $T_i(z)^\beta$ and $\Psi_{2,i}^{-1}(z)^\beta$ are of the form $z^\beta(1 + \OO(\rho))$, so the argument of $z^\beta$ is modified by some $\OO(\rho)$ by $\Psi_{1,i}$, $T_i$ and $\Psi_{2,i}^{-1}$, and up to shrinking the sectors and $\rho$, the $\Psi_{2,i}^{-1}\circ T_i \circ \Psi_{1,i}$ are defined on $\mathcal{D}_i$. Since $\Psi_{j,i} = C_{j,i}\circ \Psi_{j,i+1}$ and $T_{i+1} = C_{2,i}^{-1} \circ T_i \circ C_{1,i}$, we have on $\mathcal D_i \cap \mathcal D_{i+1}$
\[\Psi_{2,i}^{-1} \circ T_i \circ \Psi_{1,i} = \Psi_{2,i+1}^{-1} \circ C_{2,i}^{-1} \circ T_i \circ C_{1,i} \circ \Psi_{1,i+1} = \Psi_{2,i+1}^{-1} \circ T_{i+1} \circ \Psi_{1,i+1}.\]  
Since $S_i\cap S_{i+2} = \emptyset$, and $S_i\cap S_{i+1} \neq \emptyset$, two domains $\mathcal{D}_i$ and $\mathcal{D}_j$ have a nonempty intersection only when $j=i+1$ or $j=i-1$. 
Then the maps $\Psi_{2,i}^{-1}\circ T_i \circ \Psi_{1,i}$ glue into a single holomorphic map $\Psi$ on $\cup_i \mathcal{D}_i = \D_\rho^n \setminus E$, whose restriction to each $\mathcal{D}_i$ is $\Psi^{-1}_{2,i}\circ  T_i \circ \Psi_{1,i}$.
 Moreover, since $\Psi$ is bounded near $E$ (since the $\Psi_{j,i}$ and the $T_i$ admit along $E$ asymptotic expansions tangent to the identity), by the vanishing singularity theorem, $\Psi$ extends holomorphically to $\D_\rho^n$ (which is a full neighborhood of the origin in $\C^n$), with $\Psi(0)=0$ and tangent to the identity at the origin since its expansion $\hat \Phi_2^{-1} \circ \hat T \circ \hat \Phi_1$ is tangent to the identity at the origin, and it is then a biholomorphism of $(\C^n,0)$. 
For $z\in \mathcal{D}_i$ such that $F_1(z) \in \mathcal{D}_i$, since $T_i$ commutes with $\hat G$,
\[\Psi \circ F_1 = \Psi_{2,i}^{-1} \circ T_i \circ \hat G \circ \Psi_{1,i} = \Psi_{2,i}^{-1} \circ \hat G \circ  T_i \circ \Psi_{1,i} = F_2\circ \Psi_{2,i}^{-1} \circ T_i \circ \Psi_{1,i} = F_2 \circ \Psi,\]
 hence, by analytic continuation, $\Psi\circ F_1=F_2\circ\Psi$ on $\D_\rho^n$.
In other words, $F_1$ and $F_2$ are holomorphically conjugate by $\Psi$ in a neighborhood of the origin. 

Assume now that $\Theta \circ F_1 = F_2 \circ \Theta$ for some biholomorphism $\Theta$ of $(\C^n,0)$ tangent to the identity there.  
The formal series $\hat T := \hat \Phi_2 \circ \Theta \circ \hat \Phi_1^{-1}$ is tangent to the identity and commutes with $\hat G$, since $\hat \Phi_j \circ F_j = \hat G \circ \hat \Phi_j$. By Lemma \ref{centralisateur2}, $\hat T(z)^\beta$ is divisible by $z^\beta$, and as noticed above $\hat \Phi_j(z)^\beta$ is also divisible by $z^\beta$. Since $\Theta = \hat \Phi_2^{-1} \circ \hat T \circ \hat \Phi_1$, then $\Theta(z)^\beta = z^\beta v(z)$, with $v$ holomorphic on a neighborhood of $0$ with $v(0)=1$. Hence, on $\D_\rho^n$, $\Theta$ modifies the argument of $z^\beta$ by $\OO(\rho)$, and once again, up to shrinking the sectors and $\rho$, we can let $T_i := \Psi_{2,i}\circ \Theta \circ \Psi_{1,i}^{-1}$ on $\mathcal{D}_i$. We have 
\[T_i \circ \hat G = \Psi_{2,i} \circ \Theta \circ  F_1 \circ \Psi_{1,i}^{-1} = \Psi_{2,i} \circ F_2 \circ \Theta \circ\Psi_{1,i}^{-1} = \hat G \circ T_i,\]
and $T_i$ has asymptotic expansion $\hat T$ along $E$, which is tangent to the identity and does not depend on $i$. Finally, on $\mathcal{D}_i \cap \mathcal{D}_{i+1}$,
\[T_i \circ C_{1,i} \circ T_{i+1}^{-1} = \Psi_{2,i} \circ \Theta \circ \Psi_{1,i}^{-1} \circ \Psi_{1,i} \circ \Psi_{1,i+1}^{-1} \circ \Psi_{1,i+1} \circ \Theta^{-1} \circ \Psi_{2,i+1}^{-1} = C_{2,i}.\]
Suppose moreover that $\Theta(z)^\beta - z^\beta$ is divisible by $(z^\beta)^{k+2}$. As $\hat \Phi_1^{-1}(z)= z + \OO((z^\beta)^{2k + 1})$, we get $(\hat \Phi_1^{-1}(z))^\beta = z^\beta + \OO((z^\beta)^{2k + 1})$. Since $2k  + 1 \geq k+2$, and since we assumed $\Theta(z)^\beta - z^\beta$ is divisible by $(z^\beta)^{k+2}$, we get  
\[ (\Theta \circ \hat \Phi_1^{-1}(z))^\beta = z^\beta + \OO ((z^\beta)^{k+2}).\]
Moreover, since $\hat \Phi_2 (z) = z + \OO((z^\beta)^{2k + 1})$, we get
\[\hat T(z)^\beta = (\hat \Phi_2 \circ \Theta \circ \hat \Phi_1^{-1}(z))^\beta = z^\beta + \OO ((z^\beta)^{k+2}),\]
that is to say $c_{\hat T} = 0$. Lemma \ref{centralisateur2} then gives $\hat T = id$. Finally, we get that each $T_i$ is a flat perturbation of the identity which commutes with $\hat G$ on $\mathcal{D}_i$. By the uniqueness of the sectorial conjugacy proved in Section 4, we have $T_i = id$ for all $i$, hence $C_1 = C_2$.

\subsection{Exponentially flat perturbations}

\begin{prop}\label{expoflat}
    Up to shrinking $\rho$ and $r$, for all $j=0,...,2k-1$, the cocycle $C = \Psi_j \circ \Psi_{j+1}^{-1}$ can be written $C= id + \varphi$ on $\mathcal{D}_j \cap \mathcal{D}_{j+1}$, with for all $l=1,...,n$
    \[ \vert \varphi_l(z) \vert \leq C_4 e^{-\frac{M_4}{\vert z^\beta \vert^{\frac{k}{\tau +1}}}}\]
    for some constants $C_4,M_4 > 0$, where $\tau$ is given by the Diophantine condition \ref{dio}.
\end{prop}

The proof follows Section 6.3 of \cite{Loday}. We will also use ideas from \cite{bracci-raissy-stensones}, which uses Fatou coordinates and linearization techniques. We first linearize the dynamics of $G$ on the sectors, by means of a Fatou coordinate $\V$ for $G_0$ and of functions $\tilde z_l$ satisfying $\tilde z_l \circ G = \lambda_l \tilde z_l$. In these functions, the invariance equation $C \circ G = G \circ C$ can be decomposed in equations $g(V+1)=\mu g(V)$ with $\vert \mu \vert=1$, which are solved as in \cite{Loday}. The small divisors $\mu-1$ appearing naturally in the summation over the monomials $\tilde z^Q$ are controlled by the Diophantine condition.

\begin{proof}
Following \cite{Loday}, we will study the invariance equation
\begin{equation}\label{invariance} C \circ G = G \circ C\end{equation}
 on $\mathcal{D}_j \cap \mathcal{D}_{j+1}$, where the solution $C$ is of the form $C = id + \varphi$ with $\varphi=(\varphi_1,...,\varphi_n)$ and each $\varphi_i$ flat in $z^\beta$ for all $i$. 

 Since by definition $\mathcal{D}_j\cap \mathcal{D}_{j+1} = \{z \in \C^n, z^\beta \in S_j\cap S_{j+1}\}\cap\D_\rho^n$, we will work in the variable $u=z^\beta$ on the sector $S_j\cap S_{j+1}$. Notice that working at $0$ in $S_j\cap S_{j+1}$ is equivalent to working at infinity in the variable $U = u^{-k}$ on the sector $\Delta_+(\alpha,R)\cap \Delta_-(\alpha,R)$, which is a union of two sectors at infinity along the imaginary axis, non empty since $\alpha > \frac{\pi}{2}$ (see Figure \ref{intersectionsector}).

 \begin{figure} 
 
\tikzset{
pattern size/.store in=\mcSize, 
pattern size = 5pt,
pattern thickness/.store in=\mcThickness, 
pattern thickness = 0.3pt,
pattern radius/.store in=\mcRadius, 
pattern radius = 1pt}
\makeatletter
\pgfutil@ifundefined{pgf@pattern@name@_7aj1fgkh3}{
\pgfdeclarepatternformonly[\mcThickness,\mcSize]{_7aj1fgkh3}
{\pgfqpoint{0pt}{0pt}}
{\pgfpoint{\mcSize+\mcThickness}{\mcSize+\mcThickness}}
{\pgfpoint{\mcSize}{\mcSize}}
{
\pgfsetcolor{\tikz@pattern@color}
\pgfsetlinewidth{\mcThickness}
\pgfpathmoveto{\pgfqpoint{0pt}{0pt}}
\pgfpathlineto{\pgfpoint{\mcSize+\mcThickness}{\mcSize+\mcThickness}}
\pgfusepath{stroke}
}}
\makeatother

 
\tikzset{
pattern size/.store in=\mcSize, 
pattern size = 5pt,
pattern thickness/.store in=\mcThickness, 
pattern thickness = 0.3pt,
pattern radius/.store in=\mcRadius, 
pattern radius = 1pt}
\makeatletter
\pgfutil@ifundefined{pgf@pattern@name@_o088jfiod}{
\pgfdeclarepatternformonly[\mcThickness,\mcSize]{_o088jfiod}
{\pgfqpoint{0pt}{-\mcThickness}}
{\pgfpoint{\mcSize}{\mcSize}}
{\pgfpoint{\mcSize}{\mcSize}}
{
\pgfsetcolor{\tikz@pattern@color}
\pgfsetlinewidth{\mcThickness}
\pgfpathmoveto{\pgfqpoint{0pt}{\mcSize}}
\pgfpathlineto{\pgfpoint{\mcSize+\mcThickness}{-\mcThickness}}
\pgfusepath{stroke}
}}
\makeatother
\tikzset{every picture/.style={line width=0.75pt}} 

\begin{tikzpicture}[x=0.75pt,y=0.75pt,yscale=-1,xscale=1]

\draw    (160,210) -- (467,210) ;
\draw [shift={(470,210)}, rotate = 180] [fill={rgb, 255:red, 0; green, 0; blue, 0 }  ][line width=0.08]  [draw opacity=0] (8.93,-4.29) -- (0,0) -- (8.93,4.29) -- cycle    ;
\draw    (300,210) ;
\draw [shift={(300,210)}, rotate = 0] [color={rgb, 255:red, 0; green, 0; blue, 0 }  ][fill={rgb, 255:red, 0; green, 0; blue, 0 }  ][line width=0.75]      (0, 0) circle [x radius= 3.35, y radius= 3.35]   ;
\draw [shift={(300,210)}, rotate = 0] [color={rgb, 255:red, 0; green, 0; blue, 0 }  ][fill={rgb, 255:red, 0; green, 0; blue, 0 }  ][line width=0.75]      (0, 0) circle [x radius= 3.35, y radius= 3.35]   ;
\draw  [dash pattern={on 0.84pt off 2.51pt}]  (301,291.5) -- (370,210) ;
\draw [shift={(370,210)}, rotate = 310.25] [color={rgb, 255:red, 0; green, 0; blue, 0 }  ][fill={rgb, 255:red, 0; green, 0; blue, 0 }  ][line width=0.75]      (0, 0) circle [x radius= 3.35, y radius= 3.35]   ;
\draw  [dash pattern={on 0.84pt off 2.51pt}]  (301,291.5) -- (231.3,211.56) ;
\draw [shift={(231.3,211.56)}, rotate = 228.92] [color={rgb, 255:red, 0; green, 0; blue, 0 }  ][fill={rgb, 255:red, 0; green, 0; blue, 0 }  ][line width=0.75]      (0, 0) circle [x radius= 3.35, y radius= 3.35]   ;
\draw    (387,32) -- (301,132.5) ;
\draw  [dash pattern={on 0.84pt off 2.51pt}]  (301,132.5) -- (231.3,211.56) ;
\draw  [dash pattern={on 0.84pt off 2.51pt}]  (301,132.5) -- (370,210) ;
\draw    (212,31) -- (301,132.5) ;
\draw  [draw opacity=0][fill={rgb, 255:red, 184; green, 233; blue, 134 }  ,fill opacity=0.36 ] (387,32) -- (387,32) -- (301,132.5) -- (301,132.5) -- (301,132.5) -- (212,31) -- cycle ;
\draw    (301,10.5) -- (300,210) ;
\draw  [draw opacity=0][fill={rgb, 255:red, 184; green, 233; blue, 134 }  ,fill opacity=0.36 ] (276,319) -- (301,291.5) -- (301,291.5) -- (301,291.5) -- (325,320) -- cycle ;
\draw    (276,319) -- (301,291.5) ;
\draw    (301,291.5) -- (325,320) ;
\draw  [pattern=_7aj1fgkh3,pattern size=6pt,pattern thickness=0.75pt,pattern radius=0pt, pattern color={rgb, 255:red, 0; green, 0; blue, 0}] (301,30) -- (320,30) -- (320,90) -- (301,90) -- cycle ;
\draw  [dash pattern={on 4.5pt off 4.5pt}]  (320,80) -- (320,210) ;
\draw  [draw opacity=0] (343.57,182.58) .. controls (348.11,179.02) and (353.73,176.73) .. (359.93,176.29) .. controls (376.45,175.12) and (390.8,187.57) .. (391.97,204.1) .. controls (392.12,206.27) and (392.04,208.4) .. (391.75,210.46) -- (362.04,206.22) -- cycle ; \draw    (346.05,180.83) .. controls (350.1,178.28) and (354.82,176.66) .. (359.93,176.29) .. controls (376.45,175.12) and (390.8,187.57) .. (391.97,204.1) .. controls (392.12,206.27) and (392.04,208.4) .. (391.75,210.46) ;  \draw [shift={(343.57,182.58)}, rotate = 333.44] [fill={rgb, 255:red, 0; green, 0; blue, 0 }  ][line width=0.08]  [draw opacity=0] (10.72,-5.15) -- (0,0) -- (10.72,5.15) -- (7.12,0) -- cycle    ;
\draw  [draw opacity=0] (254.96,184.42) .. controls (262.17,183.74) and (269.23,186.03) .. (274.25,191.41) .. controls (279.05,196.54) and (281.07,203.49) .. (280.46,210.57) -- (253.83,211.83) -- cycle ; \draw [color={rgb, 255:red, 208; green, 2; blue, 27 }  ,draw opacity=1 ]   (257.97,184.31) .. controls (264.11,184.44) and (269.93,186.79) .. (274.25,191.41) .. controls (279.05,196.54) and (281.07,203.49) .. (280.46,210.57) ;  \draw [shift={(254.96,184.42)}, rotate = 8.37] [fill={rgb, 255:red, 208; green, 2; blue, 27 }  ,fill opacity=1 ][line width=0.08]  [draw opacity=0] (10.72,-5.15) -- (0,0) -- (10.72,5.15) -- (7.12,0) -- cycle    ;
\draw [color={rgb, 255:red, 144; green, 19; blue, 254 }  ,draw opacity=1 ]   (350,50) -- (300,210) ;
\draw [shift={(350,50)}, rotate = 107.35] [color={rgb, 255:red, 144; green, 19; blue, 254 }  ,draw opacity=1 ][fill={rgb, 255:red, 144; green, 19; blue, 254 }  ,fill opacity=1 ][line width=0.75]      (0, 0) circle [x radius= 3.35, y radius= 3.35]   ;
\draw  [draw opacity=0] (308.81,183.71) .. controls (316.02,183.03) and (323.08,185.32) .. (328.1,190.7) .. controls (332.89,195.83) and (334.92,202.78) .. (334.31,209.86) -- (307.68,211.12) -- cycle ; \draw [color={rgb, 255:red, 144; green, 19; blue, 254 }  ,draw opacity=1 ]   (311.82,183.6) .. controls (317.95,183.73) and (323.78,186.08) .. (328.1,190.7) .. controls (332.89,195.83) and (334.92,202.78) .. (334.31,209.86) ;  \draw [shift={(308.81,183.71)}, rotate = 8.37] [fill={rgb, 255:red, 144; green, 19; blue, 254 }  ,fill opacity=1 ][line width=0.08]  [draw opacity=0] (10.72,-5.15) -- (0,0) -- (10.72,5.15) -- (7.12,0) -- cycle    ;
\draw [color={rgb, 255:red, 144; green, 19; blue, 254 }  ,draw opacity=1 ] [dash pattern={on 0.84pt off 2.51pt}]  (350,50) -- (350,210) ;
\draw  [color={rgb, 255:red, 128; green, 128; blue, 128 }  ,draw opacity=1 ][pattern=_o088jfiod,pattern size=12.75pt,pattern thickness=0.75pt,pattern radius=0pt, pattern color={rgb, 255:red, 128; green, 128; blue, 128}] (256.5,31.75) .. controls (313.5,35.75) and (331.12,30.43) .. (360.04,29.72) .. controls (388.96,29.01) and (397.48,30.41) .. (389,38) .. controls (380.52,45.59) and (332.5,103.25) .. (320,120) .. controls (307.5,136.75) and (300.5,128.25) .. (260,80) .. controls (219.5,31.75) and (199.5,27.75) .. (256.5,31.75) -- cycle ;
\draw [color={rgb, 255:red, 128; green, 128; blue, 128 }  ,draw opacity=1 ]   (392.09,40.49) .. controls (413.84,45.61) and (411.54,47.22) .. (429,38) ;
\draw [shift={(390,40)}, rotate = 12.91] [fill={rgb, 255:red, 128; green, 128; blue, 128 }  ,fill opacity=1 ][line width=0.08]  [draw opacity=0] (12,-3) -- (0,0) -- (12,3) -- cycle    ;

\draw (377,211.62) node [anchor=north west][inner sep=0.75pt]  [font=\footnotesize]  {$R$};
\draw (297,216.4) node [anchor=north west][inner sep=0.75pt]  [font=\footnotesize]  {$0$};
\draw (121,102.4) node [anchor=north west][inner sep=0.75pt]  [color={rgb, 255:red, 65; green, 117; blue, 5 }  ,opacity=1 ]  {$ \begin{array}{l}
\Delta _{+}( \alpha ,R) \cap \Delta _{-}( \alpha ,R)\\
\end{array}$};
\draw (215,220.62) node [anchor=north west][inner sep=0.75pt]  [font=\footnotesize]  {$-R$};
\draw (316,212.4) node [anchor=north west][inner sep=0.75pt]  [font=\footnotesize]  {$1$};
\draw (389.04,172.4) node [anchor=north west][inner sep=0.75pt]  [font=\footnotesize]  {$\alpha $};
\draw (281,182.4) node [anchor=north west][inner sep=0.75pt]  [font=\footnotesize,color={rgb, 255:red, 208; green, 2; blue, 27 }  ,opacity=1 ]  {$\beta $};
\draw (351,32.4) node [anchor=north west][inner sep=0.75pt]  [color={rgb, 255:red, 144; green, 19; blue, 254 }  ,opacity=1 ]  {$U$};
\draw (336,186.4) node [anchor=north west][inner sep=0.75pt]  [font=\footnotesize,color={rgb, 255:red, 144; green, 19; blue, 254 }  ,opacity=1 ]  {$\gamma $};
\draw (353,102.4) node [anchor=north west][inner sep=0.75pt]  [color={rgb, 255:red, 144; green, 19; blue, 254 }  ,opacity=1 ]  {$\Im U$};
\draw (323,132.4) node [anchor=north west][inner sep=0.75pt]  [color={rgb, 255:red, 144; green, 19; blue, 254 }  ,opacity=1 ]  {$|U|$};
\draw (381,62.4) node [anchor=north west][inner sep=0.75pt]    {$ \begin{array}{l}
\Im U=|U|\sin \gamma \geq |U|\sin \beta \\
\ \ 
\end{array}$};
\draw (431,22.4) node [anchor=north west][inner sep=0.75pt]  [color={rgb, 255:red, 128; green, 128; blue, 128 }  ,opacity=1 ]  {$ \begin{array}{l}
\mathcal{V}( \Delta _{+}( \alpha ,R) \cap \Delta _{-}( \alpha ,R))\\
\end{array}$};

\end{tikzpicture}
    \caption{Intersection of sectors}
    \label{intersectionsector}
\end{figure}

 \textbf{A conjugation to a translation.} We will first study the component $G_0(z^\beta) = G(z)^{\beta}$ (see \eqref{defG0u}). We recall that, in the variable $U=u^{-k}$, $G_0$ reads $G_0(U)=U+1+E(\frac{1}{U})$ (see \eqref{G0inf}). In \cite{Loday}, the one-dimensional map is a translation at infinity. In our case, $G_0$ is a translation plus a term of the form $E(\frac{1}{U})$. We will first find a biholomorphism $\V$ from $\Delta_+(\alpha,R)$ onto its image, such that $\V (G_0(U)) = \V(U)+1$ which reduces the situation to the case of \cite{Loday}.

For all $U \in \Delta_+(\alpha,R)$, let us define $\V_1(U) = U + c\ln U$ for some constant $c$ (to be chosen later). Notice that the logarithm is well defined since $\alpha < \frac{3\pi}{4}$ and $R$ is large, so that $\R_-$ does not meet $\Delta_+(\alpha,R)$. Then
\begin{align*}
    \V_1(G_0(U))-\V_1(U)-1 & = G_0(U)+c \ln(G_0(U)) - U- c\ln(U)-1 \\
    & =1 + U + E(\frac{1}{U})+c \ln(\frac{G_0(U)}{U}) - 1 -U\\
    & = E(\frac{1}{U})+c \ln(\frac{1+U+E(\frac{1}{U})}{U}) \\
    & = \frac{E'(0)}{U}+\OO(\frac{1}{U^2}) + c\frac{1+E(\frac{1}{U})}{U} + \OO(\frac{1}{U^2}) \\
    & = \frac{E'(0)}{U} + \frac{c}{U} + \OO(\frac{1}{U^2}),
\end{align*}
 and letting $c:= -E'(0)$, we get
\[ \V_1(G_0(U)) = \V _1(U)+1 + \OO(\frac{1}{U^2}).\]
We will remove the $\OO(\frac{1}{U^2})$ terms, where $\rho(U) := \V_1(G_0(U))-\V_1(U)-1.$ 

Let now $\V(U) = \V _1(U)+ \omega(U)$. We have
\begin{align*}
    \V (G_0(U))-\V(U)-1 &= \V_1(G_0(U))+\omega(G_0(U))-\V_1(U)-\omega(U)-1 \\
    & = \rho(U)+\omega(G_0(U))-\omega(U).
\end{align*}
As we aim to compensate $\rho$ with $\omega$, we want to find $\omega$ satisfying 
\[\omega(G_0(U))-\omega(U)=-\rho(U).\]

Let us set formally $\omega(U) := \sum_{i\geq 0}\rho(G_0^{(i)}(U))$, then we have
\[ \omega(G_0(U)) -\omega(U) = \sum_{i\geq 1}\rho(G_0^{(i)}(U)) - \sum_{i\geq0}\rho(G_0^{(i)}(U)) = -\rho(U).\]
Let us show that $\omega$ is well defined and holomorphic on $\Delta_+(\alpha,R)$.
Since $\rho(U)=\OO(\frac{1}{U^2})$, there exists $c_1>0$ such that, by \eqref{estimU+p}, for all $p\geq 0$

\[\vert \rho(G_0^{(p)}(U)) \vert \leq  \frac{c_1}{\vert G_0^{(p)}(U) \vert^2} \leq \left(\frac{4}{3}\right)^2\frac{c_1}{\vert U +p\vert^2},\]
and from Lemma \ref{somme}, we obtain
\begin{equation} \label{estimomega}\vert \omega(U)\vert \leq \frac{16c_1}{9}\sum_{p\geq 0} \frac{1}{\vert U +p\vert^2}\leq \frac{16c_1}{9}\frac{A_1}{\vert U \vert}\end{equation}
and the series defining $\omega$ converges normally on $\Delta_+(\alpha,R)$ using \eqref{minz02}, so that $\omega$ is well defined and holomorphic on $\Delta_+(\alpha,R)$. Moreover, by construction, we have
\[ \V(G_0(U))-\V(U)-1 = \rho(U)+\omega(G_0(U)) - \omega(U) = 0\]
and finally, for all $U \in \Delta_+(\alpha,R)$,
\[ \V(G_0(U))=\V(U)+1.\]

Notice that $\V(U) = U -E'(0)\ln U + \omega(U)$ with $\omega(U) = \OO(\frac{1}{\vert U\vert})$ from estimate \eqref{estimomega}, so that \begin{equation}\label{equationVU}\V(U)=U + \OO(\ln \vert U\vert).\end{equation}

We constructed $\V$ defined on $\Delta_+(\alpha,R)$, which conjugates $G_0$ to the translation $T$, that is to say
\[ \V (G_0(U)) = \V(U)+1 =: T(\V(U)), \quad T(V)=1+V\]
and then
\[ \V \circ G_0 = T \circ \V.\]
Since $\mathcal{D}_j \cap \mathcal{D}_{j+1} \subset \{z \in \C^n, z^\beta \in \delta_+^{[k_0]}(\alpha,r) \}$ for some $k_0$, $\V((z^\beta)^{-k})$ is well defined for $z \in \mathcal{D}_j \cap \mathcal{D}_{j+1}$. We are now ready to study the flatness of the $\varphi_i$.

\textbf{Linearizing functions.} For $l=1,...,n$, let $\kappa_l(u):= \ln ( \frac{\Lambda_l(u)}{\lambda_l})$ where the $\Lambda_l$ are defined in \eqref{defLambda} and depend only on $u$. Since $\Lambda_l(u) = \lambda_l + a_{l,1}u^k + \OO(u^{k+1})$, we have $\kappa_l(u)=\frac{a_{l,1}}{\lambda_l}u^k+\OO(u^{k+1})$. 
Let $k_0$ such that $S_j \cap S_{j+1} \subset \delta_+^{[k_0]}(\alpha,r)$, then for $U \in \Delta_+(\alpha,R)$,
\begin{equation}\label{defkappa}
    \kappa_l(U^{-\frac{1}{k}})=\frac{c_l}{U}+ \OO(U^{-1-\frac{1}{k}}),
\end{equation}
with $c_l:=\frac{a_{l,1}}{\lambda_l}$.
We want to find a holomorphic function $\Theta_l$ such that 
\[\Theta_l(G_0(U))-\Theta_l(U) = \kappa_l(U^{-\frac{1}{k}}).\]
Using the same idea as for the construction of $\V$, let $\Theta_{l,1}(U) = c_l \ln U$. Since $G_0(U) = 1+U + E(\frac{1}{U})$, we have
\[ \Theta_{l,1}(G_0(U)) - \Theta_{l,1}(U) = c_l \ln \big(1+ \frac{1}{U} +\frac{E(\frac{1}{U})}{U} \big)  = \frac{c_l}{U} + \OO(\frac{1}{U^2})\]
so that
\[ \sigma_l(U) := \kappa_l(U^{-\frac{1}{k}}) - (\Theta_{l,1}(G_0(U))-\Theta_{l,1}(U))= \OO(U^{-1-\frac{1}{k}})- \OO(U^{-2}) = \OO(U^{-1-\frac{1}{k}}).\]
We now want to find $\Theta_{l,2}$ such that 
\[ \Theta_{l,2}(G_0(U)) - \Theta_{l,2}(U) = \sigma_l(U).
\]
As for the construction of $\omega$, let us set formally
\[ \Theta_{l,2}(U) := -\sum_{p\geq 0} \sigma_l(G_0^{(p)}(U)),\]
so that $\Theta_{l,2}(G_0(U))-\Theta_{l,2}(U) = \sigma_l(U)$.
From estimate \eqref{estimU+p}, we have $\vert G_0^{(p)}(U) \vert \geq \frac{3}{4}\vert U +p\vert$, and then for some constant $C$,
\[ \vert \sigma_l(G_0^{(p)}(U))\vert \leq \frac{C}{\vert U + p\vert^{1+\frac{1}{k}}},\]
which gives, using again \eqref{minz02}, normal convergence on $\Delta_+(\alpha,R)$ and that $\Theta_{l,2}$ is then bounded and holomorphic.
Letting $\Theta_l := \Theta_{l,1} + \Theta_{l,2}$, we obtain a holomorphic function on $\Delta_+(\alpha,R)$ satisfying
\[\Theta_l(G_0(U))-\Theta_l(U)=\kappa_l(U^{-\frac{1}{k}}).\]

Writing $c_l \ln U = (\Re c_l + i \Im c_l)(\ln \vert U \vert + i \arg U)$, we get
\begin{equation}\label{encadrementthetal}\vert e^{-\Theta_l(U)}\vert= \vert e^{-\Theta_{l,2}(U)}\vert \vert U\vert^{-\Re c_l}\,e^{\Im c_l\arg U}.\end{equation}
Let $B_l := \sup_{\Delta_+(\alpha,R)} \vert \Theta_{l,2}\vert < \infty$ since $\Theta_{l,2}$ is bounded, so that \[e^{-B_l} \leq \vert e^{-\Theta_{l,2}(U)}\vert \leq e^{B_l}.\] Moreover, on $\Delta_+(\alpha,R)$, $\vert \arg U \vert < \alpha < \pi$, so that
\[ e^{-\pi \vert \Im c_l \vert} \leq e^{\Im c_l \arg U} \leq e^{\pi \vert \Im c_l \vert}.  \]
Finally, since $\Re c_l = \Re (\frac{a_{l,1}}{\lambda_l}) = - \frac{\nu_l}{k}$ (recall that $A = -\frac{1}{k}$), we have $\vert U \vert^{-\Re c_l} = \vert U \vert^{\frac{\nu_l}{k}}$. Combining these estimates with \eqref{encadrementthetal} and recalling that $\nu_l=\Re\big(\frac{a_{l,1}}{\lambda_l A}\big)$, we get, letting $C_l := e^{B_l + \pi \vert \Im c_l\vert}$,
\begin{equation}\label{estimThetaU}C_l^{-1} \vert U \vert^{\frac{\nu_l}{k}} \leq \vert e^{-\Theta_l(U)} \vert \leq C_l \vert U \vert^{\frac{\nu_l}{k}}.\end{equation}

For $u$ in the component of $\{u \in \C, u^{-k} \in \Delta_+(\alpha,R)\}$ containing $\delta_+^{[k_0]}(\alpha,r)$ (see Figure \ref{margin}), we let $\theta_l(u) := \Theta_l(u^{-k})$. As by \eqref{G0inf}, $G_0(u)^{-k} = G_0(U)$ for $U=u^{-k}$, we have
\[ \theta_l(G_0(u)) - \theta_l(u) = \kappa_l(u),\]
and since $\vert U \vert^{\frac{\nu_l}{k}} = \vert u \vert^{-\nu_l}$, \eqref{estimThetaU} reads
\begin{equation}\label{estimthetaj}
C_l^{-1} \vert u \vert^{-\nu_l} \leq \vert e^{-\theta_l(u)}\vert \leq C_l \vert u \vert^{-\nu_l}.    
\end{equation}

For $z \in \mathcal D_j \cap \mathcal D_{j+1}$, and $l=1,...,n$, let 
\begin{equation}\label{deftildez} \tilde z_l(z) := z_l e^{-\theta_l(z^\beta)}.\end{equation}
Recall that $G_l(z) = z_l \Lambda_l(z^\beta)$ and $G(z)^\beta = G_0(z^\beta)$. Hence
\[ \tilde z_l (G(z)) = z_l \Lambda_l(z^\beta) e^{-\theta_l(G_0(z^\beta))} = z_l \Lambda_l(z^\beta)e^{-\theta_l(z^\beta)}e^{-\kappa_l(z^\beta)} = z_l \Lambda_l (z^\beta)e^{-\theta_l(z^\beta)}\frac{\lambda_l}{\Lambda_l(z^\beta)}\]
that is to say
\begin{equation}\label{linztilde}
\tilde z_l(G(z)) = \lambda_l \tilde z_l(z),
\end{equation} 
and letting $U = (z^\beta)^{-k}$, we get $(G(z)^\beta)^{-k} = G_0(z^\beta)^{-k}=G_0(U)$, and then $\V((G(z)^\beta)^{-k}) = \V(U)+1$. We obtain a linearization of the dynamics of $G$
\[ \big(\V((G(z)^\beta)^{-k}), \tilde z(G(z)) \big)= \big(\V(G_0(U)), \tilde z(G(z)) \big) = \big( \V(U)+1, \lambda_1 \tilde z_1(z), ..., \lambda_n \tilde z_n(z) \big).\]

Moreover, since $\vert z_l\vert <\rho$, by \eqref{estimthetaj}, for $z\in \mathcal{D}_j \cap \mathcal{D}_{j+1}$ and for all $l=1,...,n$,
\begin{equation}\label{estimztilde} \vert \tilde z_l(z)\vert \leq C_l \rho \vert z^\beta \vert^{-\nu_l}.\end{equation}

\textbf{Invariance equation.} Recall that the cocycle $C= id + \varphi$, with $\varphi$ flat in $z^\beta$, is defined on $\mathcal{D}_j \cap \mathcal{D}_{j+1}$ and satisfies $C \circ G = G \circ C$ for all $z \in \mathcal{D}_j \cap \mathcal{D}_{j+1}$ such that $G(z) \in \mathcal{D}_j \cap \mathcal{D}_{j+1}$ (this restriction will not be a problem later, as in the variable $V = \V(U)$, we will only use the relation on a half-strip, on which both $V$ and $V+1$ lie in $\V(\Delta_+(\alpha,R)\cap \Delta_-(\alpha,R))$, see Figure \ref{intersectionsector}). For $l=1,...,n$, we set for $U= (z^\beta)^{-k}$ and $U'=(C(z)^\beta)^{-k}$
\begin{equation} \label{defphitilde}\tilde \varphi_0(z) := \V(U')-\V(U), \quad \tilde \varphi_l(z) := \tilde z_l(C(z)) - \tilde z_l(z),\end{equation}
which are defined on $\mathcal{D}_j \cap \mathcal{D}_{j+1}$, up to shrinking the sectors (since $C(z)^\beta = z^\beta(1+o(1))$ by flatness of $\varphi$). By \eqref{linztilde}, and since $C\circ G = G \circ C$,
\[\tilde z_l(C(G(z))) = \tilde z_l (G(C(z))) =\lambda_l \tilde z_l(C(z)),\]
hence, for all $l=1,...,n$,
\[ \tilde \varphi_l(G(z)) = \lambda_l \tilde z_l(C(z)) - \lambda_l \tilde z_l(z) = \lambda_l \tilde \varphi_l(z).\]

In the same way, changing $z$ to $G(z)$, $U$ becomes $(G(z)^\beta)^{-k}=G_0(U)$ and $U'$ becomes $(C(G(z))^\beta)^{-k} = (G(C(z))^\beta)^{-k} = G_0(U')$. As $\V(G_0(U)) = \V(U) +1$, we get
\[ \tilde \varphi_0(G(z)) = \V(G_0(U')) - \V(G_0(U)) = \V(U')-\V(U) = \tilde \varphi_0(z). \]

\textbf{Decomposition.} Write, for $l=0,...,n$,
\begin{equation}\label{devfourier}
\tilde \varphi_l (z) = \sum_Q f_{l,Q}(z^\beta)z^Q,
\end{equation}
where the sum is over the $Q \in \N^n$ such that $z^Q$ is not divisible by $z^\beta$, and each $f_{l,Q}$ is a holomorphic function of $z^\beta$.
As usual, let
\begin{equation} \label{equatildez}
\tilde z(z)^Q := \prod_{s=1}^n \tilde z_s(z)^{q_s} = z^Q e^{-\sum_s q_s \theta_s(z^\beta)},
\end{equation}
where the second equality comes from \eqref{deftildez}. Setting
\[ \tilde \varphi_{l,Q}(u) := f_{l,Q}(u)e^{\sum_s q_s \theta_s(u)}, \]
we get
\begin{equation}\label{equatildephi} \tilde \varphi_l(z) = \sum_Q \tilde \varphi_{l,Q}(z^\beta)\tilde z(z)^Q.\end{equation}

By \eqref{linztilde}, $\tilde z(G(z))^Q = \lambda^Q \tilde z(z)^Q$. Since $G(z)^\beta = G_0(z^\beta)$, the equation $\tilde\varphi_l(G(z))=\lambda_l \tilde \varphi_l(z)$ gives
\[ \sum_Q\tilde\varphi_{l,Q}(G_0(z^\beta)) \lambda^Q \tilde z(z)^Q = \lambda_l \sum_Q \tilde \varphi_{l,Q}(z^\beta) \tilde z(z)^Q,\]
and $\tilde \varphi_0(G(z))= \tilde\varphi_0(z)$ gives
\[\sum_Q \tilde \varphi_{0,Q}(G_0(z^\beta)) \lambda^Q \tilde z(z)^Q = \sum_Q \tilde \varphi_{0,Q}(z^\beta) \tilde z(z)^Q.\]
Identifying the coefficients of $\tilde z(z)^Q$, we get, with $u=z^\beta$,
\[\tilde\varphi_{0,Q}(G_0(u)) \lambda^Q = \tilde \varphi_{0,Q}(u),\]
and for all $l=1,...,n$
\[\tilde \varphi_{l,Q}(G_0(u)) \lambda^Q = \lambda_l \tilde \varphi_{l,Q}(u).\]
Let for all $l=1,...,n$,
\begin{equation}\label{defmu}\mu_{0,Q} := \frac{1}{\lambda^Q}, \qquad \mu_{l,Q} := \frac{\lambda_l}{\lambda^Q},\end{equation} which have modulus $1$. 
For $u \in S_j \cap S_{j+1}$, let $U = u^{-k}$ and $V = \V(U)$, and let
\begin{equation}\label{defglQ} g_{l,Q}(V) := \tilde \varphi_{l,Q}(u),\end{equation}
where $u = (\V^{-1}(V))^{-\frac{1}{k}}$ and the $\frac{1}{k}$-root is chosen such that $u \in S_j \cap S_{j+1}$. As $G_0(u)^{-k}=G_0(U)$ according to the convention \eqref{G0inf}, and since $\V(G_0(U)) = \V(U)+1$, we get $\tilde \varphi_{l,Q}(G_0(u)) = g_{l,Q}(V+1)$ and the previous equations become, for all $l=0,...,n$,
\begin{equation}\label{invariancehomo}g_{l,Q}(V+1) = \mu_{l,Q} g_{l,Q}(V).\end{equation}

We now study the solutions of \eqref{invariancehomo}.

\textbf{Cauchy estimates.} For $l=1,...,n$, since $C_l(z)=z_l+\varphi_l(z)$, by definitions of $\tilde \varphi_l$ in \eqref{defphitilde} and of $\tilde z_l$ in \eqref{deftildez},
\begin{align}\label{termestildephi} \tilde \varphi_l(z) & = \tilde z_l(C(z))- \tilde z_l(z) = (z_l+\varphi_l(z))e^{-\theta_l(C(z)^\beta)}-z_le^{-\theta_l(z^\beta)} \\
&= \varphi_l(z)e^{-\theta_l(C(z)^\beta)}+z_l \big( e^{-\theta_l(C(z)^\beta)}-e^{-\theta_l(z^\beta)}\big). \nonumber\end{align}
Moreover, $\varphi$ is flat in $z^\beta$, so that up to shrinking $r$, we have $\vert C(z)^\beta - z^\beta\vert \leq \frac{1}{2}\vert z^\beta\vert$, and then
\[ \vert C(z)^\beta \vert \geq \frac{1}{2}\vert z^\beta\vert.\]

We will give a bound to the two terms of \eqref{termestildephi}. 
For the first one, using \eqref{estimthetaj} combined with $\vert C(z)^\beta \vert \geq \frac{1}{2}\vert z^\beta\vert$ and by flatness of $\varphi_l$, for all $N \in \N$ there exists $C_N > 0$ such that
\[\vert \varphi_l(z)e^{-\theta_l(C(z)^\beta)} \vert \leq C_N \vert z^\beta\vert^N C_l 2^{\nu_l}\vert z^\beta\vert^{-\nu_l}.\]
For the second term, denote $f(u) := e^{-\theta_l(u)}$. Then $f$ is holomorphic with $\vert f(u)\vert \leq C_l \vert u\vert^{-\nu_l}$ by \eqref{estimthetaj}. 

\begin{figure} 

\tikzset{
pattern size/.store in=\mcSize, 
pattern size = 5pt,
pattern thickness/.store in=\mcThickness, 
pattern thickness = 0.3pt,
pattern radius/.store in=\mcRadius, 
pattern radius = 1pt}
\makeatletter
\pgfutil@ifundefined{pgf@pattern@name@_45rboku45}{
\pgfdeclarepatternformonly[\mcThickness,\mcSize]{_45rboku45}
{\pgfqpoint{0pt}{0pt}}
{\pgfpoint{\mcSize+\mcThickness}{\mcSize+\mcThickness}}
{\pgfpoint{\mcSize}{\mcSize}}
{
\pgfsetcolor{\tikz@pattern@color}
\pgfsetlinewidth{\mcThickness}
\pgfpathmoveto{\pgfqpoint{0pt}{0pt}}
\pgfpathlineto{\pgfpoint{\mcSize+\mcThickness}{\mcSize+\mcThickness}}
\pgfusepath{stroke}
}}
\makeatother
\tikzset{every picture/.style={line width=0.75pt}} 

\begin{tikzpicture}[x=0.75pt,y=0.75pt,yscale=-1,xscale=1]

\draw    (70,149.16) ;
\draw [shift={(70,149.16)}, rotate = 0] [color={rgb, 255:red, 0; green, 0; blue, 0 }  ][fill={rgb, 255:red, 0; green, 0; blue, 0 }  ][line width=0.75]      (0, 0) circle [x radius= 3.35, y radius= 3.35]   ;
\draw [shift={(70,149.16)}, rotate = 0] [color={rgb, 255:red, 0; green, 0; blue, 0 }  ][fill={rgb, 255:red, 0; green, 0; blue, 0 }  ][line width=0.75]      (0, 0) circle [x radius= 3.35, y radius= 3.35]   ;

\path[fill={rgb, 255:red, 80; green, 227; blue, 194 },fill opacity=0.29]
(207.22,113.14) .. controls (209.11,120.39) and (210.46,127.89) .. (211.2,135.59) .. controls (215.78,183.24) and (196.22,227.65) .. (162.52,256.69) -- (70,149.16) -- cycle;
\draw [color={rgb, 255:red, 80; green, 227; blue, 194 },draw opacity=0]
(207.22,113.14) .. controls (209.11,120.39) and (210.46,127.89) .. (211.2,135.59) .. controls (215.78,183.24) and (196.22,227.65) .. (162.52,256.69);

\path[fill={rgb, 255:red, 245; green, 166; blue, 35 },fill opacity=0.3]
(147.85,30.56) .. controls (154.11,34.68) and (160.13,39.35) .. (165.83,44.58) .. controls (201.12,76.92) and (216.31,123) .. (210.72,167.14) -- (70,149.16) -- cycle;
\draw [color={rgb, 255:red, 245; green, 166; blue, 35 },draw opacity=0]
(147.85,30.56) .. controls (154.11,34.68) and (160.13,39.35) .. (165.83,44.58) .. controls (201.12,76.92) and (216.31,123) .. (210.72,167.14);

\draw  [draw opacity=0][fill={rgb, 255:red, 255; green, 255; blue, 255 }  ,fill opacity=1 ] (473.12,74.51) .. controls (484.79,93.91) and (492.46,116.13) .. (494.88,140.15) .. controls (501.62,206.9) and (465.84,268.15) .. (409.52,296.99) -- (337.39,156.04) -- cycle ; \draw [color={rgb, 255:red, 189; green, 16; blue, 224 }  ,draw opacity=1 ][fill={rgb, 255:red, 255; green, 255; blue, 255 }  ,fill opacity=1 ]   (473.12,74.51) .. controls (484.79,93.91) and (492.46,116.13) .. (494.88,140.15) .. controls (501.62,206.9) and (465.84,268.15) .. (409.52,296.99) ;  

\draw [color={rgb, 255:red, 189; green, 16; blue, 224 }  ,draw opacity=1 ][pattern=_45rboku45,pattern size=3.75pt,pattern thickness=0.75pt,pattern radius=0pt, pattern color={rgb, 255:red, 189; green, 16; blue, 224}]   (330,150) -- (409.52,296.99) ;

\draw  [draw opacity=0][fill={rgb, 255:red, 255; green, 255; blue, 255 }  ,fill opacity=1 ][dash pattern={on 4.5pt off 4.5pt}] (407.85,31.4) .. controls (414.11,35.52) and (420.13,40.19) .. (425.83,45.42) .. controls (461.12,77.76) and (476.31,123.84) .. (470.72,167.98) -- (330,150) -- cycle ; \draw [color={rgb, 255:red, 245; green, 166; blue, 35 }  ,draw opacity=1 ][fill={rgb, 255:red, 255; green, 255; blue, 255 }  ,fill opacity=1 ][dash pattern={on 4.5pt off 4.5pt}] [dash pattern={on 4.5pt off 4.5pt}]  (407.85,31.4) .. controls (414.11,35.52) and (420.13,40.19) .. (425.83,45.42) .. controls (461.12,77.76) and (476.31,123.84) .. (470.72,167.98) ;  

\draw [color={rgb, 255:red, 189; green, 16; blue, 224 }  ,draw opacity=1 ]   (330,149.16) -- (473.12,74.51) ;

\draw [color={rgb, 255:red, 245; green, 166; blue, 35 }  ,draw opacity=1 ] [dash pattern={on 4.5pt off 4.5pt}]  (329.65,151.47) -- (407.85,31.4) ;

\draw [color={rgb, 255:red, 245; green, 166; blue, 35 }  ,draw opacity=1 ] [dash pattern={on 4.5pt off 4.5pt}]  (331.8,150.07) -- (470.72,167.98) ;

\path[fill={rgb, 255:red, 80; green, 227; blue, 194 },fill opacity=0.29]
(467.22,113.14) .. controls (469.11,120.39) and (470.46,127.89) .. (471.2,135.59) .. controls (475.78,183.24) and (456.22,227.65) .. (422.52,256.69) -- (330,149.16) -- cycle;
\draw [color={rgb, 255:red, 80; green, 227; blue, 194 },draw opacity=0]
(467.22,113.14) .. controls (469.11,120.39) and (470.46,127.89) .. (471.2,135.59) .. controls (475.78,183.24) and (456.22,227.65) .. (422.52,256.69);

\draw [color={rgb, 255:red, 144; green, 19; blue, 254 }  ,draw opacity=1 ]   (310,220) .. controls (310,235.68) and (345.53,243.68) .. (368.6,230.81) ;
\draw [shift={(370,230)}, rotate = 148.67] [color={rgb, 255:red, 144; green, 19; blue, 254 }  ,draw opacity=1 ][line width=0.75]    (10.93,-3.29) .. controls (6.95,-1.4) and (3.31,-0.3) .. (0,0) .. controls (3.31,0.3) and (6.95,1.4) .. (10.93,3.29)   ;

\draw [color={rgb, 255:red, 80; green, 227; blue, 194 }  ,draw opacity=1 ]   (330,149.16) -- (467.22,113.14) ;

\draw [color={rgb, 255:red, 80; green, 227; blue, 194 }  ,draw opacity=1 ]   (331.8,150.07) -- (422.52,256.69) ;

\draw  [draw opacity=0][fill={rgb, 255:red, 255; green, 255; blue, 255 }  ,fill opacity=0 ] (392.2,116.81) .. controls (392.76,116.94) and (393.3,117.11) .. (393.84,117.33) .. controls (398.65,119.33) and (401.13,124.77) .. (400.12,130.31) -- (389,129) -- cycle ; \draw [fill={rgb, 255:red, 255; green, 255; blue, 255 }  ,fill opacity=0 ]   (392.2,116.81) .. controls (392.76,116.94) and (393.3,117.11) .. (393.84,117.33) .. controls (398.65,119.33) and (401.13,124.77) .. (400.12,130.31) ;  

\draw  [fill={rgb, 255:red, 248; green, 231; blue, 28 }  ,fill opacity=0.35 ] (410.47,121.91) .. controls (410.47,108.52) and (421.33,97.66) .. (434.72,97.66) .. controls (448.11,97.66) and (458.97,108.52) .. (458.97,121.91) .. controls (458.97,135.3) and (448.11,146.16) .. (434.72,146.16) .. controls (421.33,146.16) and (410.47,135.3) .. (410.47,121.91) -- cycle ;

\draw    (425.72,99.16) -- (435.23,121.4) ;
\draw [shift={(435.23,121.4)}, rotate = 66.83] [color={rgb, 255:red, 0; green, 0; blue, 0 }  ][fill={rgb, 255:red, 0; green, 0; blue, 0 }  ][line width=0.75]      (0, 0) circle [x radius= 3.35, y radius= 3.35]   ;

\draw  [dash pattern={on 4.5pt off 4.5pt}]  (280,150) -- (310.5,149.54) ;
\draw [shift={(313.5,149.5)}, rotate = 179.14] [fill={rgb, 255:red, 0; green, 0; blue, 0 }  ][line width=0.08]  [draw opacity=0] (10.72,-5.15) -- (0,0) -- (10.72,5.15) -- (7.12,0) -- cycle    ;

\draw    (30,149.16) -- (238.5,149.16) -- (257,149.16) ;
\draw [shift={(260,149.16)}, rotate = 180] [fill={rgb, 255:red, 0; green, 0; blue, 0 }  ][line width=0.08]  (8.93,-4.29) -- (0,0) -- (8.93,4.29) -- cycle    ;

\draw (63,153.56) node [anchor=north west][inner sep=0.75pt]  [font=\footnotesize]  {$0$};
\draw (215,149.16) node [anchor=north west][inner sep=0.75pt]   [align=left] {$\displaystyle r$};
\draw (151,210.56) node [anchor=north west][inner sep=0.75pt]  [font=\small]  {$S_{j}$};
\draw (143,131.4) node [anchor=north west][inner sep=0.75pt]  [font=\small]  {$ \begin{array}{l}
S_{j} \cap S_{j}{}_{+}{}_{1}\\
\end{array}$};
\draw (411,210.4) node [anchor=north west][inner sep=0.75pt]  [color={rgb, 255:red, 74; green, 144; blue, 226 }  ,opacity=1 ]  {$S_{j}$};
\draw (403,42.4) node [anchor=north west][inner sep=0.75pt]  [color={rgb, 255:red, 245; green, 166; blue, 35 }  ,opacity=1 ]  {$S_{j+1}$};
\draw (241,199.4) node [anchor=north west][inner sep=0.75pt]  [font=\scriptsize,color={rgb, 255:red, 144; green, 19; blue, 254 }  ,opacity=1 ]  {$\left\{u,u^{-k} \in \Delta _{+}( \alpha ,R)\right\}$};
\draw (387,122.4) node [anchor=north west][inner sep=0.75pt]  [font=\scriptsize]  {$\epsilon $};
\draw (435.72,107.06) node [anchor=north west][inner sep=0.75pt]  [font=\tiny]  {$c|u|$};
\draw (431.5,125.4) node [anchor=north west][inner sep=0.75pt]  [font=\scriptsize]  {$u$};
\draw (141,43.4) node [anchor=north west][inner sep=0.75pt]  [font=\small]  {$S_{j+1}$};

\end{tikzpicture}
    \caption{Left: the two sectors $S_j$, $S_{j+1}$ and their intersection. \\ Right: Near $0$, the domain of $\theta_l$ (a petal of opening $\frac{2\alpha}{k}$) contains the sector $S_j$ (of opening $\frac{2\alpha}{k}-2\epsilon$), giving an angular margin of $\epsilon$ on each side. Hence, for $r$ small enough and $u \in S_{j}\cap S_{j+1}$, the disc $D(u,c \vert u\vert)$ of radius proportional to $\vert u \vert$, which $0<c< \sin \epsilon$ is contained in the domain of $\theta_l$ and Cauchy's inequality can be applied on it.}
    \label{margin}
\end{figure}
The function $f = e^{-\theta_l}$ is holomorphic in $u$ on a petal containing $S_j\cap S_{j+1}$ with a positive angular margin, thanks to the $-\epsilon$ in \eqref{sector} (see Figure \ref{margin}). Hence Cauchy's inequalities on discs of center $u \in [z^\beta, C(z)^\beta]$ and radius proportional to $\vert u \vert$, together with \eqref{estimthetaj}, give 
for all $u \in [z^\beta, C(z)^\beta]$ (up to shrinking $r$, since $C(z)^\beta-z^\beta$ is flat, the discs are contained in the domain of $\theta_l$)
\begin{equation}\label{cauchythetal}\vert \frac{d}{du} e^{-\theta_l(u)} \vert \leq C'_l \vert u \vert^{-\nu_l-1}.\end{equation}
By the mean value inequality,
\[  \vert z_l(e^{-\theta_l(C(z)^\beta)}-e^{-\theta_l(z^\beta)}\vert \leq \rho \sup_{u \in [z^\beta, C(z)^\beta]} \vert \frac{d}{du}(e^{-\theta_l(u)})\vert \vert C(z)^\beta - z^\beta \vert \]
Notice that a point $v \in [z^\beta, C(z)^\beta]$ have modulus at least $\frac{1}{2}\vert z^\beta\vert$ (write for some $t \in [0,1]$, $v = z^\beta + t(C(z)^\beta-z^\beta)$, then $\vert v - z^\beta \vert \leq \vert C(z)^\beta - z^\beta \vert \leq \frac{1}{2}\vert z^\beta \vert$, and use the reverse triangle inequality). Then by previous Cauchy's inequality we get 
\begin{equation} \label{cauchythetal2}\vert z_l(e^{-\theta_l(C(z)^\beta)}-e^{-\theta_l(z^\beta)}\vert \leq \rho  C'_l 2^{\nu_l+1} \vert z^\beta \vert^{-\nu_l-1} \vert C(z)^\beta-z^\beta \vert.\end{equation}
 Finally, since $C(z)^\beta - z^\beta$ and $\varphi_l$ are flat in $z^\beta$, both terms of \eqref{termestildephi} are flat in $z^\beta$, and then $\tilde \varphi_l$ is flat in $z^\beta$.
The same argument applies to $\tilde \varphi_0(z) = \V(U')- \V(U)$ since $U'-U$ is flat in $z^\beta$ (since $C(z)^\beta - z^\beta$ is flat in $z^\beta$), and $\V'$ is bounded on $[U,U']$ by Cauchy inequalities, as for $\theta_l$.

In particular, the $\tilde \varphi_l$ are bounded on $\mathcal D_j \cap \mathcal D_{j+1}$, and denote
\[ M := \max_{0 \leq l \leq n} \sup_{\mathcal D_j \cap \mathcal D_{j+1}} \vert \tilde \varphi_l \vert\quad (<\infty).\]

Recall the expansion \eqref{devfourier}.
Fix $u = z^\beta$, with $\vert u \vert \leq \tilde \rho^{\vert \beta \vert}$ where $0 < \tilde \rho < \rho$. For all $t_1, ..., t_n \leq \tilde \rho$ such that $\prod_{s=1}^n t_s^{\beta_s} = \vert u \vert$, the set \[\mathcal{T} := \{ z \in \C^n, \forall s=1,...,n,\vert z_s \vert = t_s , z^\beta = u\} \subset \mathcal{D}_j \cap \mathcal{D}_{j+1}\] satisfies in particular $\mathcal{T} \subset \{z^\beta = u \} \cap \D_\rho^n$, and then on $\mathcal{T}$, $\vert \tilde \varphi_l\vert \leq M$.

For $z \in \mathcal{T}$, we write $z= (t_1e^{i\psi_1}, ..., t_n e^{i\psi_n})$ with $\psi = (\psi_1,...,\psi_n)$. The condition $z^\beta = u$ writes $t^{\beta} e^{i\langle \psi, \beta \rangle} = u$, that is to say $t^\beta = \vert u \vert$ (which is automatically satisfied by the choice of the $t_s$) and $\langle \psi, \beta \rangle = \arg u \mod 2\pi$. We can then parametrize $\mathcal{T}$ by
\[ \mathbb{T}_u := \{ \psi \in \R^n /(2 \pi \Z)^n, \langle \beta, \psi \rangle = \arg u \mod 2\pi \},\]
and we denote $z(\psi) := (t_1 e^{i\psi_1}, ..., t_n e^{i \psi_n}) \in \mathcal{T}$ for $\psi \in \mathbb{T}_u$. Then the expansion of \eqref{devfourier} writes
\[\tilde \varphi_l (z(\psi)) = \sum_{Q'} f_{l,Q'}(u)z(\psi)^{Q'} = \sum_{Q'} f_{l,Q'}(u)t^{Q'} e^{i\langle Q', \psi \rangle}.\]
Denote by $\hat C$ the constant such that $\hat C \int_{\mathbb{T}_u} d\psi = 1$.
Let $Q \in \N^n$ with $z^Q$ not divisible by $z^\beta$. For $Q' \neq Q$, $Q'-Q$ is not in $\beta \Z$, otherwise $Q = Q' + p \beta$ for some $p \in \Z^{\ast}$ and one of the $z^Q$ or $z^{Q'}$ would be divisible by $z^\beta$. Then $Q'-Q \notin \beta \Z$ and there exists $\zeta \in \R^n$ such that $\langle \beta, \zeta \rangle \in 2 \pi \Z$ and $\langle Q'-Q, \zeta\rangle \notin 2\pi \Z$. Hence
\[ \int_{\mathbb{T}_u} e^{i \langle Q'-Q,\psi \rangle} d\psi =\int_{\mathbb{T}_u} e^{i \langle Q'-Q,\psi +\zeta\rangle} d\psi = e^{i \langle Q'-Q, \zeta \rangle}  \int_{\mathbb{T}_u} e^{i \langle Q'-Q,\psi \rangle }d\psi\]
and since $\langle Q'-Q,\zeta \rangle \notin 2\pi \Z$,
\[ \int_{\mathbb{T}_u} e^{i \langle Q'-Q, \psi \rangle}d\psi = 0.\]
Then, multiplying the expansion of $\tilde \varphi_l$ by $e^{-i \langle Q, \psi \rangle}$ provides the coefficient $f_{l,Q}(u)t^Q$ by integrating on $\mathbb{T}_u$, and we get
\[ f_{l,Q}(u) t^Q = \hat C \int_{\mathbb{T}_u} \tilde \varphi_l(z(\psi))  e^{-i \langle Q, \psi\rangle }d\psi. \]
Hence,
\[ \vert f_{l,Q}(u) \vert t^Q \leq \hat C \int_{\mathbb{T}_u} \vert \tilde \varphi_l(z(\psi))\vert  \vert e^{-i \langle Q, \psi\rangle } \vert d\psi  \leq M \hat C \int_{\mathbb{T}_u}d\psi = M,\]
and we finally get
 \[ \vert f_{l,Q}(u)\vert \leq  
\frac{M}{\prod_{s=1}^n t_s^{q_s}}.\] 

Since $z^Q$ is not divisible by $z^\beta$, there exists $s_0 \leq m$ with $q_{s_0} < \beta_{s_0}$. We take $t_s = \tilde \rho$ for $s\neq s_0$, and $t_{s_0}$ is given by the relation $\prod_s t_s^{\beta_s} = \vert u\vert$. We get $t_{s_0}^{\beta_{s_0}} = \vert u \vert \big( \prod_{s \neq s_0} t_s^{\beta_s}\big)^{-1} = \vert u \vert \tilde \rho^{\beta_{s_0} - \vert \beta \vert}$ so that $t_{s_0}\leq \tilde \rho$ since $\vert u \vert  \leq \tilde \rho^{\vert \beta \vert}$. Then
\[ \prod_{s=1}^n t_s^{q_s} = \tilde \rho^{\vert Q \vert-q_{s_0}}t_{s_0}^{q_{s_0}} = \tilde \rho ^{\vert Q \vert} \big( \frac{\vert u \vert}{\tilde \rho^{\vert \beta \vert}} \big)^{\frac{q_{s_0}}{\beta_{s_0}}}. \]
Since $\frac{q_{s_0}}{\beta_{s_0}}< 1$, and $\frac{\tilde \rho^{\vert \beta \vert}}{\vert u \vert} \geq 1$, we get
\[ \vert f_{l,Q}(u)\vert \leq \frac{M}{\tilde \rho^{\vert Q \vert}} \big( \frac{ \tilde \rho^{\vert \beta \vert}}{\vert u \vert} \big)^{\frac{q_{s_0}}{\beta_{s_0}}} \leq M \tilde \rho^{\vert \beta \vert - \vert Q \vert}\vert u \vert^{-1}.\]
Finally, the left-hand side of \eqref{estimthetaj} gives $\vert e^{\theta_s(u)}\vert \leq C_s \vert u \vert^{\nu_s}$, and since $\tilde \varphi_{l,Q}(u) = f_{l,Q}(u)e^{\sum_s q_s \theta_s(u)}$, we get with $K := \max_s C_s$ and $\langle \nu, Q\rangle := \sum_{s=1}^n \nu_s q_s$,
\begin{equation}\label{cauchytilde}
    \vert \tilde \varphi_{l,Q}(u) \vert \leq M \tilde \rho^{\vert \beta \vert} \big( \frac{K}{\tilde \rho}\big)^{\vert Q \vert} \vert u \vert^{\langle \nu, Q \rangle-1}.
\end{equation}

\textbf{An intermediate equation.} Since $\vert \mu_{l,Q}\vert = 1$ (see \eqref{defmu}), we can write $\mu_{l,Q} := e^{2i\pi \vartheta_{l,Q}}$ for some $\vartheta_{l,Q} \in \R \mod \Z$. We will study the component of $\Delta_+(\alpha,R)\cap \Delta_-(\alpha,R)$ where $\Im U > 0$, for which we choose $\vartheta_{l,Q} \in ]-1,0]$ (the case $\Im U < 0$ is similar, changing $t = e^{2i\pi V}$ by $t = e^{-2i\pi V}$ and $\vartheta_{l,Q} \in ]-1,0]$ by $\vartheta_{l,Q} \in [0,1[$).  Let $\gamma_{l,Q}(V) := e^{-2i \pi \vartheta_{l,Q} V}g_{l,Q}(V)$ (notice that $V \mapsto e^{2i\pi \vartheta_{l,Q} V}$ is a particular solution of \eqref{invariancehomo}), then Equation \eqref{invariancehomo} becomes
\[ \gamma_{l,Q}(V+1) = e^{-2i\pi \vartheta_{l,Q} (V+1)}g_{l,Q}(V+1) = e^{-2i\pi \vartheta_{l,Q} V}e^{-2i\pi \vartheta_{l,Q} }\mu_{l,Q} g_{l,Q}(V) = \gamma_{l,Q}(V)\]
for $V$ and $V+1$ in $\V(\Delta_+(\alpha,R) \cap \Delta_-(\alpha,R))$, which contains the half-strip $[ 0, 1] \times [y_0, + \infty[$ and its translate by $1$ for $y_0$ large enough (since $\alpha > \frac{\pi}{2}$). Hence $\gamma_{l,Q}$ extends by periodicity to a $1$-periodic function of $\{ \Im V > y_0 \}$, equal to $\gamma_{l,Q}$ on $\V(\Delta_+(\alpha,R) \cap \Delta_-(\alpha,R)) \cap \{ \Im V > y_0 \}$, and following \cite{Loday}, we write $\gamma_{l,Q}(V) = h_{l,Q}(e^{2i\pi V})$.
The map $V \mapsto t = e^{2i\pi V}$ sends this half-strip to a punctured disc centered at $0$ in $\C$. Moreover, a fundamental system of neighborhoods of infinity in this half-stripe is sent on a fundamental system of neighborhoods of $0$. We also have $\gamma_{l,Q}(V) \rightarrow 0$ when $\Im(V)\rightarrow + \infty$. Indeed, 
\[ \vert \gamma_{l,Q}(V) \vert = \vert e^{-2i\pi \vartheta_{l,Q} V}  g_{l,Q}(V) \vert =e^{\Re(-2i\pi \vartheta_{l,Q} V)} \vert g_{l,Q}(V)\vert =e^{2\pi \vartheta_{l,Q} \Im V} \vert g_{l,Q}(V)\vert,\]
and since $\vartheta_{l,Q} \in ]-1,0]$, $e^{2\pi \vartheta_{l,Q} \Im V} \leq 1$ for $\Im V > 0$ and by flatness of $g_{l,Q}$ (recall that $g_{l,Q}(V) = \tilde \varphi_{l,Q}(u)$ by \eqref{defglQ}, and the flatness of $\tilde \varphi_{l,Q}$ is obtained from the flatness of $\tilde \varphi_{l}$ using the previous Cauchy inequalities with $\sup_{\mathcal{T}} \vert \tilde \varphi_l \vert \leq C_N \vert u \vert^N$ instead of $M$), we have that $\gamma_{l,Q}(V)\rightarrow 0$ as $\Im V \rightarrow +\infty$.
Finally, $h_{l,Q}(t) \rightarrow 0$ when $t \rightarrow 0$ and by the removable singularity theorem, $h_{l,Q}$ can be continued into a holomorphic map at $0$ with $h_{l,Q}(0)=0$.

\textbf{Estimate of the coefficients of $h_{l,Q}$.} Write
\begin{equation} \label{defhlQ}h_{l,Q}(t) = \sum_{d\geq 1} h_{l,Q,d}t^d,\end{equation}
and recall that $d + \vartheta_{l,Q} > 0$ for $d\geq 1$ since $\vartheta_{l,Q} \in ]-1,0]$.

Recall that $g_{l,Q}(V) = \tilde \varphi_{l,Q}(u)$ with $U = \V^{-1}(V) = u^{-k}$ (see \eqref{defglQ}). Up to increasing $y_0$, we may assume that $y_0 \geq \max(1,R)$ and that, for $\Im V \geq y_0$, we have $\frac{1}{2}\vert V \vert \leq \vert U \vert \leq 2 \vert V \vert$ and $\vert u \vert \leq \tilde \rho^{\vert \beta \vert}$. This is possible by \eqref{equationVU} and since $\vert u \vert = \vert U \vert^{-\frac{1}{k}} \leq (\frac{\Im V}{2})^{-\frac{1}{k}}$. The condition $\vert u \vert \leq \tilde \rho^{\vert \beta \vert}$ is necessary to apply \eqref{cauchytilde}, and $y_0$ depends only on $\V$, $R$ and $\tilde \rho$, not on $l$ and $Q$. For a fixed $y\geq y_0$, the circle $\{ \vert t \vert = e^{-2\pi y} \}$ is the image of the segment $\{ V = x +i y, x \in [0,1]\}$ of the half-strip by the map $V \mapsto t=e^{2i\pi V}$, on which $\vert h_{l,Q}(t)\vert = \vert \gamma_{l,Q}(V)\vert =  e^{2\pi \vartheta_{l,Q}y}\vert g_{l,Q}(V)\vert$.
Cauchy's inequality on the circle gives
\[ \vert h_{l,Q,d} \vert \leq e^{2\pi d y} \sup_{\vert t \vert = e^{-2\pi y}} \vert h_{l,Q} \vert = e^{2\pi (d+ \vartheta_{l,Q})y} \sup_{\Re V \in [0,1], \Im V = y} \vert g_{l,Q}(V) \vert.\]

On the segment $\{ \Im V = y, \Re V \in [0,1] \}$ of the half-strip (recall $y \geq y_0 \geq R$), we have $y \leq \vert V \vert \leq y+1 \leq 2y$ for $y\geq 1$, hence $\frac{y}{2} \leq \vert U \vert \leq 4y$. Therefore, by \eqref{cauchytilde}, since $\vert u \vert^{\langle \nu, Q \rangle -1} = \vert U \vert^{\frac{1}{k}}\vert U\vert^{-\frac{\langle \nu, Q \rangle}{k}}$, and with $M_Q: = M \tilde\rho^{\vert \beta \vert}\big(\frac {K}{\tilde \rho}\big)^{\vert Q\vert}$, using $\vert U \vert \leq 4y$ for the first factor and $\vert U \vert \geq \frac{y}{2}$ for the second one (recall $\langle \nu, Q \rangle \geq 0$), we get, for all $y\geq y_0$
\begin{equation}\label{estimhlQd}\vert h_{l,Q,d} \vert \leq M_Q (4y)^{\frac{1}{k}}\big(\frac{y}{2} \big)^{-\frac{\langle \nu, Q\rangle }{k}}e^{2\pi (d+\vartheta_{l,Q})y}.\end{equation}

\textbf{Estimates of $g_{l,Q}$.} Let $\xi_Q := \frac{\langle \nu, Q\rangle}{k}$, and let $V' \in \V(\Delta_+(\alpha,R) \cap \Delta_-(\alpha,R))$ with $Y:= \Im V' \geq 2y_0$. Choosing $p \in \Z$ such that $V' + p$ is in the half-strip, the $1$-periodicity of $\gamma_{l,Q}$ gives $\gamma_{l,Q}(V') = \gamma_{l,Q}(V'+p) = h_{l,Q}(e^{2i\pi V'})$. Estimate \eqref{estimhlQd} with $y=\frac{Y}{2} \geq y_0$ gives, for all $d\geq 1$
\[\vert h_{l,Q,d} \vert \leq M_Q (2Y)^{\frac{1}{k}}\big(\frac{Y}{4} \big)^{-\frac{\langle \nu, Q\rangle }{k}}e^{\pi (d+\vartheta_{l,Q})Y}.\]
Since $g_{l,Q}(V') = e^{2i\pi \vartheta_{l,Q} V'} h_{l,Q}(e^{2i\pi V'})$, and $\vert e^{2i\pi V'}\vert = \vert e^{2i\pi \Re V' - 2\pi \Im V' }\vert = e^{- 2\pi Y}$, we get by \eqref{defhlQ},
\[ \vert g_{l,Q}(V') \vert \leq \sum_{d\geq 1} \vert h_{l,Q,d}\vert e^{-2\pi (d+\vartheta_{l,Q})Y} \leq M_Q (2Y)^{\frac{1}{k}}\big( \frac{Y}{4}\big)^{-\xi_Q} \sum_{d\geq 1}e^{-\pi (d+\vartheta_{l,Q})Y}\]
and since $e^{-\pi Y} \leq \frac{1}{2}$ for $Y \geq 1$, this is
\[ \leq 2M_Q (2Y)^{\frac{1}{k}}\big( \frac{Y}{4}\big)^{-\xi_Q} e^{-\pi (1+\vartheta_{l,Q})Y}.\]

\textbf{Estimate of the terms of $\tilde \varphi_l$.}
 Let $0 < \rho' < \tilde \rho$ and $z \in \mathcal D_j \cap \mathcal D_{j+1}$ with $\vert z_s \vert \leq \rho'$ for all $s$, and $U = (z^\beta)^{-k}$. Exactly as in \cite{Loday}, on $\Delta_+(\alpha,R) \cap \Delta_-(\alpha,R)$, we have $ \vert \Im U \vert \geq \vert U \vert \sin \alpha$ (see Figure \ref{intersectionsector}), and since $\Im \V(U) = \Im U+ \OO( \ln \vert U\vert)$ by \eqref{equationVU}, we get $\Im \V(U) \geq \frac{\sin \alpha}{2} \vert U\vert$ for $\vert U\vert$ large enough, and in particular, $\Im \V(U) \geq 2y_0$ (for $\vert U \vert$ large enough) and $\Im \V(U) \leq \vert \V(U)\vert \leq 2 \vert U \vert$ by definition of $y_0$. Moreover, by \eqref{deftildez} and \eqref{estimthetaj}, $\vert \tilde z_s(z)\vert \leq C_s \rho' \vert z^\beta \vert^{-\nu_s} = C_s \rho' \vert U \vert^{\frac{\nu_s}{k}}$, so that
 \[\vert \tilde z(z)^Q \vert \leq(K \rho')^{\vert Q\vert} \vert U\vert^{\frac{\langle \nu, Q \rangle}{k}} = (K \rho')^{\vert Q\vert} \vert U \vert^{\xi_Q}.\]
We apply the estimate of $g_{l,Q}$ at the point $V' = \V(U)$ (then $Y = \Im \V(U)$), and recall that $\tilde \varphi_{l,Q}(z^\beta) = g_{l,Q}(\V(U))$ (see \eqref{defglQ}).  Since $\xi_Q \geq 0$, we have $(2Y)^{\frac{1}{k}} \leq (4\vert U \vert)^{\frac{1}{k}}$ and $(\frac{Y}{4})^{-\xi_Q} \leq (\frac{\sin \alpha}{8} \vert U \vert)^{-\xi_Q}$, and we get
\[\vert \tilde \varphi_{l,Q}(z^\beta)\tilde z(z)^Q \vert \leq 2M_Q (4 \vert U \vert)^{\frac{1}{k}}\big( \frac{\sin \alpha}{8} \vert U\vert\big)^{-\xi_Q} (K \rho')^{\vert Q\vert} \vert U \vert^{\xi_Q}e^{-\pi (1+\vartheta_{l,Q})Y}, \]
that is 
\begin{equation}\label{estimtildephitildez}\vert \tilde \varphi_{l,Q}(z^\beta)\tilde z(z)^Q \vert \leq 2.4^{\frac{1}{k}}M_Q \big( \frac{\sin \alpha}{8}\big)^{-\xi_Q} (K \rho')^{\vert Q\vert} \vert U \vert^{\frac{1}{k}}e^{-\pi (1+\vartheta_{l,Q})Y}\end{equation}

\textbf{Small divisors.} We will precise how the factor $1+\vartheta_{l,Q}$ depends on $Q$. If $\mu_{l,Q}=1$, then $\vartheta_{l,Q}=0$ and $1+\vartheta_{l,Q}=1$. 

Otherwise, for $l=1,...,n$, since $\vert \lambda^Q \vert =1$ and by \eqref{defmu}, we have 
\[\vert \mu_{l,Q}-1\vert = \vert \frac{\lambda_l}{\lambda^Q}-1\vert = \vert \lambda_l-\lambda^Q\vert,\]
and for $l=0$, since $\mu_{0,Q}=\frac{1}{\lambda^Q}$ and $\vert \lambda_1 \vert = 1$,
\[ \vert\mu_{0,Q}-1\vert = \vert 1-\lambda^Q \vert = \vert \lambda_1-\lambda^{Q+E_1}\vert. \]

Moreover, since we could also write $\mu_{l,Q} = e^{2i\pi \vartheta_{l,Q}}$, we get
\[ \vert \mu_{l,Q} - 1 \vert = \vert e^{2i\pi \vartheta_{l,Q}}-1\vert = \vert e^{i\pi \vartheta_{l,Q}}(e^{i\pi \vartheta_{l,Q}}-e^{-i\pi \vartheta_{l,Q}})\vert = \vert e^{i\pi \vartheta_{l,Q}}\vert \vert 2\sin(\pi \vartheta_{l,Q})\vert = \vert 2\sin(\pi \vartheta_{l,Q})\vert. \]
Since $\vert \sin(\pi \vartheta_{l,Q})\vert = \vert\sin(\pi(p+\vartheta_{l,Q}))\vert$ for all $p\in \Z$, we have $\vert \sin (\pi \vartheta_{l,Q})\vert = \vert\sin (\pi \dist(\vartheta_{l,Q},\Z))\vert$. 

On the one hand, since $\dist(\vartheta_{l,Q},\Z) \in [0, \frac{1}{2}]$ and from usual inequality $\sin t \leq t$,
\[ \vert \mu_{l,Q} - 1\vert = \vert 2 \sin(\pi \vartheta_{l,Q})\vert = 2 \vert \sin(\pi \dist(\vartheta_{l,Q},\Z))\vert \leq 2\pi \dist (\vartheta_{l,Q},\Z).\]
On the other hand, using Diophantine condition (Definition \ref{dio}) and previous computation, we have
\[ \vert \mu_{l,Q} - 1 \vert = \vert \lambda_l-\lambda^Q\vert \geq \frac{C_0}{(\vert Q \vert+1)^{\tau}}\]
for some $C_0,\tau$ (Diophantine condition applies when $\vert Q \vert \geq 2$, since $z^\beta$ does not divide $z^Q$, and the finitely many remaining $Q$ with $\mu_{l,Q} \neq 1$ are compensated by the constant $C_0$).
For $l=0$, we write $Q + E_1 = Q'' + p \beta$ with $z^{Q''}$ not divisible by $z^\beta$, so that $\lambda^{Q +E_1} = \lambda^{Q''}$, $\vert Q '' \vert \leq \vert Q \vert +1$ and $Q'' \neq E_1$ since $Q \neq 0$, and the same remark on the application of the Diophantine condition applies to $Q''$. Finally, since $\vartheta_{l,Q} \in ]-1,0]$, we have $1 + \vartheta_{l,Q} \geq \dist(\vartheta_{l,Q}, \Z)$, and this gives, in all cases (up to decreasing $C_0$ so that $C_0\leq2\pi$),

\begin{equation} \label{estimdistance}1 + \vartheta_{l,Q} \geq \frac{C_0}{2\pi (\vert Q \vert +1)^\tau}.\end{equation}

\textbf{Summation.} Recall that $\tilde \varphi_l (z) = \sum_Q \tilde \varphi_{l,Q}(z^\beta)\tilde z(z)^Q$. Given $N,m \geq 1$, we recall the useful estimates:
\begin{equation}\label{estimQ1} \# \{Q \in \N^n, \vert Q \vert \leq N\} \leq C_nN^n \end{equation}
\begin{equation}\label{estimQ2} \# \{Q \in \N^n, \vert Q \vert = m\} \leq C_nm^{n-1} \end{equation}
for some constant $C_n$ depending on $n$.

Let $\nu_{\max} := \max_{s=1,...,n} \nu_s$ so that $\xi_Q   \leq \frac{\nu_{\max}}{k} \vert Q \vert$. Since $\frac{8}{\sin \alpha} > 1$, and recalling that $M_Q = M \tilde \rho ^{\vert \beta \vert} \big(\frac{K}{\tilde \rho}\big)^{\vert Q \vert}$, we get
\[ 2.4^{\frac{1}{k}} M_Q \big(\frac{8}{\sin \alpha}\big)^{\xi_Q} (K \rho')^{\vert Q \vert} \leq 2.4^{\frac{1}{k}} M \tilde \rho^{\vert \beta \vert} \big(\frac{K}{\tilde \rho}\big)^{\vert Q \vert}\big(\frac{8}{\sin \alpha}\big)^{\frac{\nu_{\max}}{k}\vert Q \vert} (K \rho')^{\vert Q \vert} = 2.4^{\frac{1}{k}}M \tilde \rho^{\vert \beta \vert}r_0^{\vert Q \vert}\]
where
\[ r_0 := \frac{K^2\rho'}{\tilde \rho} \big( \frac{8}{\sin \alpha}\big)^{\frac{\nu_{max}}{k}} < 1\]
up to shrinking $\rho'$ (recall that $\rho' < \tilde \rho)$.
Moreover, by \eqref{estimdistance} and since $Y \geq \frac{\sin \alpha}{2}\vert U\vert$,
\[\pi (1+ \vartheta_{l,Q})Y \geq \pi \frac{C_0}{2\pi(\vert Q \vert+1)^\tau} \frac{\sin \alpha}{2}\vert U \vert = \frac{c_1 \vert U \vert}{(\vert Q \vert +1)^\tau}\]
where
\[ c_1 := \frac{C_0 \sin \alpha}{4}.\]
Notice that the constants $M, \tilde \rho, K, r_0$ and $c_1$ do not depend on $l$ and $Q$. Combining these two estimates in \eqref{estimtildephitildez}, we get
\begin{equation}\label{estimcoeftildephi}
    \vert \tilde \varphi_{l,Q}(z^\beta)\tilde z(z)^Q \vert \leq  2.4^{\frac{1}{k}}M \tilde \rho^{\vert \beta \vert}\vert U \vert^{\frac{1}{k}}r_0^{\vert Q \vert}e^{-\frac{c_1 \vert U \vert}{(\vert Q \vert +1)^\tau}}.
\end{equation}

For $\vert Q \vert \leq N$, we have $e^{-\frac{c_1 \vert U \vert}{(\vert Q \vert+1)^\tau}} \leq e^{-\frac{c_1 \vert U\vert}{(N+1)^\tau}}$ and $r_0^{\vert Q \vert} \leq 1$, so that, by \eqref{estimQ1},
\[\sum_{\vert Q \vert \leq N} \vert \tilde \varphi_{l,Q}(z^\beta) \tilde z(z)^Q \vert \leq 2.4^{\frac{1}{k}} M \tilde \rho^{\vert \beta \vert} C_n  N^n\vert U \vert^{\frac{1}{k}}e^{- \frac{c_1 \vert U \vert}{(N+1)^\tau}}.\]

For $\vert Q \vert > N$, we bound the exponential by $1$, and \eqref{estimQ2} gives
\[\sum_{\vert Q \vert > N} \vert \tilde \varphi_{l,Q}(z^\beta) \tilde z(z)^Q \vert \leq 2.4^{\frac{1}{k}} M \tilde \rho^{\vert \beta \vert} C_n  \vert U \vert^{\frac{1}{k}} \sum_{m \geq N}m^{n-1}r_0^m \leq C_2 \vert U \vert^{\frac{1}{k}}{r_0'}^N\]
for any $r_0 < r_0' < 1$ and for some constant $C_2 >0$ (since $m^{n-1} (\frac{r_0}{r_0'})^m$ is bounded) not depending on $Q$ or $l$. Finally,
\begin{equation}\label{estimphiexpo3}\vert\tilde\varphi_l(z) \vert \leq 2. 4^{\frac{1}{k}} M \tilde\rho^{\vert \beta \vert}C_n N^n \vert U\vert^{\frac{1}{k}}e^{-\frac{c_1 \vert U \vert}{(N+1)^\tau}}+C_2 \vert U\vert^{\frac1k}{r_0'}^{N}.\end{equation}

The first term of $\eqref{estimphiexpo3}$ increases with $N$ (since $e^{-\frac{c_1 \vert U \vert}{(N+1)^\tau}}$ increases with $N$), while the second term $C_2 \vert U \vert^{\frac{1}{k}} {r_0'}^N = C_2 \vert U \vert^{\frac{1}{k}} e^{-N \vert \ln r_0' \vert}$ decreases with $N$. We choose $N$ so that the two exponents are of the same order, that is, up to the constants,
\[ \frac{\vert U \vert}{N^\tau} = N \Leftrightarrow N^{\tau +1} = \vert U \vert \Leftrightarrow N = \vert U \vert^{\frac{1}{\tau+1}}.\]
Hence, we choose $N= \vert U \vert^{\frac{1}{\tau+1}}$ (we may assume that $N$ is an integer). Then $(N+1)^{\tau} \leq 2^\tau N^\tau$, so that $\frac{\vert U \vert}{(N+1)^\tau} \geq 2^{-\tau}N$, and both terms of \eqref{estimphiexpo3} are bounded by $CN^n \vert U \vert^{\frac{1}{k}} e^{-M'' N}$ for some $C,M'' >0$ not depending on $l,Q$. The polynomial factor $N^n \vert U \vert^{\frac{1}{k}}$ being absorbed by the exponential, we get, for some $C', M'>0$ not depending on $l,Q$,
\begin{equation}\label{estimtildephil}
    \vert \tilde \varphi_l(z) \vert \leq C' e^{-M' \vert U \vert^{\frac{1}{\tau +1}}},
\end{equation}
where $U = (z^\beta)^{-k}$.

\textbf{Back to $\varphi$. }
Let $u = z^\beta$. By \eqref{deftildez} and by definition of $\tilde \varphi_l$ in \eqref{defphitilde}, for all $l=1,...,n$,
\[ C_l(z) = \tilde z_l(C(z))e^{\theta_l(C(z)^\beta)} = (\tilde z_l (z) + \tilde \varphi_l(z))e^{\theta_l(C(z)^\beta)},\]
hence
\begin{equation}\label{termesphi}
\varphi_l(z) = C_l(z) - z_l = \tilde \varphi_l(z)e^{\theta_l(C(z)^\beta)}+ \tilde z_l(z) \big( e^{\theta_l(C(z)^\beta)} - e^{\theta_l(z^\beta)}\big).
\end{equation}
We will give a bound to the two terms of \eqref{termesphi} as in the part "Cauchy estimates". Recall that $\vert C(z)^\beta-u \vert \leq \frac{1}{2}\vert u \vert$, so that $\frac{1}{2}\vert u\vert \leq \vert C(z)^\beta \vert \leq \frac{3}{2}\vert u \vert$.

For the first term, the left-hand side of \eqref{estimthetaj} gives $\vert e^{\theta_l(C(z)^\beta)} \vert \leq C_l \vert C(z)^\beta \vert^{\nu_l} \leq C_l (\frac{3}{2})^{\nu_l} \vert u \vert^{\nu_l}$, which is bounded since $\vert u \vert < r$ (recall $u\in S_j \cap S_{j+1}$). Hence
\[\vert\tilde \varphi_l(z)e^{\theta_l(C(z)^\beta)} \vert \leq C_l (\frac{3}{2})^{\nu_l} r^{\nu_l} \vert \tilde \varphi_l(z)\vert.\]

For the second term, by \eqref{deftildez} and \eqref{estimthetaj}, $\vert \tilde z_l(z) \vert \leq C_l \rho \vert u \vert^{-\nu_l}$. As before, the function $e^{\theta_l}$ is holomorphic in $u$ on a petal containing $S_j \cap S_{j+1}$ with a positive angular margin thanks to the $-\epsilon$ in \eqref{sector} (see again Figure \ref{margin}). Recall estimates \eqref{cauchythetal}. Similarly to \eqref{cauchythetal} (notice that in our case, the exponential is $+ \theta_l$ and not $-\theta_l$, so that the constant and the power of $\vert u \vert$ change), we get for all $v\in [z^\beta, C(z)^\beta]$ (up to shrinking $r$ as before),
\[ \vert \frac{d}{dv} e^{\theta_l(v)} \vert\leq C_l'' \vert v \vert^{\nu_l-1},\]
and by the mean value inequality,
\[ \vert e^{\theta_l(C(z)^\beta)}-e^{\theta_l(u)}\vert \leq \sup_{v \in [u, C(z)^\beta]}\vert  \frac{d}{dv} e^{\theta_l(v)}\vert \vert C(z)^\beta - u\vert  \leq C_l'' \sup_{v \in [u, C(z)^\beta]} \vert v \vert^{\nu_l-1} \vert C(z)^\beta - u\vert.\]
As for $v \in [u, C(z)^\beta]$, we have $\frac{1}{2}\vert u \vert \leq \vert v \vert \leq \frac{3}{2}\vert u \vert$. Therefore, using the lower bound or the upper-bound of this inequality depending on the sign of $\nu_l-1$, we get $\vert v \vert^{\nu_l-1} \leq c_l' \vert u \vert^{\nu_l-1}$, with $c_l' := \max((\frac{1}{2})^{\nu_l-1}, (\frac{3}{2})^{\nu_l-1})$, and then
\[ \vert e^{\theta_l(C(z)^\beta)}-e^{\theta_l(u)}\vert  \leq C_l'' c_l' \vert u \vert^{\nu_l-1} \vert C(z)^\beta -u \vert.\]

It remains to bound $\vert C(z)^\beta - u\vert$. Recall that, for $u=z^\beta$, $U = u^{-k}$ and $U' = (C(z)^\beta)^{-k}$, and that $\tilde \varphi_0(z) = \V(U')-\V(U)$ by \eqref{defphitilde}. 
Since $C(z)^\beta - u$ is flat in $u$, we can write $\frac{C(z)^\beta}{u}= 1 + \tilde w(z)$ with $\tilde w(z)$ flat in $u=z^\beta$. Then \[U'-U=(C(z)^\beta)^{-k} - u^{-k} = u^{-k}\big( (\frac{C(z)^\beta}{u})^{-k}-1\big) = u^{-k}\big((1+ \tilde w(z))^{-k}-1 \big)\]
and then, since $\tilde w$ is flat, $\vert U'-U\vert  \leq \frac{1}{2}\vert U \vert$ for $\vert U \vert$ large enough, and any $W \in [U,U']$ satisfy $\vert W \vert \geq \frac{1}{2}\vert U \vert$. Since $\V(U) = U - E'(0)\ln U + \omega(U)$ with $\omega(U)= \OO(\frac{1}{U})$ (and then $\omega'(U) = \OO(\frac{1}{U^2})$), we have, for $W \in [U,U']$, $\V'(W) = 1 + \OO(\frac{1}{U})$. Hence
\[ \tilde \varphi_0(z) = \int_{[U,U']}\V'(W) dW = \int_{[U,U']} (1+ [\V'(W)-1]) dW = (U'-U) + \int_{[U,U']} (\V'(W)-1) dW, \]
where the last integral is bounded by $\frac{C}{\vert U \vert} \vert U'-U\vert$ for some $C >0$ (since $(\V'(W)-1 = \OO(\frac{1}{U})$). Hence
\[ \vert \tilde \varphi_0(z) \vert \geq \vert U'-U \vert - \frac{C}{\vert U \vert}\vert U'-U\vert = (1 - \frac{C}{\vert U \vert})\vert U'-U\vert \geq \frac{1}{2}\vert U'-U \vert \]
for $\vert U \vert \geq 2C$ large enough, that is to say 
\[ \vert U'-U\vert \leq 2 \vert \tilde \varphi_0(z)\vert .\]

Since $u= U^{-\frac{1}{k}}$ and $C(z)^\beta = {U'}^{-\frac{1}{k}}$, we have $C(z)^\beta - u = {U'}^{-\frac{1}{k}} - U^{-\frac{1}{k}} = f(U')-f(U)$ with $f(W) := W^{-\frac{1}{k}}$, and $\vert f'(W) \vert = \frac{1}{k} \vert W \vert^{-1-\frac{1}{k}} \leq \frac{1}{k}(\frac{\vert U \vert}{2})^{-1-\frac{1}{k}}$ on $[U,U']$. Then mean value inequality and $\vert U' -U \vert \leq 2 \vert \tilde \varphi_0(z)\vert$ give
\[ \vert C(z)^\beta - u \vert = \vert f(U')- f(U)\vert \leq \sup_{W\in [U,U']}\vert f'(W)\vert \vert U'-U\vert \leq C_3 \vert U \vert^{-1-\frac{1}{k}} \vert \tilde \varphi_0(z) \vert \]
with $C_3 := \frac{2^{2+\frac{1}{k}}}{k}$.
Combining with the previous estimates, since $\vert u \vert^{-1} = \vert U \vert^{\frac{1}{k}}$, and using \eqref{estimztilde} the second term of \eqref{termesphi} satisfy
\[\vert \tilde z_l(z) \big( e^{\theta_l(C(z)^\beta)} -  e^{\theta_l(u)}\big) \vert \leq C_l \rho \vert u \vert^{-\nu_l}. C_l'' c_l' \vert u \vert^{\nu_l-1}.C_3 \vert U \vert^{-1-\frac{1}{k}} \vert \tilde \varphi_0(z) \vert \leq C_l C_l'' c_l' C_3 \rho \vert U \vert^{-1} \vert \tilde \varphi_0(z)\vert.\]

Finally, by \eqref{estimtildephil} applied to $\tilde \varphi_l$ and $\tilde \varphi_0$, and since $\vert U \vert^{\frac{1}{\tau +1}} = \vert z^\beta \vert^{-\frac{k}{\tau +1}}$, we obtain, for all $l=1,...,n$ and all $z\in \mathcal{D}_j \cap \mathcal{D}_{j+1}$, up to shrinking $\rho$ and $r$,
\[ \vert \varphi_l(z) \vert \leq C_4 e^{-\frac{M_4}{\vert z^\beta\vert^{\frac{k}{\tau+1}}}}\]
for some constants $C_4,M_4>0$.
\end{proof}

\begin{remark}
    In the case $n=1$ (and then $m=1$), we have $\beta = \beta_1$, the $Q=q_1$ defining the series of $\tilde \varphi_l$ are the $q_1 < \beta_1$. The sum is then finite and we have the same result as \cite{Loday}, that is to say $\varphi$ is exponentially flat at order $k$. In particular, the formal normalizing transformation is summable in the sense of Ramis-Malgrange-Braaksma-Balser \cite{malgrange-cours,ramis-panorama,Brak1,schafke-mozo-sum} in the monomial $z^{\beta}$. We would have had the same conclusion in case $\tau=0$, but in the 1-resonant case with all eigenvalues of modulus 1, this cannot hold, that is there are small divisors. 
\end{remark}
\begin{proof}[Proof of Corollary \ref{cor-gevrey}]
Let $F$ be a germ as in Theorem \ref{main-sect}, formally conjugate to a polynomial normal form $G$ by a formal diffeomorphism $\hat\Phi$ tangent to identity at the origin. Let $\{\Phi_i\}_i$ be a collection of sectorial normalizations over sectorial domains $\{{\mathcal D}_i\}_i$ having the same asymptotic expansion $\hat \Phi $. In ${\mathcal D}_i\cap {\mathcal D}_{i+1}$, $\Phi_{i+1}\circ\Phi_i^{-1}$ leaves the normal form $G$ invariant. Hence, according to Theorem \ref{thm-exp}, $\Phi_{i+1}\circ\Phi_i^{-1}=\text{Id}+\phi_{i,i+1}$ with $\phi_{i,i+1}$ exponentially decreasing of order $\frac{k}{\tau+1}$ in $z^{\beta}$. Hence, $\Phi_i^{-1}= \Phi_{i+1}^{-1}(\text{Id}+\phi_{i,i+1})= \Phi_{i+1}^{-1}+\psi_{i,i+1}$ where $\psi_{i,i+1}$ is exponentially decreasing of order $\frac{k}{\tau+1}$ in $z^{\beta}$. According to Ramis-Sibuya theorem \cite{ramis-sibuya}, the formal diffeomorphism $\hat \Phi^{-1}$ is $\frac{\tau+1}{k}$-Gevrey in $z^{\beta}$. Working with $\Phi_i$ instead of $\Phi_i^{-1}$, we obtain the same result for $\hat \Phi$.
\end{proof}

		\bibliographystyle{alpha}
		\bibliography{biblio-CS2}

\begin{thebibliography}{KPRR25}

\bibitem[Aba15]{abate-parabolic}
Marco Abate.
\newblock Fatou flowers and parabolic curves.
\newblock In {\em Complex analysis and geometry}, volume 144 of {\em Springer
  Proc. Math. Stat.}, pages 1--39. Springer, Tokyo, 2015.

\bibitem[Arn88]{Arnold}
V.~I. Arnold.
\newblock {\em Geometrical methods in the theory of ordinary differential
  equations}, volume 250 of {\em Grundlehren der mathematischen Wissenschaften
  [Fundamental Principles of Mathematical Sciences]}.
\newblock Springer-Verlag, New York, second edition, 1988.
\newblock Translated from the Russian by Joseph Sz\"{u}cs [J\'{o}zsef M.
  Sz\H{u}cs].

\bibitem[Bir39]{birkhoff}
George~D. Birkhoff.
\newblock D\'{e}formations analytiques et fonctions auto-\'{e}quivalentes.
\newblock {\em Ann. Inst. H. Poincar\'{e}}, 9:51--122, 1939.

\bibitem[BM04]{bracci-molino}
Filippo Bracci and Laura Molino.
\newblock The dynamics near quasi-parabolic fixed points of holomorphic
  diffeomorphisms in {$\Bbb C^2$}.
\newblock {\em Amer. J. Math.}, 126(3):671--686, 2004.

\bibitem[Bra92]{Brak1}
Boele L.~J. Braaksma.
\newblock Multisummability of formal power series solutions of nonlinear
  meromorphic differential equations.
\newblock {\em Ann. Inst. Fourier (Grenoble)}, 42(3):517--540, 1992.

\bibitem[BRSn21]{bracci-raissy-stensones}
Filippo Bracci, Jasmin Raissy, and Berit Stens\o~nes.
\newblock Automorphisms of {$\Bbb C^k$} with an invariant non-recurrent
  attracting {F}atou component biholomorphic to {$\Bbb C\times(\Bbb
  C^\ast)^{k-1}$}.
\newblock {\em J. Eur. Math. Soc. (JEMS)}, 23(2):639--666, 2021.

\bibitem[Bru72]{Bruno}
A.D. Bruno.
\newblock {Analytical form of differential equations}.
\newblock {\em Trans. Mosc. Math. Soc}, 25,131-288(1971); 26,199-239(1972),
  1971-1972.

\bibitem[BRZ13]{bracci-raissy-zaitsev}
Filippo Bracci, Jasmin Raissy, and Dmitri Zaitsev.
\newblock Dynamics of multi-resonant biholomorphisms.
\newblock {\em Int. Math. Res. Not. IMRN}, (20):4772--4797, 2013.

\bibitem[BS07]{stolo-braaksma}
Boele Braaksma and Laurent Stolovitch.
\newblock Small divisors and large multipliers.
\newblock {\em Ann. Inst. Fourier (Grenoble)}, 57(2):603--628, 2007.

\bibitem[BZ13]{Bracci}
Filippo Bracci and Dmitri Zaitsev.
\newblock Dynamics of one-resonant biholomorphisms.
\newblock {\em J. Eur. Math. Soc. (JEMS)}, 15(1):179--200, 2013.

\bibitem[CS87]{camacho-sad-book}
C\'{e}sar Camacho and Paulo Sad.
\newblock {\em Pontos singulares de equa\c{c}\~{o}es diferenciais
  anal\'{\i}ticas}.
\newblock 16$^{\rm o}$ Col\'{o}quio Brasileiro de Matem\'{a}tica. [16th
  Brazilian Mathematics Colloquium]. Instituto de Matem\'{a}tica Pura e
  Aplicada (IMPA), Rio de Janeiro, 1987.

\bibitem[Eca]{EcalleIII}
J.~Ecalle.
\newblock {Sur les fonctions r\'{e}surgentes}.
\newblock I,II,III Publ. Math. d'Orsay.

\bibitem[Fat24]{fatou}
P.~Fatou.
\newblock Substitutions analytiques et \'{e}quations fonctionnelles \`a deux
  variables.
\newblock {\em Ann. Sci. \'{E}cole Norm. Sup. (3)}, 41:67--142, 1924.

\bibitem[Hak94]{hakim94}
Monique Hakim.
\newblock Attracting domains for semi-attractive transformations of {${\bf
  C}^p$}.
\newblock {\em Publ. Mat.}, 38(2):479--499, 1994.

\bibitem[IY08]{ilyashenko-yakovenko-book}
Y.~Ilyashenko and S.~Yakovenko.
\newblock {\em Lectures on analytic differential equations}, volume~86 of {\em
  Graduate Studies in Mathematics}.
\newblock American Mathematical Society, Providence, RI, 2008.

\bibitem[Jen08]{jenkins-parbolic}
Adrian Jenkins.
\newblock Further reductions of {P}oincar\'{e}-{D}ulac normal forms in {${\bf
  C}^{n+1}$}.
\newblock {\em Proc. Amer. Math. Soc.}, 136(5):1671--1680, 2008.

\bibitem[KH95]{Kat}
Anatole Katok and Boris Hasselblatt.
\newblock {\em Introduction to the modern theory of dynamical systems},
  volume~54 of {\em Encyclopedia of Mathematics and its Applications}.
\newblock Cambridge University Press, Cambridge, 1995.
\newblock With a supplementary chapter by Katok and Leonardo Mendoza.

\bibitem[Kim71]{kimura}
Tosihusa Kimura.
\newblock On the iteration of analytic functions.
\newblock {\em Funkcial. Ekvac.}, 14:197--238, 1971.

\bibitem[KPRR25]{KMRR25}
M.~Klime\\check{s}, M.~Pavao, G.~Radunovi\'c, and M.~Resman.
\newblock Reading analytic invariants of parabolic diffeomorphisms from their
  orbits.
\newblock {\em ANNALI SCUOLA NORMALE SUPERIORE - CLASSE DI SCIENZE},
  26:1017--1047, 2025.

\bibitem[KS22]{stolo-klimes}
M.~Klime\c{s} and L.~Stolovitch.
\newblock Reversible parabolic diffeomorphisms of $(\mathbb{C}^2,0)$ and
  exceptional hyperbolic {CR}-singularities, 2022.
\newblock ARXIV.2204.09449, p.1-110.

\bibitem[Lea97]{Leau}
L\'{e}opold Leau.
\newblock \'{E}tude sur les \'{e}quations fonctionnelles \`a une ou \`a
  plusieurs variables.
\newblock {\em Ann. Fac. Sci. Toulouse Sci. Math. Sci. Phys.}, 11(2):E1--E24,
  1897.

\bibitem[LR16]{Loday}
Mich\`ele Loday-Richaud.
\newblock {\em Divergent series, summability and resurgence. {II}}, volume 2154
  of {\em Lecture Notes in Mathematics}.
\newblock Springer, [Cham], 2016.
\newblock Simple and multiple summability, With prefaces by Jean-Pierre Ramis,
  \'{E}ric Delabaere, Claude Mitschi and David Sauzin.

\bibitem[Mal82]{Malgrange-bourbaki}
Bernard Malgrange.
\newblock Travaux d'\'{E}calle et de {M}artinet-{R}amis sur les syst\`emes
  dynamiques.
\newblock In {\em Bourbaki {S}eminar, {V}ol. 1981/1982}, volume 92-93 of {\em
  Ast\'{e}risque}, pages 59--73. Soc. Math. France, Paris, 1982.

\bibitem[Mal95a]{malgrange-somm}
B.~Malgrange.
\newblock Sommation des s\'eries divergentes.
\newblock {\em Exposition. Math.}, 13(2-3):163--222, 1995.

\bibitem[Mal95b]{malgrange-cours}
B.~Malgrange.
\newblock Sommation des s\'eries divergentes.
\newblock {\em Exposition. Math.}, 13(2-3):163--222, 1995.

\bibitem[MFS19]{schafke-mozo-sum}
Jorge Mozo~Fern\'{a}ndez and Reinhard Sch\"{a}fke.
\newblock Asymptotic expansions and summability with respect to an analytic
  germ.
\newblock {\em Publ. Mat.}, 63(1):3--79, 2019.

\bibitem[Mil06]{milnor}
John Milnor.
\newblock {\em Dynamics in one complex variable}, volume 160 of {\em Annals of
  Mathematics Studies}.
\newblock Princeton University Press, Princeton, NJ, third edition, 2006.

\bibitem[MR82]{Ram-Mart1}
J.~Martinet and J.-P. Ramis.
\newblock {Probl\`{e}mes de modules pour des \'{e}quations diff\'{e}rentielles
  non lin\'{e}aires du premier ordre}.
\newblock {\em Publ. Math. I.H.E.S}, 55,63-164, 1982.

\bibitem[MR83]{Ram-Mart2}
J.~Martinet and J.-P. Ramis.
\newblock {Classification analytique des \'{e}quations diff\'{e}rentielles non
  lin\'{e}aires r\'{e}sonantes du premier ordre}.
\newblock {\em Ann. Sci. E.N.S}, 4\`{e}me s\'{e}rie,16,571-621, 1983.

\bibitem[P\"86]{Pos86}
J\"{u}rgen P\"{o}schel.
\newblock On invariant manifolds of complex analytic mappings near fixed
  points.
\newblock {\em Exposition. Math.}, 4(2):97--109, 1986.

\bibitem[R\"02]{russmann-lin}
Helmut R\"{u}ssmann.
\newblock Stability of elliptic fixed points of analytic area-preserving
  mappings under the {B}runo condition.
\newblock {\em Ergodic Theory Dynam. Systems}, 22(5):1551--1573, 2002.

\bibitem[Ram79]{borel-ritt-ramis}
J.-P. Ramis.
\newblock \`a propos du th\'{e}or\`eme de {B}orel-{R}itt \`a plusieurs
  variables.
\newblock In {\em \'{E}quations diff\'{e}rentielles et syst\`emes de {P}faff
  dans le champ complexe ({S}em., {I}nst. {R}ech. {M}ath. {A}vanc\'{e}e,
  {S}trasbourg, 1975)}, volume 712 of {\em Lecture Notes in Math.}, pages
  289--292. Springer, Berlin, 1979.
\newblock Appendice \`a l'article: ``\'{E}tude de certains syst\`emes de Pfaff
  avec singularit\'{e}s'' [\'{E}quations diff\'{e}rentielles et syst\`emes de
  Pfaff dans le champ complexe, pp. 131--288, Lecture Notes in Math., 712,.

\bibitem[Ram80]{ramis-ksum}
J.-P. Ramis.
\newblock Les s\'eries {$k$}-sommables et leurs applications.
\newblock In {\em Complex analysis, microlocal calculus and relativistic
  quantum theory (Proc. Internat. Colloq., Centre Phys., Les Houches, 1979)},
  volume 126 of {\em Lecture Notes in Phys.}, pages 178--199. Springer, 1980.

\bibitem[Ram93]{ramis-panorama}
Jean-Pierre Ramis.
\newblock S\'eries divergentes et th\'eories asymptotiques.
\newblock {\em Bull. Soc. Math. France}, 121(Panoramas et Syntheses, suppl.),
  1993.

\bibitem[Ron10]{rong10}
Feng Rong.
\newblock Quasi-parabolic analytic transformations of {${\bf C}^n$}.
  {P}arabolic manifolds.
\newblock {\em Ark. Mat.}, 48(2):361--370, 2010.

\bibitem[Ron16]{rong16}
Feng Rong.
\newblock Attracting domains for quasi-parabolic analytic transformations of
  {$\Bbb{C}^2$}.
\newblock {\em Proc. Roy. Soc. Edinburgh Sect. A}, 146(3):665--670, 2016.

\bibitem[RS89]{ramis-sibuya}
J.-P. Ramis and Y.~Sibuya.
\newblock Hukuhara domains and fundamental existence and uniqueness theorems
  for asymptotic solutions of {G}evrey type.
\newblock {\em Asymptotic Anal.}, 2(1):39--94, 1989.

\bibitem[RS93]{ramis-stolo-cours}
J.-P. Ramis and L.~Stolovitch.
\newblock {Divergent series and holomorphic dynamical systems}, 1993.
\newblock Unpublished lecture notes from J.-P. Ramis lecture at the SMS {\it
  Bifurcations et orbites p\'{e}riodiques des champs de vecteurs}, Montr\'eal
  1992. 57p.

\bibitem[Sha20]{semihyperbolic}
P.~A. Shaikhullina.
\newblock A realization theorem in the problem of a strict analytical
  classification of typical germs of semihyperbolic mappings.
\newblock {\em Chelyabinski{\u{\i}} Fiz.-Mat. Zh.}, 5(1):105--113, 2020.

\bibitem[Sto94]{Sto94}
Laurent Stolovitch.
\newblock Sur un th\'{e}or\`eme de {D}ulac.
\newblock {\em Ann. Inst. Fourier (Grenoble)}, 44(5):1397--1433, 1994.

\bibitem[Sto96]{Stolo-classif}
L.~Stolovitch.
\newblock {Classification analytique de champs de vecteurs $1$-r\'{e}sonnants
  de $(\Bbb C^n,0)$}.
\newblock {\em Asymptotic Analysis}, 12:91--143, 1996.

\bibitem[Sto15]{Sto15}
L.~Stolovitch.
\newblock Family of intersecting totally real manifolds of {$(\Bbb C^n,0)$} and
  germs of holomorphic diffeomorphisms.
\newblock {\em Bull. Soc. math. France}, 143(1):247--263, 2015.

\bibitem[SV17]{voronin-semihyperbolic}
Polina~Alekseevna Sha{\u{\i}}khullina and Serge{\u{\i}}~Mikha{\u{\i}}lovich
  Voronin.
\newblock Functional invariant for typical germs of semihyperbolic mappings.
\newblock {\em Chelyabinski{\u{\i}} Fiz.-Mat. Zh.}, 2(4):447--455, 2017.

\bibitem[Tey25]{teyssier-real}
Lo\"{\i}c Teyssier.
\newblock Spherical normal forms for germs of parabolic line biholomorphisms.
\newblock {\em Ergodic Theory Dynam. Systems}, 45(10):3255--3304, 2025.

\bibitem[Ued99]{ueda-contraction}
T.~Ueda.
\newblock Normal forms of attracting maps.
\newblock {\em Math. J. Toyama Univ.}, 22:25--34, 1999.

\bibitem[Vor81]{voronin}
S.~M. Voronin.
\newblock Analytic classification of germs of conformal mappings {$({\bf
  C},\,0)\rightarrow ({\bf C},\,0)$}.
\newblock {\em Funktsional. Anal. i Prilozhen.}, 15(1):1--17, 96, 1981.

\end{thebibliography}

	\end{document}